\documentclass [12pt]{amsart}
\usepackage[utf8]{inputenc}
\pdfoutput=1

\usepackage{amsmath,amssymb,amsthm,amsfonts}
\usepackage{mathtools}%

\usepackage[main=english,french]{babel}%
\usepackage{comment}%
\usepackage[bookmarksdepth=4]{hyperref}%
\usepackage{bbm}%
\usepackage{mathrsfs}%

\usepackage{tikz,graphicx,color}
\usepackage{tikz-cd}%
\usepackage{tikz-3dplot}%
\usetikzlibrary{calc}%
\usetikzlibrary{arrows}%
\usetikzlibrary{shapes}%
\usetikzlibrary{patterns}%
\usetikzlibrary{positioning}%
\usetikzlibrary{arrows.meta}
\usetikzlibrary{decorations.markings}
\usepackage{epstopdf}%
\pdfpageattr{/Group <</S /Transparency /I true /CS /DeviceRGB>>}

\usepackage[arrow]{xy}%
\usepackage{diagbox}%
\usepackage{subfig}%
\usepackage{arcs}%
\usepackage{xcolor}%
\usepackage{tabu}
\usepackage{booktabs}%

\usepackage[margin=1in]{geometry}

\usepackage{soul}
\usepackage{accents}
\usepackage{enumerate}

\usepackage{letltxmacro}
\usepackage{thmtools,etoolbox}

\usepackage{nameref,hyperref}
\usepackage[capitalize]{cleveref}

\usepackage{tikz-cd}
\usepackage{stmaryrd}

\usepackage{xspace}

\usepackage{stmaryrd}

\usepackage{bm}

\newtheorem{theorem}{Theorem}[section]

\newtheorem{lemma}[theorem]{Lemma}
\newtheorem{proposition}[theorem]{Proposition}
\newtheorem{corollary}[theorem]{Corollary}

\theoremstyle{definition}
\newtheorem{remark}[theorem]{Remark}
\theoremstyle{definition}
\newtheorem{definition}[theorem]{Definition}

\theoremstyle{definition}

\theoremstyle{definition}
\newtheorem{example}[theorem]{Example}
\theoremstyle{definition}
\newtheorem{notation}[theorem]{Notation}
\theoremstyle{definition}
\newtheorem{algorithm}[theorem]{Algorithm}

\crefname{figure}{Figure}{Figures}

\addtotheorempostheadhook[theorem]{\crefalias{theoremlisti}{theorem}}
\addtotheorempostheadhook[lemma]{\crefalias{theoremlisti}{lemma}}
\addtotheorempostheadhook[proposition]{\crefalias{theoremlisti}{proposition}}
\addtotheorempostheadhook[corollary]{\crefalias{theoremlisti}{corollary}}

\def\Acal{\mathcal{A}}\def\Fcal{\mathcal{F}}\def\Ucal{\mathcal{U}}

\def\C{\mathbb{C}}
\def\R{\mathbb{R}}

\def\Z{\mathbb{Z}}

\def\<{{\langle}}
\def\>{{\rangle}}

\def\la{{\lambda}}

\def\RR{{\mathbb R}}
\def\CC{{\mathbb C}}

\def\GL{\operatorname{GL}}
\def\SL{\operatorname{SL}}

\def\Z{{\mathbb Z}}
\def\R{{\mathbb R}}

\def\GL{{\rm GL}}

\def\id{{\operatorname{id}}}

\newcounter{todobackgr}[section]

\newcounter{todofigure}[section]

\newcommand{\smat}[1]{\left[\begin{smallmatrix}
      #1
    \end{smallmatrix}\right]}

\def\t{{\mathbf{t}}}

\def\d#1{\dot{#1}}
\def\ds{\d{s}}
\def\dw{\d{w}}
\def\du{\d{u}}

\def\alphacheck{\alpha^\vee}

\makeatletter
\DeclareRobustCommand{\cev}[1]{%
  \mathpalette\do@cev{#1}%
}
\newcommand{\do@cev}[2]{%
  \fix@cev{#1}{+}%
  \reflectbox{$\m@th#1\vec{\reflectbox{$\fix@cev{#1}{-}\m@th#1#2\fix@cev{#1}{+}$}}$}%
  \fix@cev{#1}{-}%
}
\newcommand{\fix@cev}[2]{%
  \ifx#1\displaystyle
    \mkern#23mu
  \else
    \ifx#1\textstyle
      \mkern#23mu
    \else
      \ifx#1\scriptstyle
        \mkern#22mu
      \else
        \mkern#22mu
      \fi
    \fi
  \fi
}

\makeatother

\def\Rich_#1^#2{\accentset{\circ}{R}_{#1,#2}}
\def\Rtp_#1^#2{R_{#1,#2}^{>0}}

\def\Rtnn_#1^#2{R_{#1,#2}^{\geq0}}

\def\PR_#1^#2{\accentset{\circ}{\Pi}_{#1,#2}}%
\def\PRtp_#1^#2{\Pi_{#1,#2}^{>0}}%
\def\PRtnn_#1^#2{\Pi_{#1,#2}^{\geq0}}%
\def\PRcl_#1^#2{\Pi_{#1,#2}}%
\def\PRR_#1^#2{\accentset{\circ}{\Pi}_{#1,#2}^\R}%
\def\PRRcl_#1^#2{\Pi_{#1,#2}^\R}%

\def\bt{{\mathbf{t}}}
\def\bw{{\mathbf{w}}}
\def\bv{{\mathbf{v}}}
\def\bu{{\mathbf{u}}}
\def\ba{{\mathbf{a}}}

\def\Fcal{\mathcal{F}}

\def\pd#1{_{#1}}
\def\pu#1{^{(#1)}}

\def\BR{B(\R)}

\def\hjmap{\kappa}
\def\hjmp_#1{\hjmap_{#1}}

\def\Uom_#1{U^{\diamond,-}_{#1}}

\def\xrasim{\xrightarrow{\sim}}

\def\U{{\mathcal U}}

\def\Richaff_#1^#2{\accentset{\circ}{\mathcal{R}}_{#1}^{#2}}

\def\U{\Ucal}

\def\Povar_#1{\accentset{\circ}{\Pi}_{#1}}
\def\Povarcl_#1{\Pi_{#1}}
\def\RPovar_#1{\accentset{\circ}{\Pi}^\R_{#1}}
\def\RPovarcl_#1{\Pi^\R_{#1}}
\def\Povtp_#1^#2{\Pi_{#1,#2}^{>0}}
\def\Povtnn_#1{\Pi_{#1}^{\geq0}}

\def\Star_#1{\operatorname{Star}_{#1}}
\def\Startnn_#1{\operatorname{Star}^{\geq0}_{#1}}

\def\Link{\operatorname{Lk}}
\def\Lkx_#1{\Link_{#1}}
\def\Lkxx_#1^#2{\accentset{\circ}{\Link}_{#1}^{#2}}

\def\Lktxx_#1^#2{\Link^{>0}_{#1,#2}}
\def\Starxx_#1^#2{\operatorname{Star}_{#1,#2}}
\def\Startxx_#1^#2{\operatorname{Star}^{\geq0}_{#1,#2}}

\def\sctnn_#1{\sc^{\geq0}_{#1}}

\def\sctp_#1^#2{\sc^{>0}_{#1,#2}}

\def\Seps_#1{S_{#1}}

\def\Lktpe_#1^#2{\Link^{>0}_{#1,#2}}
\def\Lktnne_#1{\Link^{\geq0}_{#1}}

\def\Lktp_#1^#2{\Link^{>0}_{#1,#2}}
\def\Lktnn_#1{\Link^{\geq0}_{#1}}

\def\sc{Z}

\def\sco_#1^#2{\accentset{\circ}{\sc}_{#1,#2}}

\def\sccl_#1^#2{\sc_{#1}^{#2}}

\def\Y{\mathcal{Y}}
\def\Yo_#1{\accentset{\circ}{\Y}_{#1}}
\def\Ycl_#1{\Y_{#1}}

\def\Ytp_#1{\Y_{#1}^{>0}}

\def\strg(#1){\normg{#1}}

\def\normg#1{\|#1\|}

\def\Spec{\operatorname{Spec}}
\def\int{{\operatorname{init}}}

\def\DOM^#1_#2{\Delta^{#1 \omega_i}_{#2 \omega_i}}
\def\DOMr^#1_#2{\Delta^{#1 \omega_r}_{#2 \omega_r}}
\def\DOMir^#1_#2{\Delta^{#1 \omega_{i_r}}_{#2 \omega_{i_r}}}
\def\om{\omega}

\def\sv{s^{\mathbf{v}}}

\newcommand*{\smallcap}{{\mathbin{\scalebox{0.5}{\ensuremath{\cap}}}}}%

\def\RLtp_#1^#2{\cev R_{#1,#2}^{>0}}
\def\Rsf_#1^#2{\Rich_{#1}^{#2}(K)}
\def\LRsf_#1^#2{\LRich_{#1}^{#2}(K)}

\def\pre{{\,\operatorname{pre}}}
\def\tpre_#1^#2{\vec\twistop^\pre_{#1,#2}}
\def\Ltpre_#1^#2{\cev\twistop^\pre_{#1,#2}}
\def\tpreL_#1^#2{\Ltpre_{#1}^{#2}}
\def\twistop{\tau}
\def\twist_#1^#2{\vec\twistop_{#1,#2}}
\def\twistL_#1^#2{\cev\twistop_{#1,#2}}

\def\Sn{S_n}

\def\om{\omega}
\def\sv_#1{s^{\bv}_{#1}}
\def\capBWB(#1){(#1)^{\smallcap w}}
\def\capBWoB(#1){(#1)^{\smallcap w_0}}
\def\capBWBnopar(#1){#1^{\smallcap w}}

\def\Pio_#1^#2{\Pi^\circ_{#1,#2}}

\def\mis{\phi} 

\def\chmnrR_#1{\vec\Delta_{#1}}
\def\chmnrL_#1{\cev\Delta_{#1}}

\def\Yo{y_0}

\def\H{H}

\def\Dpm_#1{\Delta^{\pm}_{#1}}
\def\Dmp_#1{\Delta^{\mp}_{#1}}

\def\MS{\operatorname{MS}}
\def\TMSR_#1^#2{\vec \tau^{\,\MS}_{#1,#2}}
\def\TMSL_#1^#2{\cev \tau^{\,\MS}_{#1,#2}}

\def\Vlax_#1{V(\la)_{#1}}

\def\A{{\mathcal A}}
\def\C{{\mathbb C}}
\def\Z{{\mathbb Z}}

\def\x{{\mathbf{x}}}

\def\Y{{\overline{Y}}}

\def\ord{{\mathrm{ord}}}

\def\dlog{{\rm dlog}}

\newcommand{\Rrel}[1]{\stackrel{#1}{\longrightarrow}}
\newcommand{\Rwrel}[1]{\stackrel{#1}{\Longrightarrow}}
\newcommand{\Lrel}[1]{ \stackrel{#1}{\longleftarrow}}
\newcommand{\Lwrel}[1]{\stackrel{#1}{\Longleftarrow}}
\title{Braid variety cluster structures, I: 3D plabic graphs}

\author{Pavel Galashin}
\address{Department of Mathematics, University of California, Los Angeles, 520 Portola Plaza,
Los Angeles, CA 90025, USA}
\email{\href{mailto:galashin@math.ucla.edu}{galashin@math.ucla.edu}}

\author{Thomas Lam}
\address{Department of Mathematics, University of Michigan, 2074 East Hall, 530 Church Street, Ann Arbor, MI 48109-1043, USA}
\email{\href{mailto:tfylam@umich.edu}{tfylam@umich.edu}}

\author{Melissa Sherman-Bennett}
\address{Massachusetts Institute of Technology, 77 Massachusetts Avenue, Cambridge, MA 02139}
\email{\href{mailto:msherben@mit.edu}{msherben@mit.edu}}

\author{David Speyer}
\address{Department of Mathematics, University of Michigan, 2844 East Hall, 530 Church Street, Ann Arbor, MI 48109-1043, USA}
\email{\href{mailto:speyer@umich.edu}{speyer@umich.edu}}

\thanks{P.G.\ was supported by an Alfred P. Sloan Research Fellowship and by the National Science Foundation under Grants No.~DMS-1954121 and No.~DMS-2046915. T.L.\ was supported by Grant No.~DMS-1953852 from the National Science Foundation. M.S.B. was supported by the National Science Foundation under Award
	No.~DMS-2103282. D.E.S was supported by Grants No.~DMS-1854225 and No.~DMS-1855135. Any opinions, findings, and conclusions or recommendations expressed in this material are
	those of the authors and do not necessarily reflect the views of the National Science
	Foundation.}

\date\today

\def\Quw{Q_{u,\br}}
\def\Guw{G_{u,\br}}
\def\Cycle{C}
\def\Surfaceuw{\widehat{S}_{u,\br}}
\def\Suw{\Surfaceuw}
\def\CV(#1){x_{#1}}

\def\BR{\accentset{\circ}{R}}

\def\wo{w_\circ}
\def\pmn{\pm I}
\def\pmnm{(\pmn)^m} %
\def\DRW{(\pmn)^m} %
\def\Demprod{\pi}

\def\br{\beta}  %
\def\bigBR{\accentset{\circ}{\mathcal{Y}}} %
\def\BigBR{\mathcal{Y}} %
\def\torus{T}
\def\ttorus{\tilde{T}}
\def\torus(#1){T_{#1}}
\def\ttorus(#1){\tilde{T}_{#1}}
\def\Cx{{\C^\times}}
\newcommand{\apd}[1]{^{\< #1\>}}
\newcommand{\crossing}[1]{\Delta_{#1}}

\def\grid_#1^#2{\Delta_{#1}^{#2}}
\def\grid_#1{\Delta_{#1}}
\def\bmref#1{{\normalfont(\hyperref[#1]{B\ref*{#1}})}\xspace}
\def\z{\bar{z}}
\newcommand{\seed}[1]{\Sigma_{#1}}
\def\seedubr{\seed{u, \br}}
\def\seedbr{\seed{\br}}
\def\xubr{\x_{u,\br}}
\def\xbr{\x_{\br}}
\def\ubr{{u, \br}}
\def\ubrp{{u, \br'}}
\def\Qubrp{\widetilde{Q}_{u, \br'}}

\def\pbr{\mathbf{p}_{\brp}}
\def\Aubr{\A_{\ubr}}
\def\Aubrp{\A_{\ubrp}}

\def\Disk{D}

\makeatletter
\newcommand*\bigcdot{\mathpalette\bigcdot@{.4}}
\newcommand*\bigcdot@[2]{\mathbin{\vcenter{\hbox{\scalebox{#2}{$\m@th#1\bullet$}}}}}
\makeatother

\def\demR{\vartriangleright}
\def\demL{\vartriangleleft}

\def\upu#1{u_{#1}}

\def\PD{\Gamma}
\def\bridge{b}
\def\fro{\operatorname{fro}}
\def\mut{\operatorname{mut}}
\def\Jfro{J_{\bu}^{\fro}}
\def\Jmut{J_{\bu}^{\mut}}

\def\Cycle{C}

\def\MP{\mathbf{M}} %
\def\Cham{{\mathcal{R}}}

\def\Hollow{[m]\setminus\Jo}
\def\Gubar{\overline{G}_{u,\br}}
\def\Surf{\mathbf{S}}
\def\Subr{\Surf_{u,\br}}
\def\Surfaceuw{\Surf_{u,\bw}}
\def\Suw{\Surfaceuw}

\def\Lmut{\Lambda^{\operatorname{closed}}_{u,\br}}
\def\Lubr{\Lambda_{u,\br}}
\def\Lubrst{\Lambda^\ast_{u,\br}}

\def\strand{S}
\def\bgamma{\bm{\gamma}}

\def\Qice{\widetilde{Q}}

\def\BQice{\widetilde{B}(\Qice)}
\def\Vice{\widetilde{V}}

\def\fro{\operatorname{fro}}
\def\mut{\operatorname{mut}}
\def\Vfro{V^{\fro}}
\def\Vmut{V^{\mut}}

\def\Jfro{J_{\bu}^{\fro}}
\def\Jmut{J_{\bu}^{\mut}}

\def\Quw{\Qice_{u,\bw}}
\def\Guw{G_{u,\bw}}
\def\Qubr{\Qice_{u,\br}}
\def\Gubr{G_{u,\br}}
\def\Gbr{G_{\br}}
\def\bice{\tilde{b}}

\def\Bice{\widetilde{B}}

\numberwithin{equation}{section}

\def\todoflr#1(#2){\textcolor{green!70!black}{\bf \{#1\}(#2)}\xspace}

\def\bmref#1{{\normalfont(\hyperlink{#1}{B\ref*{#1}})}\xspace}
\def\brp{\br'}
\def\Gubrp{G_{u,\brp}}
\def\Gbrp{G_{u,\brp}}
\def\Subrp{\Surf_{u,\brp}}

\def\Dot{D}

\def\cumv{{\normalfont(\hyperlink{fig:moves}{M1})}\xspace}
\def\sqmv{{\normalfont(\hyperlink{fig:moves}{M2})}\xspace}

\def\Qubrmut{Q_{u,\br}}

\def\Qubrpmut{Q_{u,\brp}}
\def\Qubrppmut{Q_{u,\brpp}}

\def\brpp{\br''}

\def\Gubrpp{G_{u,\brpp}}

\def\sinkrec{sink-recurrent\xspace}
\def\NtQ{N^{\operatorname{in}}_s(Q)}

\def\Skn{S_n^{(k)}}

\def\Guwar{\overline{G}_{u,\bw}}
\def\Seif{S}

\def\qq{q^{\frac12}}
\def\qqi{q^{-\frac12}}

\def\HOMP{P}
\def\top{{\operatorname{top}}}
\def\Ptop_#1(#2){P^\top(#1;#2)}

\def\conn{\operatorname{c}}

\def\gswap{\sigma}
\def\vd#1{v\apd{d}\pd{#1}}

\def\Chminsetr{\nabla^+}
\def\Chminsetb{\nabla^-}
\def\Chminsetp{\nabla_p}

\def\redop{\operatorname{red}}
\def\blueop{\operatorname{blue}}
\def\pired{\pi_{\redop}}
\def\piblue{\pi_{\blueop}}

\def\Luw{\Lambda_{u,\bw}}
\def\Luwst{\Lambda^\ast_{u,\bw}}

\def\igrw[#1]#2{\includegraphics[width=#1\textwidth]{#2}}

\def\figref#1(#2){Figure~\hyperref[#1]{\ref*{#1}(#2)}}

\def\bmref#1{{\normalfont(\hyperref[#1]{B\ref*{#1}})}\xspace}
\def\cumv{{\normalfont(\hyperref[fig:moves]{M1})}\xspace}
\def\sqmv{{\normalfont(\hyperref[fig:moves]{M2})}\xspace}

\def\Surfof{\Surf}

\def\Dot{\mathfrak{d}} %
\def\Dots{\mathfrak{D}}

\def\Qbr{\Qice_{\br}}
\def\xbr{\x_\br}
\def\Qbrp{\Qice_{\br'}}

\def\Jo{J_{u,\br}}
\def\Jop{J_{u,\br'}}
\def\Jow{J_{u,\bw}}
\def\Poincare{Poincar\'e\xspace}

\def\Jfro{\Jo^{\fro}}
\def\Jmut{\Jo^{\mut}}

\def\Cham{\mathbf{Ch}}

\def\lemmaref#1(#2){Lemma~\hyperref[#1]{\ref*{#1}(#2)}}

\def\Lubrp{\Lambda_{u,\brp}}

\def\Chmin(#1){\crossing{#1}}

\def\ds{\dot{s}}

\def\Xbul{X_{\bullet}}
\def\Ybul{Y_{\bullet}}

\def\cham{chamber\xspace}

\def\Gvw{G^{w,v}}
\def\Gvwtp{G_{>0}^{w,v}}
\def\Lvw{L^{w,v}}
\def\xivw{\xi^{w,v}}
\def\Qtp{\widetilde{Q}'}
\def\Qt{\widetilde{Q}}

\def\ba{\bm{a}}
\def\Cx{\Cycle_{\ba}}
\def\vu{u} %
\def\OmegaOp{\Omega^{\operatorname{op}}}
\def\omop{\omega^{\operatorname{op}}}
\def\upu#1{u_{#1}} %

\begin{document}

\makeatletter
\@namedef{subjclassname@2020}{%
  \textup{2020} Mathematics Subject Classification}
\makeatother



\begin{abstract}
We introduce $3$-dimensional generalizations of Postnikov's plabic graphs and use them to establish cluster structures for type $A$ braid varieties.
Our results include known cluster structures on open positroid varieties and double Bruhat cells, and establish new cluster structures for type $A$ open Richardson varieties.
\end{abstract}
\maketitle
\setcounter{tocdepth}{1}
\tableofcontents

\section{Introduction}

\emph{Braid varieties} $\BR_{\ubr}$ are smooth, affine, complex algebraic varieties associated to a permutation $u$ and a braid word $\beta$, that is, a word representing an element of the positive braid monoid.  The purpose of this work is to construct a cluster algebra structure \cite{FZ} on the coordinate ring of a braid variety.

\begin{theorem}\label{thm:main_braid}
The coordinate ring $\C[\BR_\ubr]$ of a braid variety is a cluster algebra.
\end{theorem}

The braid varieties we consider were studied, rather recently, in~\cite{Mellit_cell,CGGS}, generalizing varieties considered previously in~\cite{Deo,MR,WY}. \Cref{thm:main_braid} resolves conjectures of \cite{Lec,CGGS2}.

Certain special cases of braid varieties, namely open Richardson varieties and open positroid varieties, have played a key role in the combinatorial geometry of flag varieties and Grassmannians.  \emph{Open Richardson varieties} $\Rich_u^w$ are subvarieties of the variety $\SL_n/B_+$ of complete flags in $\C^n$, arising in the study of total positivity and Poisson geometry~\cite{Deo, Lus98, MR, Pos, BGY, KLS}.  %
An explicit cluster structure for open Richardson varieties was conjectured by Leclerc~\cite{Lec}. M\'enard~\cite{Menard} gave an alternative conjectural cluster structure, which was proved by Cao and Keller~\cite{CaoKel} to be an upper cluster structure; M\'enard's and Leclerc's cluster structures are expected to coincide.  Another upper cluster structure was constructed by Ingermanson~\cite{Ing}.  The cluster structure of \Cref{thm:main_braid}, in the case of open Richardson varieties, agrees with that of Ingermanson.  It is related to the cluster structure of Leclerc~\cite{Lec} by the twist automorphism \cite{GL_twist,SSB}.

We also prove that the cluster varieties in \Cref{thm:main_braid} are locally acyclic \cite{Mul}; in the case of open Richardson varieties we use this to establish a variant of a conjecture of Lam and Speyer~\cite{LS}.  In particular, the cohomology of braid varieties satisfies the curious Lefschetz phenomenon (\cref{thm:CL}).

\Cref{thm:main_braid} generalizes the (type A) results of \cite{FZ_double,GY} on double Bruhat cells and \cite{ShWe} on double Bott--Samelson cells, and also the results of Ingermanson~\cite{Ing} and Cao and Keller~\cite{CaoKel}, who found upper cluster structures on open Richardson varieties (see also \cite{Menard}).  Furthermore, \cref{thm:main_braid} generalizes the main result of~\cite{GL_cluster} (see also~\cite{Scott,MSLA,SSBW}), where the same statement was proved for \emph{open positroid varieties}~\cite{KLS}, which are special cases of open Richardson varieties. Positroid varieties are parametrized by \emph{plabic graphs}~\cite{Pos}, whose planar dual quivers describe the cluster algebra structure on the associated open positroid varieties.

Ever since the completion of~\cite{GL_cluster}, it has been our hope that constructing a cluster structure for open Richardson varieties would lead to a meaningful generalization of Postnikov's plabic graphs; indeed, discovering such a generalization turned out to be a crucial step in our proof of \cref{thm:main_braid}.  We associate a \emph{3D plabic graph} to each pair $(u,\br)$ consisting of a permutation $u$ and a (double) braid word $\br$, and use the combinatorics of this graph to construct our cluster structure. The reader is invited to look forward at the examples in \crefrange{fig:bic}{fig:3D-ex}. Code for computing 3D plabic graphs and the associated seeds is available at \cite{GalTut}.

An important geometric ingredient in our approach is the study of the Deodhar geometry of braid varieties, originally used by Deodhar \cite{Deo} in the flag variety setting.  We define an open Deodhar torus $T_\ubr \subset \BR_\ubr$, and our cluster variables are interpreted as characters of $T_\ubr$ that have certain orders of vanishing along the Deodhar hypersurfaces in the complement of the Deodhar torus.  We expect this geometric approach to have applications to other settings where cluster structures are expected to make an appearance.

Our work has a number of applications.  Our cluster structure implies, via \cite{LS}, a curious Lefschetz phenomenon for the cohomology of braid varieties.  Our approach is closely related to the combinatorics of \emph{braid Richardson links} that are associated to a braid variety; in particular, we relate certain quiver point counts to the HOMFLY polynomial of these links.

We learned at the final stages of completing this manuscript that a cluster structure for braid varieties (of arbitrary Lie type) was independently announced in a recent preprint~\cite{CGGLSS}.  It would be interesting to investigate the relation between our 3D plabic graphs and the approach of~\cite{CGGLSS}.  %

\subsection*{Our construction in other Lie types}
The first three authors have continued this line of research to encompass braid varieties of arbitrary Lie type~\cite{GLSBS2}, again using the Deodhar decomposition as the main geometric tool. 
The cluster structure that we construct in type $A$ is much more explicit and combinatorial than the ones that we construct in other Lie types. 
In particular, our type $A$ structure has clear connections to 
the plabic graphs of~\cite{Pos}. In all Lie types, there is a family of ``chamber minors" (generalizing those appearing in~\cite{FZ_double,MR}) %
which are monomials in the cluster variables. In type $A$, the matrix transforming cluster variables to chamber minors is a $\{0,1\}$ matrix; see~\cref{ssec:seeds} and~\cref{prop:combGeoMatch}. 
However, in other Lie types, this matrix can have entries larger than $1$, and the combinatorics of these matrix entries is complicated; see~\cref{rem:orderone} and \cite[Section 7]{GLSBS2}.

On the other hand, there are arguments which are simpler and clearer when presented in all Lie types. 
We have chosen to place those arguments in~\cite{GLSBS2}, and removed longer computations which establish them only in type $A$ from this paper.

\subsection*{Overview}
 In \cref{sec:Rich}, we give a synopsis of our main results in the setting of open Richardson varieties.  In the rest of the paper we work in the setting of braid varieties.  We define 3D plabic graphs and the associated quivers $\Qubr$ in \cref{sec:double_braid_quivers}. Next, we develop the combinatorics of 3D plabic graphs and show that the quivers $\Qubr$ are invariant under braid moves on the word $\br$, naturally extending \emph{square moves} from Postnikov's plabic graphs to 3D plabic graphs; see \cref{sec:invar-under-moves}. We discuss cluster algebras associated to 3D plabic graphs in \cref{sec:cluster+3D} and show that they are \emph{locally acyclic} in the sense of~\cite{Mul}. In \cref{sec:dblBraidVar,sec:seeds}, we study the Deodhar geometry of $\BR_{\ubr}$ and construct a seed in $\CC(\BR_{\ubr})$ for each 3D plabic graph. We then show that the seeds are related by mutation in \cref{sec:movesOnSeed}. Finally, in \cref{sec:proofmain}, we prove \cref{thm:main} by induction on the length of $\br$. We conclude with some applications of our approach in \cref{sec:applications}, and explain how our results and definitions specialize to those in~\cite{Pos,FZ_double,Ing}.

\subsection*{Acknowledgments}
T.L. and D.E.S. thank our students Ray Karpman and Gracie Ingermanson for helping us understand the relationship between Deodhar's positive subexpressions and Postnikov's combinatorics and for the other ideas discussed in \cref{sec:gen}. M.S.B. thanks Daping Weng for illuminating conversations on \cite{ShWe}. P.G. thanks Terrence George for discussions regarding~\cite{GoKe}. We thank Lauren Williams for her comments on the first version of this manuscript. 
We also appreciate many conversations with Allen Knutson about Richardson and Bott--Samelson varieties, and Deodhar tori.
 We thank Roger Casals, Eugene Gorsky, and Anton Mellit for conversations related to this project. Finally, we thank the authors of~\cite{CGGLSS} for sharing their exciting results with us.

\section{Open Richardson varieties}\label{sec:Rich}
In this section, we give a more detailed explanation of \cref{thm:main_braid} in the case of open Richardson varieties.

\subsection{Open Richardson varieties}
Let $G=\SL_n$, and let $B_+$, $B_-$ be the opposite Borel subgroups of upper and lower triangular matrices, respectively.  For two permutations $u,w\in\Sn$ such that $u\leq w$ in the Bruhat order, the \emph{open Richardson variety} $\Rich_u^w$ is defined as
\begin{equation*}%
  \Rich_u^w:=(B_-uB_+\cap B_+wB_+)/B_+.
\end{equation*}

\begin{figure}
  \includegraphics[width=0.8\textwidth]{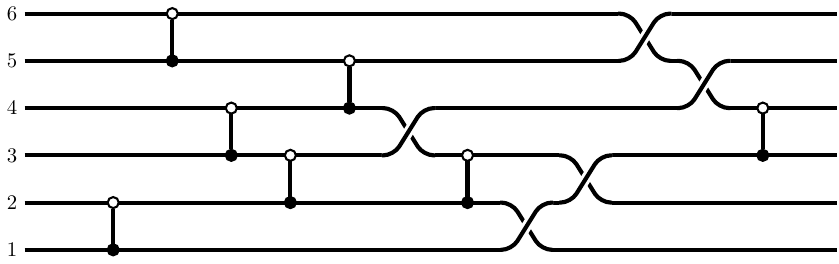}
  \caption{\label{fig:bic}A 3D plabic graph $\Guw$.}
\end{figure}

To each pair $u\leq w$ and to each reduced word $\bw$ for $w$ we associate an \emph{ice quiver} $\Quw$ (\cref{sec:intro:quiver}). Let $\Acal(\Quw)$ be the associated cluster algebra; see \cref{sec:cluster} for background. %

\begin{theorem}\label{thm:main_Rich}
For all $u\leq w$ in $\Sn$, we have an isomorphism
\begin{equation*}%
  \C[\Rich_u^w]\cong \Acal(\Quw).
\end{equation*}
Moreover, the cluster algebra $\Acal(\Quw)$ is locally acyclic and really full rank. %
\end{theorem}

The cluster algebra terminology in Theorem~\ref{thm:main_Rich} will be introduced in \cref{sec:cluster}.  We now describe the 3D plabic graph $\Guw$, the quiver $\Quw$, and the associated cluster algebra $\Acal(\Quw)$.

\begin {figure}
  \includegraphics[width=0.8\textwidth]{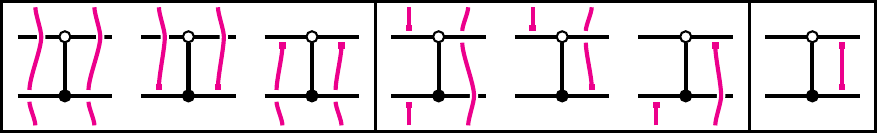}
  \caption{\label{fig:intro:propag}Propagation rules (right to left) for the relative cycles in $\Guw$.}
\bigskip
  \includegraphics[width=0.8\textwidth]{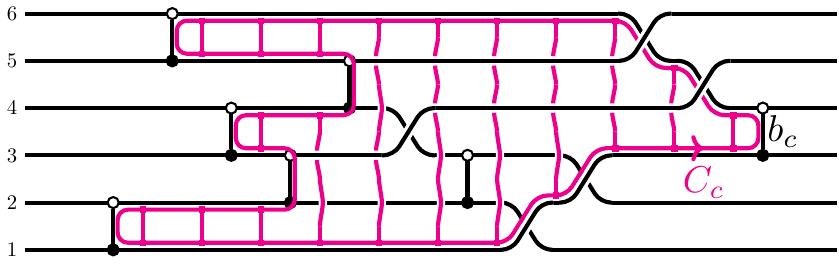}
  \caption{\label{fig:face} Applying propagation rules to find one relative cycle $\Cycle_{c}$ of $\Guw$.}
\end{figure}

\subsection{3D plabic graphs}\label{sec:intro:3D_plabic}
Let $\bw=(i_1,i_2,\dots,i_m)$ be a reduced word for $w$. Consider the unique rightmost subexpression $\bu$ for $u$ inside $\bw$, and let $\Jow\subset[m]:=\{1,2,\dots,m\}$ be the set of indices \emph{not} used in $\bu$. The \emph{3D plabic graph $\Guw$} is obtained from the wiring diagram for $w$ by replacing all crossings in $[m]\setminus \Jow$ by overcrossings and replacing each crossing $c\in\Jow$ by a \emph{black-white bridge edge} $\bridge_c$; see \cref{fig:bic}. We place a \emph{marked point} on each of the $n$ leftmost boundary vertices of $\Guw$, and denote by $\MP$ the set of these marked points.

The number of bridges in $\Guw$ is $|\Jow|=\ell(w)-\ell(u)$, which is the dimension of $\Rich_u^w$. To each index $c\in \Jow$ we will associate an (oriented) \emph{relative cycle} $\Cycle_c$ in $\Guw$, which by definition is either a cycle in $\Guw$ or a union of oriented paths in $\Guw$ with endpoints in $\MP$.

 Each relative cycle $\Cycle_c$ will naturally bound a disk $\Disk_c$.\footnote{More precisely, when $\Cycle_c$ is a cycle, $\partial\Disk_c=\Cycle_c$, and when $\Cycle_c$ is a union of paths with marked endpoints,  $\partial\Disk_c$ is the union of $\Cycle_c$ together with several straight line segments connecting pairs of marked points.}  For instance, in \cref{fig:face}, the vertical sections of $\Disk_c$ are shown in wavy pink lines. We indicate the relative position of $\Disk_c$ in $\R^3$ with respect to the edges of $\Guw$ by over/under-crossings. 
 We will compute $\Disk_c$, and therefore its boundary $\Cycle_c$, starting from the bridge $\bridge_c$ and proceeding to the left
 using the  \emph{propagation rules} in \cref{fig:intro:propag}. We choose the counterclockwise orientation of $\Cycle_c$, so that as one traverses $\Cycle_c$, the disk $\Disk_c$ is to the left. See \cref{sec:rel_cycles} for a  description of relative cycles in the case of double braid varieties.

\subsection{The quiver}\label{sec:intro:quiver}
A \emph{quiver} $Q$ is a directed graph without directed cycles of length $1$ and $2$. An \emph{ice quiver} $\Qice$ is a quiver whose vertex set $\Vice=V(\Qice)$ is partitioned into \emph{frozen} and \emph{mutable} vertices: $\Vice=\Vfro\sqcup \Vmut$. The arrows between pairs of frozen vertices are automatically omitted.

The procedure in \cref{sec:intro:3D_plabic} yields a bicolored graph $\Guw$ decorated with a family $(\Cycle_c)_{c\in\Jow}$ of relative cycles. To this data, we associate an ice quiver $\Quw$. Our construction will rely on the results of~\cite{FoGo_cluster,GoKe}. The vertex set $V(\Quw):=\Jow$ is in bijection with the set of relative cycles. If a relative cycle $\Cycle_c$ is actually a cycle in $\Guw$ then $c$ is a mutable vertex of $\Quw$; otherwise, if $\Cycle_c$ is a union of paths with endpoints in $\MP$, $c$ is a frozen vertex of $\Quw$.

To compute the arrows of $\Quw$, we consider $\Guw$ as a \emph{ribbon graph}, with counterclockwise half-edge orientations around white vertices and clockwise half-edge orientations around black vertices. Let $\Suw$ be the surface with boundary obtained by replacing every edge of $\Guw$ by a thin ribbon and gluing the ribbons together according to the local orientations at the vertices of $\Guw$. See \figref{fig:2D-ex}(d) for an example of $\Suw$. The $n$ marked points of $\Guw$ give rise to $n$ marked points on $\partial\Suw$, the set of which is also denoted by $\MP$. 

We view each relative cycle $\Cycle_c$ as an element of the relative homology $\Luw:=H_1(\Suw,\MP)$. It turns out that each mutable relative cycle can be also viewed as an element of the dual lattice $\Luwst$; see \cref{sec:surfaces}. The (signed) number of arrows between two vertices $c,d\in\Jow$ in $\Quw$, where $d$ is mutable, is defined to be the (signed) intersection number $\<\Cycle_c,\Cycle_d\>$ of the relative cycles $\Cycle_c$ and $\Cycle_d$. These intersection numbers can be computed explicitly using simple pictorial rules; see \cref{alg:int_form}. Alternatively, one can compute the quiver $\Quw$ by summing up half-arrow contributions; see \cref{sec:half-arrow-descr}. One can check that in the case of our running example (\cref{fig:bic}), both descriptions yield the quiver shown in \cref{fig:intro_quiver}. See also \cref{ex:3D}.

\begin{figure}
\igrw[0.4]{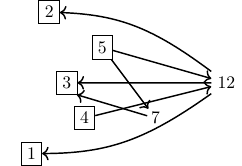}
  \caption{\label{fig:intro_quiver} The quiver $\Quw$ corresponding to the graph $\Guw$ from \cref{fig:bic}. For example, vertex $12$ of $\Quw$ corresponds to the relative cycle $\Cycle_{12}$ is shown in \cref{fig:face}. The frozen vertices are boxed.}
\end{figure}

\begin{figure}
\igrw[1.0]{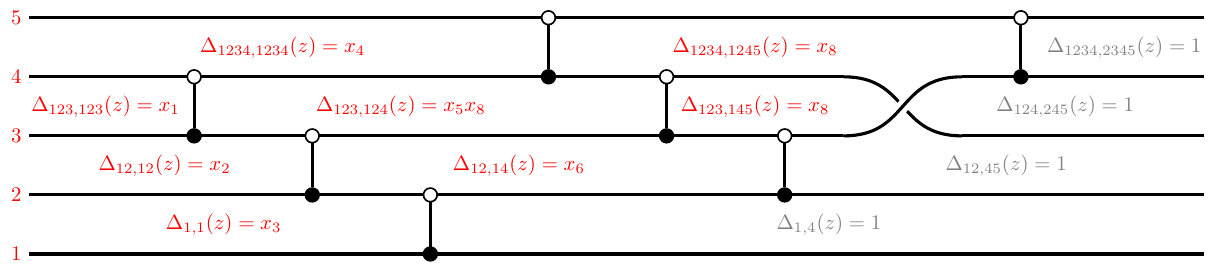}
  \caption{\label{fig:gracie_intro} Factorization of chamber minors into cluster variables.}
\end{figure}

\subsection{The seed}\label{sec:intro:seed}
To each $c\in\Jow$ we associate a \emph{cluster variable} $x_c\in\C[\Rich_u^w]$. Let $\Cham_c$ be the \emph{chamber} (i.e., a connected component of the complement of $\Guw$ in the plane) located immediately to the left of the bridge $\bridge_c$. To this data, one can associate a regular function on $\Rich_u^w$ called a \emph{chamber minor} $\Chmin(c)\in\C[\Rich_u^w]$. Variants of these functions appear in~\cite{MR,Lec,Ing}. 

We now define these chamber minors.
For $v \in \Sn$ and $0 \leq i \leq n$, we write $v[i]:=\{v(1),\dots,v(i)\}$. For a matrix $g \in \GL_n$, we write $\Delta_{A,B}(g)$ for the matrix minor with row set $A$ and column set $B$.
Given an element $gB_+\in\Rich_u^w$, one can find a matrix $z\in B_+$ such that $gB_+=zwB_+$ and $\Delta_{u[i],w[i]}(z)=1$ for all $i=1,2,\dots,n$.  The matrix $z$ is not unique, but is unique up to right multiplication by elements of $U_+ \cap w U_+ w^{-1}$ (where $U_+\subset B_+$ is the unipotent subgroup; see \cref{sec:doublebraid_Notation}), and all the minors of $z$ which we consider are invariant under this multiplication. The wiring diagram of $w$ is obtained from $\Guw$ by replacing every bridge with a crossing, and the wiring diagram of $u$ is obtained by erasing all bridges from $\Guw$. Given any chamber $\Cham$, let $A$ be the set of left endpoints of $u$-strands (i.e., the strands in the wiring diagram of $u$) passing below $\Cham$, and let $B$ be the set of left endpoints of $w$-strands passing below $\Cham$. The chamber minor for $\Cham$ is then given by $\Delta_{A,B}(z)$.

For $d\in \Jow$, we say that $\Cham_c$ is \emph{inside $\Cycle_d$} if it is contained inside the projection of the disk $\Disk_d$ to the plane. 
 Then the cluster variables $(x_c)_{c\in\Jow}$ are uniquely defined by the invertible monomial transformation
\begin{equation}\label{eq:Chmin}
  \Chmin(c)=\prod_{d\in \Jow:\ \text{$\Cham_c$ is inside $\Cycle_d$}} x_d.
\end{equation}
See \cref{fig:gracie_intro} for an example when $w=s_3s_2s_1s_4s_3s_2s_3s_4$ and $u=s_3$, and see \cref{sec:gracie_compare} for further details.

The cluster $\{x_c\}_{c \in \Jow}$ together with the quiver $\Qice_{u,\bw}$ is a seed $\Sigma_{u, \bw}$ in $\C(\Rich_u^{w})$. To prove \cref{thm:main_Rich}, we show that $\C[\Rich_u^{w}]= \A(\Sigma_{u, \bw})$.

\subsection{Deodhar geometry}
In \cite{Deo}, Deodhar constructed a stratification of $\Rich_u^{w}$ for each reduced word $\bw$ of $w$. The strata are of the form $\C^a \times (\C^*)^b$. The dense open stratum in Deodhar's stratification is called the \emph{Deodhar torus}, and denoted $T_{u,\bw}\subset \Rich_u^w$.  The Deodhar torus is the initial cluster torus in our cluster structure.  We show in \cref{prop:deodharIsTorus} that the chamber minors $\{\Chmin(c)\}_{c \in \Jow}$ are a basis of characters for $T_{u,\bw}$.  

We show that the cluster variables $\{x_d\}_{d \in \Jow}$ are certain distinguished characters on $T_{u,\bw}$. 
Namely, the complement $\Rich_u^w \setminus T_{u,\bw}$ is a union of irreducible components called \emph{(mutable) Deodhar hypersurfaces} $V_d$.  We deduce from~\eqref{eq:Chmin} that the mutable cluster variable $x_d$ is the unique character of $T_{u,\bw}$ which vanishes to order one along $V_d$ and has no zeroes along other Deodhar hypersurfaces.  Frozen Deodhar hypersurfaces and frozen cluster variables are constructed using a slightly different geometric approach.

\subsection{Comparison to known cluster structures}\label{sec:gen}
Our construction simultaneously includes several known cluster structures.

When $w$ is a $k$-Grassmannian permutation, the open Richardson variety $\Rich_u^w$ is an open positroid variety in the sense of~\cite{Pos,KLS}. 
In this case, Karpman~\cite{Karpman} showed that the graph $\Guw$ is one of Postnikov's reduced plabic graphs~\cite{Pos} with $n$ extra leaves attached.  
In particular, in this case $\Guw$ is planar and each cycle $C_c$ bounds a single face of $\Guw$, which is not true for general $w\in\Sn$.
In this case, the first two authors~\cite{GL_cluster} showed that the open positroid variety is a cluster variety, with quiver coming from the plabic graph; this gives the quiver $\Quw$.

Many of our ideas appeared in the unpublished Ph.D. dissertation of Ingermanson~\cite{Ing}, which constructs an upper cluster structure on $\Rich_u^w$.  The fourth author was Ingermanson's advisor and is grateful for the ideas he learned from her.  We summarize the constructions in our paper that appear in \cite{Ing}.  Ingermanson writes down the same monomial transformation as \eqref{eq:Chmin}, but defines it via a much more involved recursion.
Ingermanson constructs a bridge diagram (in the case that $\bw$ is a particular reduced word called the ``unipeak word") which is isomorphic to our graph $\Guw$ but is defined as an abstract graph rather than embedded in $\RR^3$.  Ingermanson's quiver is identical to our $\Quw$, but it is described in a very different way, along the lines of \cref{sec:half-arrow-descr}.
\subsection{Braid Richardson varieties}\label{sec:braid}
Open Richardson varieties generalize to \emph{braid Richardson varieties}.  We say that two flags $B_1,B_2\in G/B_+$ are in \emph{relative position $w\in \Sn$} if there exists a matrix $g\in G$ such that $(gB_1,gB_2)=(B_+,wB_+)$. In this case, we write $B_1\Rrel{w} B_2$.  Consider a (not necessarily reduced) word $\br=(i_1,i_2,\dots,i_m)\in[n-1]^m$ and let $u\in \Sn$.  Following the nomenclature of \cite{GLTW}, define the \emph{braid Richardson variety}
\begin{equation}\label{eq:braidRich}
  \BR_{u,\br} = \left\{(B_1,B_2,\dots,B_m)\in (G/B_+)^m\mid B_+\xrightarrow{s_{i_1}} B_1  \xrightarrow{s_{i_2}} \cdots \xrightarrow{s_{i_m}} B_m \xleftarrow{\wo u} B_-\right\}.
\end{equation}
This variety is nonempty whenever $u\leq \br$ in the sense of \cref{dfn:u<=br}.  In the case $u = \wo$, the definition \eqref{eq:braidRich} is the same as the definition of the braid variety $X(\beta)$ considered in \cite{CGGLSS}, and is isomorphic to the braid variety $X(\beta,\wo)$ in \cite{CGGS}.  We shall show that braid Richardson varieties are isomorphic to the braid varieties we study in \cref{sec:dblBraidVar}.

When $\br$ is a reduced word for some $w\in\Sn$ then using a variant of \lemmaref{lem:positionFacts}(2--3), we see that $\BR_{u,\br}$ is isomorphic to the space 
$\left\{B \in G/B_+ \mid B_+\xrightarrow{w} B \xleftarrow{\wo u} B_-\right\}$
 which can be identified with the Richardson variety $\Rich_u^w \subset G/B_+$.  Thus, braid Richardson varieties generalize open Richardson varieties. \Cref{thm:main_Rich} extends to the setting of braid Richardson varieties.

In \cref{sec:double_braid_words}, we describe our construction in the most general setting of double braid varieties, and define a 3D plabic graph $\Gubr$ and a quiver $\Qubr$ for a pair $(u,\br)$ consisting of a permutation $u$ and a \emph{double braid word} $\br$ (\cref{sec:double_braid_words}).  If $\br$ is a double braid word  and $u=\id$, then $\Gubr$ is again planar.
In this case, if we impose that the double word $\beta$ is reduced, we recover the classical cluster structure on type A double Bruhat cells~\cite{FZ_double,BFZ,GY}.  If we allow $\beta$ to be non-reduced, then we recover the type A results of~\cite{ShWe} on double Bott--Samelson varieties.
We note that the fact that our graphs are non-planar indicates that our construction gives a non-trivial $3$-dimensional extension of the combinatorics of~\cite{FZ_double, Pos, ShWe}.

\section{Double braid quivers}\label{sec:double_braid_quivers}

The goal of this section is to define 3D plabic graphs, conjugate surfaces, relative cycles, and the associated quivers in the extended generality of \emph{double braid words}.

\subsection{Double braid words}\label{sec:double_braid_words}
Let $I:=[n-1]$. A \emph{double braid word} is a word $\br=(i_1, i_2, \dots, i_m)\in\pmnm$ in the alphabet
\begin{equation*}%
  \pmn:=\{-1,-2,\dots,-(n-1)\}\sqcup\{1,2,\dots,n-1\}.
\end{equation*}
We denote the set of double braid words by $\DRW$. We will usually abbreviate $\br=(i_1, i_2, \dots, i_m)$ as  $\br=i_1 i_2 \dots i_m$. For $\br\in \DRW$, we write $\ell(\br):=m$. The goal of this section is to associate a quiver $\Qubr$ to a pair $(u,\br)$, where $u\in \Sn$ and $\br\in\DRW$.

The word $\br \in\DRW$ should be considered as a \emph{shuffle} of two positive braid words in the commuting alphabets $I, -I$. We emphasize that the letter $-i$ does \emph{not} correspond to $\sigma_i^{-1}$, the inverse of the braid group generator $\sigma_i$. We call elements of $I$ \emph{red} and elements of $-I$ \emph{blue}. 
For $i\in\pmn$, let
\begin{equation*}%
 s_i^+:=
  \begin{cases}
    s_i, &\text{if }i>0,\\
    \id, &\text{if }i<0,
  \end{cases}\quad   s_i^-:=
  \begin{cases}
    \id, &\text{if }i>0, \\
    s_{-i}, &\text{if }i<0.\\
  \end{cases}
\end{equation*}
We use the convention that the positive indices act on the right while the negative indices act on the left; for a permutation $u\in\Sn$ and $i\in\pmn$, we write this action as $u\mapsto s_i^-us_i^+$. 

Let $\ast$ denote the standard Demazure product on $\Sn$: for $v\in\Sn$ and $i\in I$, let $v\ast s_i$ be equal to $vs_i$ if $v<vs_i$ and to $v$ otherwise. This operation is associative. For a double braid word $\br$, its \emph{Demazure product} is defined by
\begin{equation*}%
  \Demprod(\br):= s_{i_m}^- *s_{i_{m-1}}^-*\cdots* s_{i_1}^- * s_{i_1}^+*s_{i_2}^+ * \cdots* s_{i_m}^+ \quad \in\Sn.
\end{equation*}
\begin{definition}\label{dfn:u<=br}
For $u\in\Sn$ and a double braid word $\br$, we write $u\leq \br$ if $u\leq \Demprod(\br)$ in the Bruhat order on $\Sn$.
\end{definition}

Let $\br=(i_1, i_2, \dots, i_m)\in\pmnm$ and $u\in\Sn$. A \emph{$u$-subexpression} of $\br$ is a sequence
$\bu=(\upu0,\upu1,\dots,\upu m) \in S_n^{m+1}$ such that $\upu0=\id$, $\upu m=u$, and such that for each $c\in[m]$, we have either $\upu c=\upu{c-1}$ or $\upu c=s_{i_c}^-\upu{c-1}s_{i_c}^+$. It is clear that $\br$ contains a $u$-subexpression if and only if $u\leq \br$. %

Suppose that $u\leq\br$. Out of all $u$-subexpressions of $\br$, there exists a unique ``rightmost'' one, called the \emph{$u$-positive distinguished subexpression ($u$-PDS)}. It can be computed explicitly using the following operation which we call \emph{Demazure quotient}: for $u\in\Sn$ and $i\in I$, set
\begin{equation*}%
  s_i \demR u= \begin{cases}
		s_iu,  & \text{ if }s_iu  < u,\\
		u, & \text{otherwise,}
              \end{cases}
              \quad\text{and}\quad
  u \demL s_i= \begin{cases}
		u s_i, & \text{ if }u s_i < u,\\
		u, & \text{otherwise.}
              \end{cases}
\end{equation*}
By convention, we set $\id\demR u=u=u\demL \id$. The $u$-PDS $\bu=(\upu0,\upu1,\dots,\upu m)$ is computed iteratively starting from $\upu m=u$. For $c=m,m-1,\dots,1$, we set
\begin{equation}\label{eq:PDS_recursion}
  \upu{c-1}:=s_{i_c}^- \demR \upu c\demL s_{i_c}^+.
\end{equation}
\noindent Since $u\leq\br$, we have $\upu0=\id$.
We set $\Jo:=\{c\in[m]\mid \upu c=\upu{c-1}\}$. We refer to the indices in $\Jo$ as \emph{solid crossings} and to the indices in $\Hollow$ as \emph{hollow crossings}. 
The indices in $\Hollow$ form a reduced word for $u$, i.e., we have $|\Jo|=\ell(\br)-\ell(u)$.

\begin{remark}
When $\br \in I^m$ is a reduced word for a permutation and $u\leq \br$, the sequence $(u\pd{0}, \dots, u\pd{m})$ is a positive distinguished subexpression in the sense of~\cite[Definition~3.4]{MR}. 
The terminology of ``solid" and ``hollow" crossings is drawn from~\cite{MR}, who draw wiring diagrams in this way.
When the positive and negative subwords of $\br$ are both reduced and $u \leq \br$, $(u\pd{0}, \dots, u\pd{m})$ is a positive double distinguished subexpression in the sense of~\cite{WY}.
\end{remark}

For the rest of this section, we fix a pair $(u,\br)\in\Sn\times \DRW$ satisfying $u\leq \br$, and let $\bu$ be the $u$-PDS of $\br$.

\begin{remark}\label{rmk:u=w0}
All our constructions (including quivers, 3D plabic graphs, cluster algebra structures, and braid varieties) will be invariant under the following operation of appending hollow crossings on the right: if $i\in \pmn$ is such that $u<u':=s_i^-us_i^+$ then we are allowed to replace $(u,\br)$ with $(u',\br i)$. In particular, starting with any pair $(u,\br)$ satisfying $u\leq \br$, we can append hollow crossings to obtain a pair $(\wo,\brp)$ for some double braid word $\brp\in(\pmn)^{m+\ell(\wo)-\ell(u)}$.
\end{remark}

\begin{figure}
\def\wid{0.98\textwidth}
\begin{tabular}{c}
\includegraphics[width=\wid]{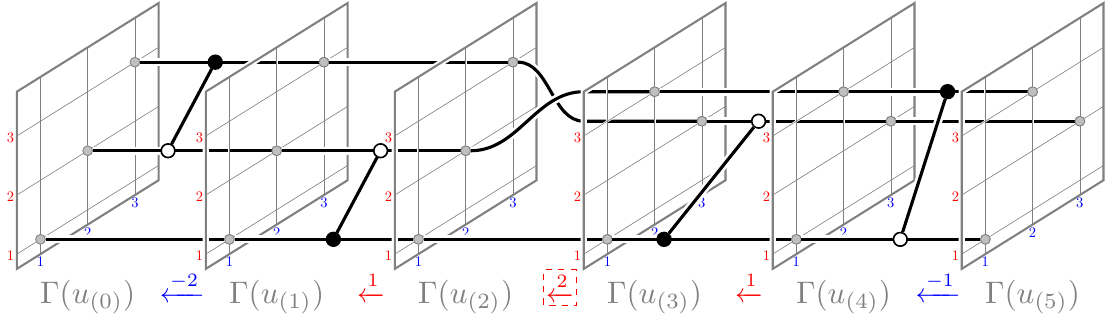}\\[5pt]
(a) 3D plabic graph $\Gubr$\\[5pt]
\includegraphics[width=\wid]{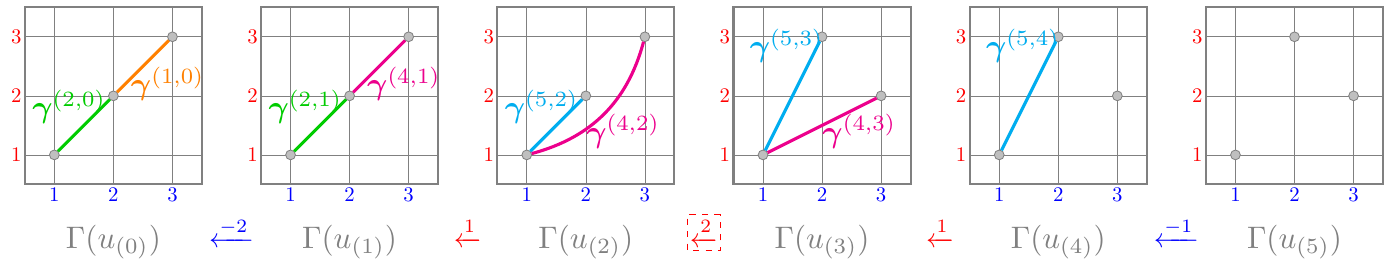}\\[5pt]
(b) monotone curves in $\Gubr$\\[5pt]
\includegraphics[width=\wid]{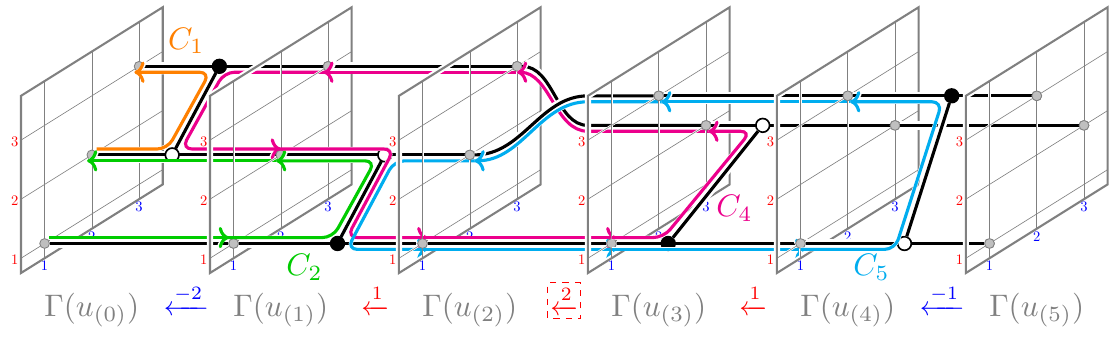}\\[5pt]
(c) relative cycles in $\Gubr$
\end{tabular}
  \caption{\label{fig:3D-ex}A 3D plabic graph $\Gubr$ for $u=s_2$ and $\br=(-2,1,2,1,-1)$; see \cref{ex:3D}.}
\end{figure}

\begin{figure}
\def\widd{0.27}
\def\wid{0.4}
\setlength{\tabcolsep}{1pt}
\hspace{-0.2in}
\begin{tabular}{c}
\begin{tabular}{ccc}
\igrw[\widd]{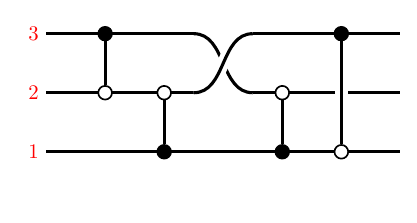}&
\igrw[\widd]{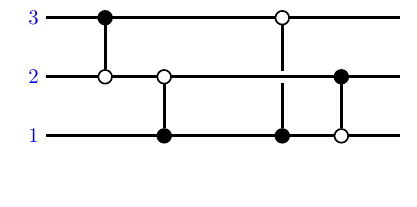}&
\igrw[\widd]{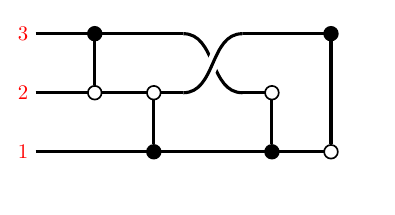}\\[-15pt]
(a) $\pired(\Gubr)$&
(b) $\piblue(\Gubr)$&
(c) $\pired(\Gubar)$
\end{tabular}
\\
\begin{tabular}{cc}
\igrw[\wid]{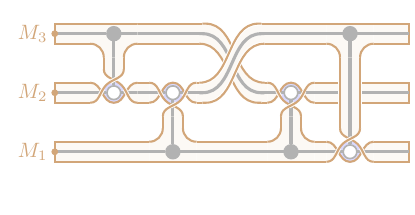}&
\igrw[\wid]{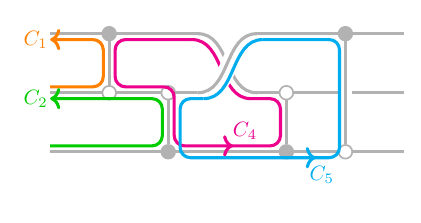}\\[-16pt]
(d) $\pired(\Subr)$&
(e) $\pired(\Cycle_c)$, $c\in\Jo$\\[-0pt]
\end{tabular}
\end{tabular}
  \caption{\label{fig:2D-ex}Projections of objects associated to 3D plabic graphs (cf. \cref{fig:3D-ex,ex:3D}).}
\end{figure}

\subsection{3D plabic graphs}\label{sec:3D_plabic}
We view permutations $z\in\Sn$ as bijections $[n]\to[n]$, with multiplication given by composition. Thus, if $z=xy$ then we have $z(i)=x(y(i))$ for $i\in[n]$. The \emph{permutation diagram} $\PD(z)\subset\Z^2$ of $z\in\Sn$ is the set of dots $(z(i),i)$ for $i\in[n]$. We use Cartesian coordinates for permutation diagrams.\footnote{The reader who likes permutation \emph{matrices} should therefore flip our permutation diagrams upside down, and should also remember that the convention for matrices is to list the vertical coordinate first and the horizontal coordinate second.} Thus, the dot $(i,j)$ is located in column~$i$ and row $j$, with the dot $(1,1)$ located in the bottom-left corner. We let $\preceq$ be the partial order on $\Z^2\subset \R^2$ given by $(x,y)\preceq (x',y')$ whenever $x\leq x'$ and $y\leq y'$.

 We consider $\Z^3$ with coordinates $(i,j,t)$, where the third coordinate $t$ is referred to as the \emph{time}. For each $c=0,1,\dots,m$, place the permutation diagram $\PD(\upu c)$ in the $t=c$ plane.%

We now define the 3D plabic graph $\Gubr$; see \cref{ex:3D} and \figref{fig:3D-ex}(a). We first give an informal description; see below for an explicit description in coordinates in $\Z^3$. 
\begin{definition}\label{dfn:3D_informal}
Start by drawing $n$ strands in $\R^3\supset\Z^3$ whose time coordinate is monotone increasing, so that for $c=0,1,\dots,m$, each of the $n$ dots of $\PD(\upu c)$ belongs to exactly one strand. For each $c\in[m]$, the permutation diagrams $\PD(\upu{c-1})$ and $\PD(\upu c)$ either are identical or differ by a row (if $i_c>0$) or by a column (if $i_c<0$) transposition. The strands of $\Gubr$ connect each dot of $\PD(\upu{c-1})$ to the corresponding dot of $\PD(\upu c)$. In addition, for each solid crossing $c\in\Jo$, we add a \emph{bridge edge} $\bridge_c$ at time $t=c-\frac12$ between the two strands $\strand_1\prec \strand_2$ participating in the solid crossing, where the partial order $\preceq$ is extended from the dots of $\PD(\upu{c-1})=\PD(\upu c)$ to the strands passing through them. If $i_c>0$ (resp., $i_c<0$), the strands $\strand_1,\strand_2$ are located in adjacent rows (resp., columns), and the bridge $\bridge_c$ is black at $\strand_1$ and white at $\strand_2$ (resp., white at $\strand_1$ and black at $\strand_2$). The vertex of $\bridge_c$ on $\strand_1$ (resp., on $\strand_2$) is called the \emph{start} (resp., the \emph{end}) of $\bridge_c$.
\end{definition}

 For $0<c<m$, we do not view the dots in $\PD(\upu c)$ as vertices of $\Gubr$. The dots in $\PD(\upu0)$ are viewed as degree $1$ \emph{marked boundary vertices} of $\Gubr$. The dots in $\PD(\upu m)$ are viewed as unmarked degree $1$  vertices of $\Gubr$. We let $\Gubar$ be obtained from $\Gubr$ by deleting these $n$ vertices and the $n$ edges incident to them; see \figref{fig:2D-ex}(c). 

\begin{figure}
\igrw[1.0]{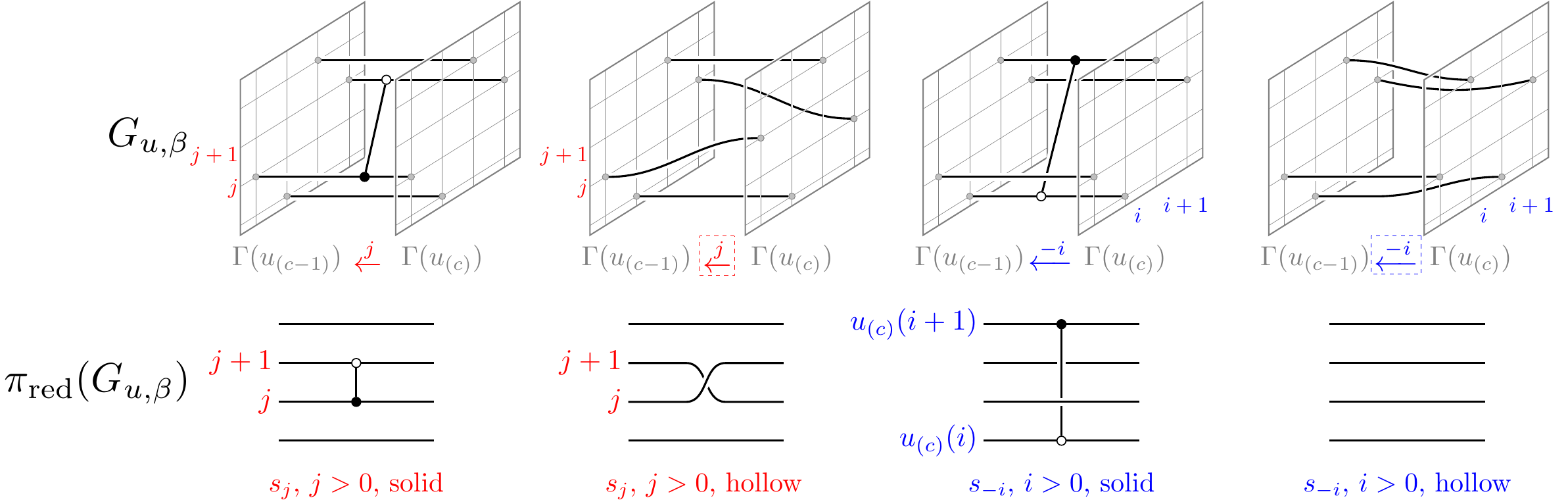}
  \caption{\label{fig:plabic-conven}Applying the red projection to a 3D plabic graph.}
\end{figure}

It is convenient to project 3D plabic graphs to the plane. There are two natural choices for such projections. The \emph{red projection} of $\Gubr$ is its image under the map $\pired:\R^3\to\R^2$, $(i,j,t)\mapsto (t,j)$.
 Similarly, the \emph{blue projection} of $\Gubr$ is obtained by applying the map $\piblue:\R^3\to\R^2$, $(i,j,t)\mapsto (t,i)$. We will mostly work with the red projection $\pired(\Gubr)$; see \cref{fig:2D-ex,fig:plabic-conven}.
 	We sometimes mention the \emph{faces} of $\pired(\Gubr)$ and $\piblue(\Gubr)$; by this we mean connected components of the complement. %

We now give a formal description of $\Gubr$. Suppose first that $c\in\Jo$ is a solid crossing. Then $\upu{c-1}=\upu c$, and we connect each dot $(i,j,c-1)$ in $\PD(\upu{c-1})$ to the corresponding dot $(i,j,c)$ of $\PD(\upu c)$, where $j=\upu{c-1}(i)=\upu c(i)$. Suppose now that $c\in\Hollow$ is a hollow crossing. If $i_c>0$, then $\PD(\upu{c-1})$ and $\PD(\upu c)$ differ by a row transposition, and we connect the dots accordingly. Specifically, letting $j:=i_c$, we connect:
\begin{itemize}
\item $(i,j,c-1)\in\PD(\upu{c-1})$ to $(i,j+1,c)\in\PD(\upu c)$,
\item $(i',j+1,c-1)\in\PD(\upu{c-1})$ to $(i',j,c)\in\PD(\upu c)$, and
\item $(i'',j'',c-1)\in\PD(\upu{c-1})$ to $(i'',j'',c)\in\PD(\upu c)$ for $j''\neq j,j+1$.
\end{itemize}
Similarly, if $i_c<0$, then $\PD(\upu{c-1})$ and $\PD(\upu c)$ differ by a column transposition. Letting $i:=|i_c|$, we connect:
\begin{itemize}
\item $(i,j,c-1)\in\PD(\upu{c-1})$ to $(i+1,j,c)\in\PD(\upu c)$,
\item $(i+1,j',c-1)\in\PD(\upu{c-1})$ to $(i,j',c)\in\PD(\upu c)$, and
\item $(i'',j'',c-1)\in\PD(\upu{c-1})$ to $(i'',j'',c)\in\PD(\upu c)$ for $i''\neq i,i+1$.
\end{itemize}
We add $|\Jo|$-many bridges to $\Gubr$. Consider a solid crossing $c\in\Jo$. Suppose first that $i_c>0$ and let $j:=i_c$. Then $\Gubr$ contains strands connecting the dots $(i,j,c-1)$ to $(i,j,c)$ and $(i',j+1,c-1)$ to $(i',j+1,c)$, for some $i,i'\in[n]$. In order for $c$ to be solid, we must have $i<i'$. We put a black vertex on the segment connecting $(i,j,c-1)$ to $(i,j,c)$, and a white vertex on the segment connecting $(i',j+1,c-1)$ to $(i',j+1,c)$, and connect these two vertices by an edge, which we call a \emph{bridge} and denote $\bridge_c$. Suppose now that $i_c<0$ and let $i:=|i_c|$. Then $\Gubr$ contains strands connecting the dots $(i,j,c-1)$ to $(i,j,c)$ and $(i+1,j',c-1)$ to $(i+1,j',c)$, for some $j,j'\in[n]$, and again we must have $j<j'$. We put a black vertex on the segment connecting $(i,j,c-1)$ to $(i,j,c)$, and a white vertex on the segment connecting $(i+1,j',c-1)$ to $(i+1,j',c)$, and connect these two vertices by a bridge~$\bridge_c$.

\subsection{Conjugate surfaces}\label{sec:surfaces}
Continuing \cref{sec:intro:quiver}, we endow the graph $\Gubr$ with the structure of a \emph{marked ribbon graph} in the language of~\cite[Section~3]{FoGo_moduli}. A \emph{ribbon graph} is a graph together with a choice, for each vertex $v$, of a cyclic orientation on the half-edges emanating from $v$. Taking the red projection of $\Gubr$, we choose the counterclockwise (resp., clockwise) orientation for each white (resp., black) vertex of $\Gubr$. The marked degree $1$ vertices of $\Gubr$ are the $n$ dots in $\PD(\upu0)$.\footnote{In~\cite{FoGo_moduli}, all degree $1$ vertices of a ribbon graph are considered automatically marked, but we do not mark the $n$ dots in $\PD(\upu m)$. Removing the dots in $\PD(\upu m)$ yields the graph $\Gubar$, which is truly a marked ribbon graph in the language of~\cite{FoGo_moduli}.}

\begin{remark}\label{rmk:color_switch}
From now on, we view $\Gubr$ as a ribbon graph, not as a bicolored graph. 
 When projecting a ribbon graph $\Gubr$ to the plane, we choose the color of each vertex to be white (resp., black) if its local half-edge orientation is counterclockwise (resp., clockwise). Thus, for example, we can change the color of a given vertex $q$ by altering the drawing of $G$; see \figref{fig:moves}(left). In this case, we label the resulting vertex by $\bar q$, emphasizing that $q$ and $\bar q$ represent the same vertex of $G$.
\end{remark}

We let $\Subr=\Surfof(\Gubr)$ be the marked surface with boundary associated to $\Gubr$ in a standard way: we replace every edge of $\Gubr$ by a thin rectangle, every vertex of $\Gubr$ by a disk, and glue the rectangles to the boundaries of the disks according to the local orientation around each vertex. 
 Thus, $\partial\Subr$ has several connected components, and we stress that we do \emph{not} glue disks to them. The surface $\Subr$ can be drawn using the red projection of $\Gubr$ as shown in \figref{fig:2D-ex}(d). In particular, $\Subr$ is orientable, with black and white vertices in the red projection of $\Gubr$ corresponding to the different sides of $\Subr$.

We apply a similar construction to $\Gubar$, and it is clear that the resulting surface $\Surfof(\Gubar)$ is homeomorphic to $\Subr$. 

Let $\MP$ be the set of marked points on $\partial\Subr$; thus, $|\MP|=n$. Let 
\begin{equation*}%
  \Lubr:=H_1(\Subr,\MP).
\end{equation*}
 (All relative homology groups we consider are with integer coefficients.) The elements of $\Lubr$, called \emph{relative cycles}, are represented by $\Z$-linear combinations of \emph{arcs}, where an arc is either an oriented closed curve embedded into the interior of $\Subr$ or an oriented curve embedded into $\Subr$ with both endpoints marked. Let 
\begin{equation*}%
  \Lubrst:=H_1(\Subr\setminus\MP,\partial\Subr\setminus\MP).
\end{equation*}
 We have an \emph{intersection form} on $(\Subr,\MP)$ which gives rise to a perfect pairing
\begin{equation}\label{eq:pairing}
  \<\cdot,\cdot\>: \Lubr\otimes \Lubrst\to\Z;
\end{equation}
see e.g.~\cite[Proposition~3.48]{CaWe} or~\cite[Section~6.1]{Mellit_cell}. For $\Cycle\in\Lubr$, $\Cycle'\in\Lubrst$, the \emph{intersection number} $\<\Cycle,\Cycle'\>$ is the integer obtained by counting signed intersection points between two generic relative cycles representing $\Cycle$ and $\Cycle'$.
Denote
\begin{equation*}%
  \Lmut := H_1(\Subr).
\end{equation*}
It can be considered as a sublattice of both $\Lubr$ and $\Lubrst$. 
Mutable variables will correspond (\cref{sec:rel_cycles}) to oriented cycles in $\Gubr$, and hence to elements of $\Lmut$; frozen cluster variables will correspond to relative cycles in $\Gubr$, i.e., to elements of $\Lubr$.

For future reference, we note:
\begin{lemma} \label{lem:summand}
$\Lmut$ splits off as a direct summand of both $\Lubr$ and $\Lubrst$.
\end{lemma}

\begin{proof}
We start with the case of $\Lubr$. Recall that $\Lmut = H_1(\Subr)$ and $\Lubr=H_1(\Subr,\MP)$, so we have an exact sequence
\[ 0 = H_1(\MP) \to H_1(\Subr) \to H_1(\Subr,\MP) \to H_0(\MP) .\]
Let $A$ be the image of  $H_1(\Subr,\MP) \to H_0(\MP)$. Since $A$ is a submodule of the free $\Z$-module $H_0(\MP)$, we have $A \cong \Z^k$ for some $k$. 
So we have a short exact sequence
\[ 0 \to H_1(\Subr) \to H_1(\Subr,\MP) \to \Z^k \to 0 . \]
Since $\Z^k$ is free, we have $\Lubr = H_1(\Subr,\MP) = H_1(\Subr) \oplus \Z^k = \Lmut \oplus \Z^k$.

The case of $\Lubrst$ is similar: we have an exact sequence
\begin{equation*}%
  H_1(\partial\Subr\setminus\MP) \to
  H_1(\Subr\setminus\MP) \to
  H_1(\Subr\setminus\MP,\partial\Subr\setminus\MP) \to
H_0(\partial\Subr\setminus\MP).
\end{equation*}
Here we also have $H_1(\partial\Subr\setminus\MP) = 0$ because $\partial\Subr\setminus\MP$ is homeomorphic to a disjoint union of line segments; cf. \figref{fig:2D-ex}(d). The rest of the proof proceeds as above.
\end{proof}

\subsection{Relative cycles}\label{sec:rel_cycles}
Fix a solid crossing $c\in\Jo$. Our goal is to associate to it a relative cycle $\Cycle_c\in\Lubr$. As in \cref{sec:intro:3D_plabic}, we will obtain $\Cycle_c$ as the boundary of a certain $2$-dimensional disk $\Disk_c$ inside $\R^3$.

\begin{figure}
\igrw[1.0]{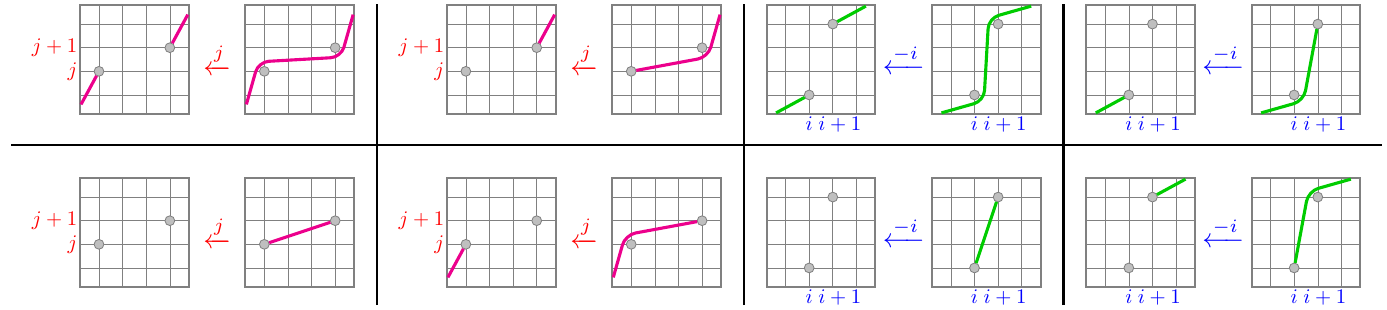}
  \caption{\label{fig:propag} Propagation rules for monotone curves; see \cref{sec:rel_cycles}. In any case not depicted here, the monotone curve is unchanged by the action of a solid crossing.}
\end{figure}

A \emph{monotone curve} inside a permutation diagram $\PD(z)\subset\R^2$ is a curve $\gamma:[0,1]\to\R^2$ whose endpoints are dots in $\PD(z)$, with no other dots of $\PD(z)$ on $\gamma$, and such that both coordinates of $\gamma$ are strictly monotone increasing. Recall that we write $(x,y)\preceq (x',y')$ if $x\leq x'$ and $y\leq y'$. Write $(x,y)\prec (x',y')$ if $(x,y)\preceq (x',y')$ and $(x,y)\neq(x',y')$. Thus, $\gamma(0)\prec \gamma(1)$. A \emph{monotone multicurve} is a collection $\bgamma=(\gamma_1,\gamma_2,\dots,\gamma_k)$ of monotone curves inside $\PD(z)$ such that \begin{equation*}%
  \gamma_1(0)\prec\gamma_1(1)\prec \gamma_2(0)\prec \gamma_2(1)\prec\cdots\prec\gamma_k(0)\prec\gamma_k(1).
\end{equation*}

The intersection of the disk $\Disk_c$ with each plane $t=r$, $0\leq r\leq c-1$, will be a monotone multicurve inside $\PD(\upu r)$ denoted $\bgamma\pu{c,r}$. For $r=c-1$, $\bgamma\pu{c,c-1}$ consists of a single monotone curve connecting the two dots on the strands which are connected by $\bridge_c$. We then compute $\bgamma\pu{c,r}$ iteratively for $r=c-2,\dots,1,0$, using the following \emph{propagation rules}. When passing from $\bgamma\pu{c,r}$ to $\bgamma\pu{c,r-1}$ for a solid crossing $r$, each monotone curve $\gamma$ in $\bgamma\pu{c,r}$ either is preserved or \emph{gets cut}. The cutting moves are shown in \cref{fig:propag}, and all other curves not shown in \cref{fig:propag} are preserved. When the crossing $r$ is hollow, 
each monotone curve $\gamma$ changes ``smoothly'' so that a dot $\Dot$ of $\PD(\upu r)$ is above (resp., below) $\gamma$ if and only if the dot $\Dot'$ of $\PD(\upu{r-1})$ that is on the same strand as $\Dot$ is above (resp., below) the image of $\gamma$ in $\PD(\upu{r-1})$; see e.g. \figref{fig:3D-ex}(b).

To give a coordinate description in the case of a solid crossing, let $r\in \Jo$ be such that $i_r>0$ and let $j:=i_r$. Let $\Dot:=(i,j)$ and $\Dot':=(i',j+1)$ be the dots in $\PD(\upu r)=\PD(\upu{r-1})$, for some $i,i'\in[n]$. Then the curve $\gamma$ \emph{gets cut} if and only if it passes weakly above $\Dot$ and weakly below $\Dot'$. Assume now that $i_r<0$ and let $i:=|i_r|$. Let $\Dot:=(i,j)$ and $\Dot':=(i+1,j')$ be the dots in $\PD(\upu r)=\PD(\upu{r-1})$, for some $j,j'\in[n]$. Then $\gamma$ \emph{gets cut} if and only if it passes weakly to the right of $\Dot$ and weakly to the left of $\Dot'$. In both cases, the cutting move consists of removing the part of $\gamma$ passing between $\Dot$ and $\Dot'$. (In particular, if neither $\Dot$ nor $\Dot'$ was an endpoint of $\gamma$ then we split $\gamma$ into two monotone curves $\gamma',\gamma''$ satisfying $\gamma'(1)=\Dot$ and $\gamma''(0)=\Dot'$. On the other hand, if both $\Dot$ and $\Dot'$ were endpoints of $\gamma$ then the whole of $\gamma$ disappears.)

\begin{remark}
These cutting rules are motivated by the combinatorics of almost positive subexpressions, described in \cref{sec:APS}.
\end{remark}

\begin{definition}\label{dfn:mut_fro}
If $\bgamma\pu{c,0}$ is empty, 
 then we declare $c$ to be \emph{mutable}, otherwise, we declare $c$ to be \emph{frozen}. We let $\Jfro$ and $\Jmut$ denote the sets of frozen and mutable indices, respectively. 
\end{definition}
\noindent Thus, we have a decomposition $\Jo=\Jfro\sqcup\Jmut$. 

If $c$ is mutable, we obtain a disk $\Disk_c$ inside $\R^3$ whose boundary $\partial\Disk_c$ is a cycle $\Cycle_c$ in $\Gubr$ that does not pass through any marked points. If $c$ is frozen, we treat $\bgamma\pu{c,0}$ as part of the boundary of the disk $\Disk_c$, and denote the rest of $\partial\Disk_c$ by $\Cycle_c:=\partial\Disk_c\setminus\bgamma\pu{c,0}$. In both cases, $\partial\Disk_c$ passes through the bridge $\bridge_c$, and we orient $\partial\Disk_c$ so that it is directed from the start to the end of $\bridge_c$; cf. \cref{dfn:3D_informal}. This induces an orientation on each arc in $\Cycle_c$, and therefore we obtain a relative cycle $\Cycle_c\in \Lubr$. See Figures~\hyperref[fig:3D-ex]{\ref*{fig:3D-ex}(c)} and~\hyperref[fig:2D-ex]{\ref*{fig:2D-ex}(e)}.
Thus, $C_c$ gives a class in $\Lubr$, which is in $\Lmut$ if and only if $c$ is mutable.

\begin{figure}
\begin{tabular}{c|c|c}
\igrw[0.2]{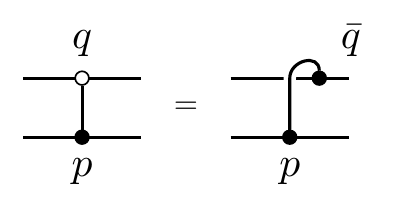} & 
\begin{tikzpicture}[baseline=(Z.base)]
\coordinate(Z) at (0,0);
\node(A) at (0,0.8){\igrw[0.5]{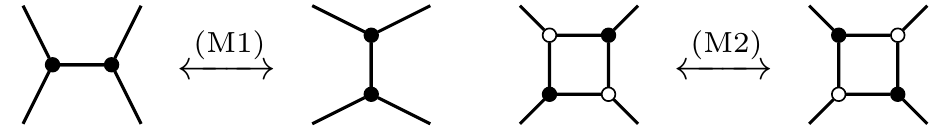}};
\end{tikzpicture} 
& 
\begin{tikzpicture}[baseline=(Z.base)]
\coordinate(Z) at (0,0);
\node(A) at (0,0.85){\igrw[0.2]{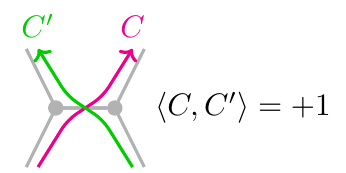}};
\end{tikzpicture} 
\end{tabular}
  \caption{\label{fig:moves} Left: switching the color of a vertex. Middle: contraction-uncontraction move~(M1) and square move~(M2). Right: sign convention for the intersection form of $\Subr$.}
\end{figure}

\subsection{The quiver}%
 Our goal is to associate an ice quiver $\Qubr$ to the pair $(u,\br)$.

The vertex set of $\Qubr$ is $\Vice:=\Jo$, with frozen vertices  $\Vfro:=\Jfro$ and mutable vertices $\Vmut:=\Jmut$ (cf. \cref{dfn:mut_fro}). Let $c\in\Vice$ and $d\in\Vmut$, and consider the corresponding relative cycles $\Cycle_c,\Cycle_d$. By the inclusion $\Lmut \subset \Lubrst$,
we may view $\Cycle_d$ as an element of $\Lubrst$. The \emph{cluster exchange matrix} $\Bice(\Qubr)=(\bice_{c,d})_{c\in\Vice,d\in\Vmut}$ of $\Qubr$ is given by 
\begin{equation}\label{eq:bice_dfn}
  \bice_{c,d}=\#\{\text{arrows $c\to d$ in $\Qice$}\} - \#\{\text{arrows $d\to c$ in $\Qice$}\}:=\<\Cycle_c,\Cycle_d\>.
\end{equation}
In other words, the (signed) number of arrows from $c$ to $d$ in $\Qubr$ is given by the intersection number $\<\Cycle_c,\Cycle_d\>$. We give an explicit algorithm for computing intersection numbers.

\begin{algorithm}\label{alg:int_form}
The intersection form $\<\Cycle,\Cycle'\>$ of the surface $\Subr=\Surfof(\Gubr)$ may be computed as follows. See \cref{fig:3D-quiver} and the top row of \figref{fig:half-arr}(right) for examples.
\setlength{\leftmargini}{0.5cm}
\setlength{\leftmarginii}{0.5cm}
\begin{itemize}
\item Viewing $\Cycle,\Cycle'$ as subgraphs of $\Gubr$, decompose $\Cycle \cap \Cycle'$ into a union of disjoint paths. 
\item For each of these paths $P$, we will have a contribution $\<\Cycle,\Cycle'\>|_P\in\{ -1, 0, 1 \}$; the intersection number $\<\Cycle,\Cycle'\>$ will be the sum of $\<\Cycle,\Cycle'\>|_P$ over the components $P$ of  $\Cycle \cap \Cycle'$.
\item Let $P = (p_0, p_1, \ldots, p_r)$ be one of the paths. The contribution $\<\Cycle,\Cycle'\>|_P$ will depend only on which neighbors of $p_0$ and $p_r$ are visited by $\Cycle$ and $\Cycle'$, and in which order.
\item  Draw $P$ in the plane with all vertices black (cf. \cref{rmk:color_switch}) and perturb $\Cycle,\Cycle'$ slightly so that they have either zero or one intersection point.
\item If $\Cycle,\Cycle'$ have zero intersection points, we have $\<\Cycle,\Cycle'\>|_P = 0$.
\item If $\Cycle,\Cycle'$ have one intersection point $p$, we have $\<\Cycle,\Cycle'\>|_P = +1$ if the tangent vectors of $(\Cycle,\Cycle')$ at $p$ form a positively oriented basis of the plane and $\<\Cycle,\Cycle'\>|_P = -1$ otherwise; see \figref{fig:moves}(right).
\end{itemize}
\end{algorithm}

\begin{remark}\label{rmk:orient}
As we pointed out in \cref{sec:surfaces}, the surface $\Subr$ is orientable. It follows that we have $\<\Cycle_d,\Cycle_d\>=0$ for all mutable $d\in\Jo$.
\end{remark}

\begin{figure}
\igrw[1.0]{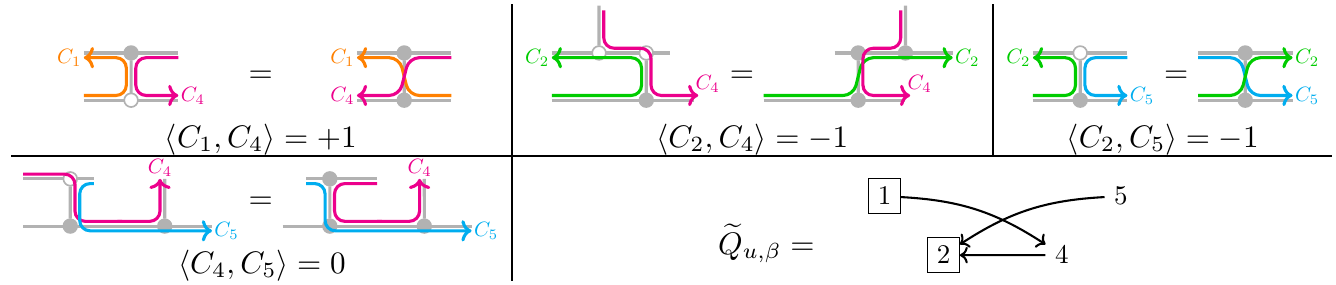}
  \caption{\label{fig:3D-quiver}Computing the quiver $\Qubr$ corresponding to \cref{fig:3D-ex,fig:2D-ex,ex:3D}; see \cref{alg:int_form}.}
\end{figure}

\begin{example}\label{ex:3D}
Let $u:=s_2\in S_3$ and $\br:=(-2,1,2,1,-1)$. Thus, $\Jo=\{1,2,4,5\}$. The graph $\Gubr$ is given in \figref{fig:3D-ex}(a), and its red and blue projections are shown in \figref{fig:2D-ex}(a,b), respectively. The monotone (multi)curves are computed in \figref{fig:3D-ex}(b) using the propagation rules from \cref{fig:propag}. The relative cycles $(\Cycle_c)_{c\in\Jo}$ are shown in Figures~\hyperref[fig:3D-ex]{\ref*{fig:3D-ex}(c)} and~\hyperref[fig:2D-ex]{\ref*{fig:2D-ex}(e)}, 
 and the surface $\Subr$ is shown in \figref{fig:2D-ex}(d). In \cref{fig:3D-quiver}, we use this data to compute the intersection numbers $\<\Cycle_c,\Cycle_d\>$ and the quiver $\Qubr$ via \cref{alg:int_form}. The frozen vertices of $\Qubr$ are boxed in \cref{fig:3D-quiver}.
\end{example}

\begin{remark}\label{rmk:embedding}
The simplified propagation rules shown in \cref{fig:intro:propag} are obtained from the rules in \cref{fig:propag} by applying the red projection. It is therefore important to distinguish between over/under-crossings when applying the rules in \cref{fig:intro:propag} to the red projection of a 3D plabic graph in \cref{fig:face}. However, the ribbon graph $\Gubr$ and the associated surface $\Subr$ are most naturally considered as abstract objects without a fixed choice of an embedding in $\R^3$. Thus, in the rest of our figures, e.g., when drawing the boundary of $\Subr$ in \figref{fig:2D-ex}(d) or the graphs in 
Figures~\ref{fig:easy-moves}--\ref{fig:hard-moves-2}, 
 the over/under-crossings are irrelevant, and are chosen arbitrarily.
\end{remark}

\begin{definition}\label{defn:quiverMutation}
	Given an ice quiver $\Qice$ and a mutable vertex $d\in\Vmut$, one can \emph{mutate} $\Qice$ in direction $d$ to obtain  another quiver $\Qice'=\mu_d(\Qice)$, the \emph{mutation} of $\Qice$ at $d$. Mutation preserves the sets of mutable and frozen vertices, and changes the arrows as follows:
	\begin{itemize}
		\item for each directed path $c\to d\to e$ in $\Qice$ of length $2$, add an arrow $c\to e$ to $\Qice'$;
		\item reverse all arrows in $\Qice$ incident to $d$;
		\item remove all directed $2$-cycles in the resulting directed graph, one at a time.
	\end{itemize}
\noindent This operation may create arrows between pairs of frozen vertices; such pairs are omitted.
\end{definition}

Recall that the quiver $\Qubr$ is obtained from the collection $(\Cycle_c)_{c\in\Jo}$ of relative cycles in $\Lubr$ via~\eqref{eq:bice_dfn}. Following~\cite[Section~4.1.2]{GoKe} and~\cite[Section~1.2]{FoGo_cluster}, we explain how for each $d\in\Jmut$, the quiver $\mu_d(\Qubr)$ is obtained in the same way via~\eqref{eq:bice_dfn} from another collection $(\mu_d(\Cycle_c))_{c\in\Jo}$ of relative cycles in $\Lubr$. Namely, we set 
\begin{equation}\label{eq:cycle_mut}
  \mu_d(\Cycle_c):=
  \begin{cases}
    \Cycle_c+\max \left(\<\Cycle_c,\Cycle_d\>,0\right) \Cycle_d, &\text{if $c\neq d$,}\\
    -\Cycle_d, &\text{ if $c=d$.}
  \end{cases}
\end{equation}
\noindent Since $d$ is mutable, we see that for each mutable $c$, the relative cycle $\mu_d(\Cycle_c)$ is still naturally an element of both $\Lubr$ and $\Lubrst$. 

\begin{lemma}[{\cite[Lemma~1.7]{FoGo_cluster}}]\label{lemma:FoGo}
The exchange matrix of $\mu_d(\Qubr)$ is given via~\eqref{eq:bice_dfn} by the intersection numbers of the relative cycles $(\mu_d(\Cycle_c))_{c\in\Jo}$.
\end{lemma}

\begin{remark}\label{rmk:Z_basis_will_show}
We will show (\cref{lemma:bridge_basis}) that $(\Cycle_c)_{c\in\Jo}$ is a $\Z$-basis of $\Lubr$. It follows that $(\mu_d(\Cycle_c))_{c\in\Jo}$ is a $\Z$-basis of $\Lubr$ as well.
It will usually \textbf{not} be true that $(\Cycle_c)_{c\in\Jmut}$ is a $\Z$-basis of $\Lmut$. However, we will show in \cref{lemma:mutable_completes_to_basis} that $(\Cycle_c)_{c\in\Jmut}$  can be completed to a $\Z$-basis of $\Lmut$. 
\end{remark}

\subsection{Distinguished subexpressions}\label{sec:APS}
The goal of this section is to explain how the propagation rules for monotone curves from \cref{sec:rel_cycles} reflect the combinatorics of \emph{almost positive subexpressions}. These results will be used in \cref{sec:seeds} for comparison to geometry.

Let $\br\in\DRW$ be a double braid word and $u\leq\br$. A $u$-subexpression $\bu$ of $\br$ is called \emph{distinguished} if $\upu{c}\leq s_{i_c}^-\upu{c-1}s_{i_c}^+$ for each $c\in[m]$. This notion originated in the study of the geometry of open Richardson varieties~\cite{Deo,MR,WY}.

\begin{definition}\label{def:almostPos}
	Let $d \in \Jo$. Let $\vu\pd{m}\apd{d}:=u$, and for $c=m,m-1,\dots,1$, define
	\[\vu\pd{c-1}\apd{d}:=\begin{cases}
		s_{i_d}^- \ast \vu\pd{d}\apd{d} \ast s_{i_d}^+,& \text{if $c=d$,}\\
		s_{i_{c}}^- \demR \vu\pd{c}\apd{d}\demL s_{i_{c}}^+,& \text{otherwise}.
	\end{cases} \]
	We call the sequence $(\vu\apd{d}\pd{0}, \dots, \vu\apd{d}\pd{m})$ the \emph{$(u,d)$-almost positive sequence} (\emph{$(u,d)$-APS}).
\end{definition}
In other words, the $u$-PDS is obtained by starting with $u$ and taking successive Demazure quotients according to the letters of $\br$. The $(u,d)$-APS is obtained by taking Demazure quotients up to crossing $d$; then making a ``mistake'' at crossing $d$ and taking Demazure product rather than Demazure quotient; then continuing to take Demazure quotients.

\begin{figure}
\igrw[0.5]{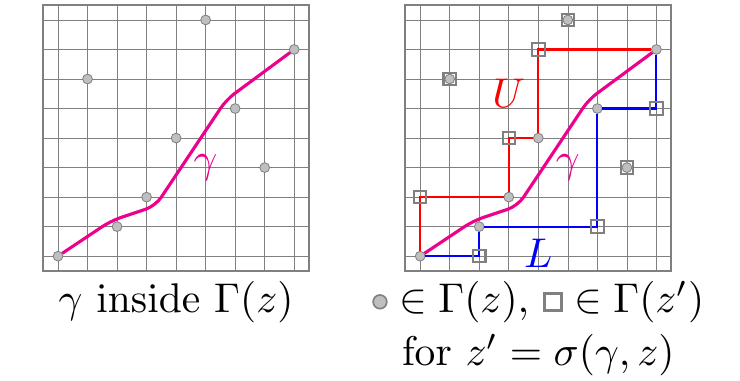}
  \caption{\label{fig:gswap} Recovering the almost positive subexpression from the corresponding monotone curve.}
\end{figure}

Consider a monotone curve $\gamma$ inside a permutation diagram $\PD(z)$ for $z\in\Sn$. 
 Take the smallest skew shape $\la/\mu$ containing $\gamma$ whose inner corners are at the dots of $\PD(z)$, and let $L$ and $U$ be its lower and upper boundaries, respectively; see \cref{fig:gswap}. Let $\Dots_\gamma$ be the set of dots of $\PD(z)$ contained in $U\cup L$. 
 Let $\Dots'_\gamma$ be the set of outer corners of $\la/\mu$, i.e., the set of lattice points where $L$ turns left or  $U$ turns right. We define $z':=\gswap(\gamma,z)\in\Sn$ to be the permutation such that $\PD(z')$ is obtained from $\PD(z)$ by replacing the dots in $\Dots_\gamma$ with the dots in $\Dots'_\gamma$. Given a monotone multicurve $\bgamma=(\gamma_1,\gamma_2,\dots,\gamma_k)$, we let $\gswap(\bgamma,z):=\gswap(\gamma_1,\gswap(\gamma_2,\dots,\gswap(\gamma_k,z)\dots)).$ The following result explains the relation between monotone curves and almost positive subexpressions.

\begin{figure}
\igrw[1.0]{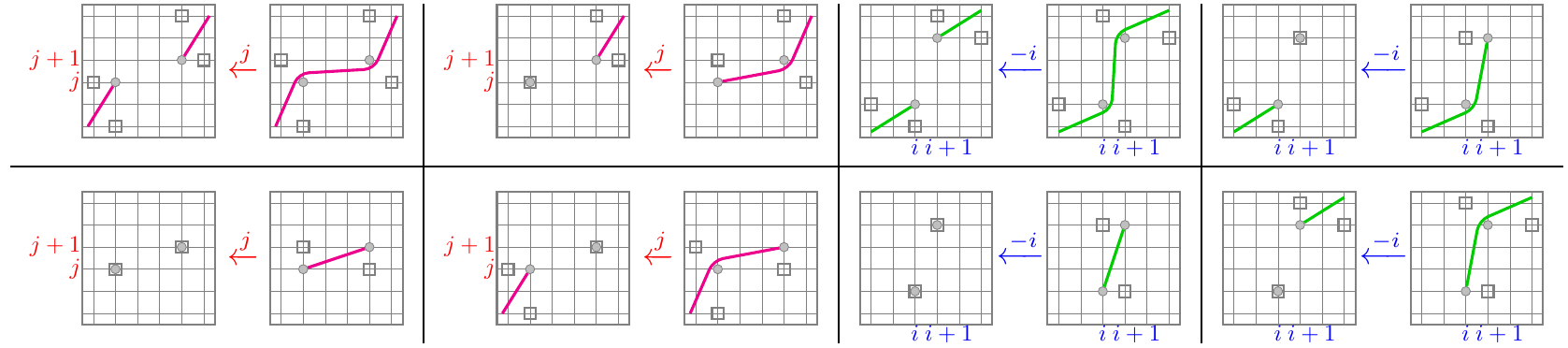}
  \caption{\label{fig:propag-x} Propagation rules (\cref{fig:propag}) shown together with the corresponding almost positive subexpressions.}
\end{figure}

\begin{proposition}\label{prop:gswap}
Let $d\in\Jo$. Then for all $c\leq d-1$, we have
\begin{equation}\label{eq:gswap}
  \vu\apd{d}\pd{c}=\gswap(\bgamma^{d,c},\upu c).
\end{equation}
\end{proposition}
\begin{proof}
For $c=d-1$, the permutations $\vd c$ and $\upu c$ differ by a single transposition which corresponds to creating the monotone curve $\bgamma\pu{d,d-1}$. For $c=d-2,\dots,0$, we check~\eqref{eq:gswap} by induction, since the propagation rules of \cref{fig:propag} turn out to exactly reflect the application of Demazure quotient to pairs $(\upu c,\vd c)$; see \cref{fig:propag-x}.
\end{proof}

\begin{corollary}\label{cor:mutable-from-aps}
A solid index $d\in\Jo$ is mutable if and only if $\vd 0=\id$.
\end{corollary}

\begin{figure}
\begin{tabular}{cc}
\begin{tikzpicture}[baseline=(Z.base)]
\coordinate(Z) at (0,0);
\node(A) at (0,2.3){\igrw[0.2]{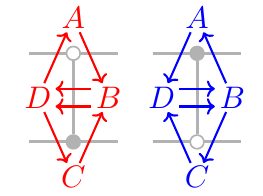}};
\end{tikzpicture}
  &
\igrw[0.75]{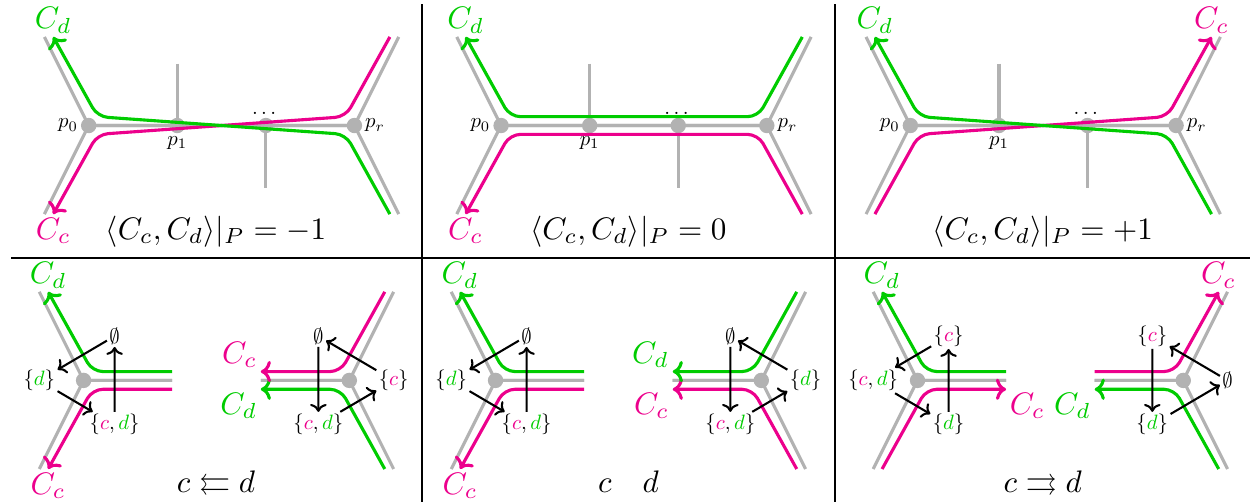}
\end{tabular}
  \caption{\label{fig:half-arr} The half-arrow description of $\Qubr$; see \cref{sec:half-arrow-descr}.}
\end{figure}

\subsection{Half-arrow description of \texorpdfstring{$\Qubr$}{the quiver}}\label{sec:half-arrow-descr}
We give an alternative description of $\Qubr$ that will be useful in \cref{prop:seedFromLabeledGraphs}. For each red crossing $c\in\Jo$, $i_c>0$, consider the four faces $A,B,C,D$ around the red projection of the bridge $\bridge_c$ as shown in \figref{fig:half-arr}(left). To each face $X\in\{A,B,C,D\}$ we associate a collection $\Chminsetr(X)$ of indices $d\in\Jo$ such that $X$ is inside the red projection of the disk $\Disk_d$; cf.~\eqref{eq:Chmin}. Then for each pair $(X,Y)=(A,B),(B,D),(D,A),(C,B),(B,D),(D,C)$ (note that $(B,D)$ is listed twice), we draw a half-arrow from each element of $\Chminsetr(X)$ to each element of $\Chminsetr(Y)$. Similarly, for each blue crossing $c\in\Jo$, $i_c<0$, we consider the four faces $A,B,C,D$ in the blue projection, and then for each $(X,Y)=(A,B),(B,D),(D,A),(C,B),(B,D),(D,C)$, we draw a half-arrow from each element of $\Chminsetb(Y)$ to each element of $\Chminsetb(X)$. Here, $\Chminsetb(X)$ is the set of indices $d\in\Jo$ such that $X$ is inside the blue projection of $\Disk_d$. We obtain a collection of half-arrows between the elements of $\Jo$.

\begin{proposition}\label{prop:half_arr}
For $c,d\in\Jo$, the difference between the number of half-arrows $c\to d$ and the number of half-arrows $d\to c$ equals $2\<\Cycle_c,\Cycle_d\>$.
\end{proposition}
\noindent In other words, by~\eqref{eq:bice_dfn}, the quiver $\Qubr$ is obtained by just summing up the signed half-arrow contributions and dividing the result by $2$.

\begin{proof}
The above half-arrow description can be replaced by the following local description. Consider a vertex $p$ of $\Gubr$, and draw the neighborhood of $\Subr$ around $p$ so that $p$ is black. Label the faces around $p$ by $A,B,C$ in counterclockwise order. For $X\in\{A,B,C\}$, let $\Chminsetp(X)$ denote the set of indices $c\in\Jo$ such that $\Cycle_c$ passes through $p$ with $X$ to the left of $\Cycle_c$. Then for each $(X,Y)=(A,B),(B,C),(C,A)$, draw a half-arrow from each element of $\Chminsetp(X)$ to each element of $\Chminsetp(Y)$.

Consider two relative cycles $\Cycle_c,\Cycle_d$, for $c,d\in\Jo$. Recall from \cref{alg:int_form} that the intersection number $\<\Cycle_c,\Cycle_d\>$ may be computed as a sum of local contributions $\<\Cycle_c,\Cycle_d\>|_P$ from maximal by inclusion paths in $\Cycle_c\cap\Cycle_d$. Such contributions are shown in the top row of \figref{fig:half-arr}(right). 
 On the other hand, as shown in the bottom row of \figref{fig:half-arr}(right), the (signed) contribution to the number of half-arrows $c\to d$ is $\pm1$ for $p_0$ and $p_r$ (and zero for each of $p_1,\dots,p_{r-1}$) so that the combined half-arrow contribution from $p_0$ and $p_r$ is exactly $2\<\Cycle_c,\Cycle_d\>|_P$. Summing over all such paths $P$, the result follows.
\end{proof}

\section{Invariance under moves}\label{sec:invar-under-moves}
In this section, we show that the mutation class of the quiver $\Qubr$ defined above is invariant under applying \emph{double braid moves} to $\br$.

\subsection{Moves for 3D plabic graphs}\label{sec:moves-for-3D}

\begin{figure}
\igrw[1.0]{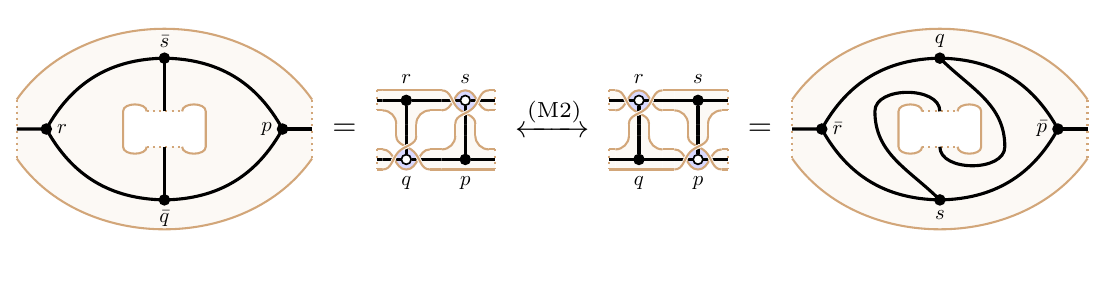}\vspace{-0.3in}
  \caption{\label{fig:sq-surf}Applying a square move~\protect\sqmv to $\Gubr$ preserves the surface $\Subr$ but changes the embedding of $\Gubr$ in $\Subr$. See also~\cite[Figure~20]{GoKe}.}
\end{figure}

Just as in Postnikov's theory~\cite{Pos}, our two main moves for 3D plabic graphs are the \emph{contraction-uncontraction move \cumv} and the \emph{square move \sqmv}, shown in \figref{fig:moves}(middle). The move \cumv can be performed on any edge $e=\{p,q\}$ of a 3D plabic graph $G$. Specifically, we draw $G$ in the plane so that $p$ and $q$ are of the same color, and then we apply the usual contraction-uncontraction move, producing two other vertices of the same color as $p$ and $q$. The move \sqmv can be performed on any $4$-cycle $(p,q,r,s)$ in $G$, as shown in \cref{fig:sq-surf}. The moves \cumv--\sqmv preserve the conjugate surface $\Surf=\Surfof(G)$ but change the embedding of $G$ inside of $\Surf$. In particular, it will be important later that as one applies the various \emph{double braid moves} to the double braid word $\br$, the surface $\Subr$ stays unchanged throughout the process.
\begin{remark}
  Surprisingly, the square move~\sqmv can be obtained by performing two contraction-uncontraction moves~\cumv on $G$; see \cref{fig:sqmv-vs-cumv}. We still distinguish~\sqmv as a separate transformation, for the following reason. We have associated three kinds of objects to a pair $(u,\br)$: a 3D plabic graph $\Gubr$, a marked surface $\Subr$, and a collection $(\Cycle_c)_{c\in\Jo}$ of relative cycles in $\Subr$. While $\Gubr$ determines $\Subr$, neither $\Gubr$ nor $\Subr$ determines $(\Cycle_c)_{c\in\Jo}$. Thus, if one applies moves~\cumv--\sqmv to $\Gubr$, one has to specify additional rules for how the tuple $(\Cycle_c)_{c\in\Jo}$ changes. These rules are given in \cref{thm:invariance} below: the tuple $(\Cycle_c)_{c\in\Jo}$ changes according to~\eqref{eq:cycle_mut} when we apply \emph{mutation braid moves} corresponding to~\sqmv, and is preserved when we apply \emph{non-mutation braid moves} corresponding to other sequences of moves~\cumv.
\end{remark}

\begin{figure}
\igrw[0.8]{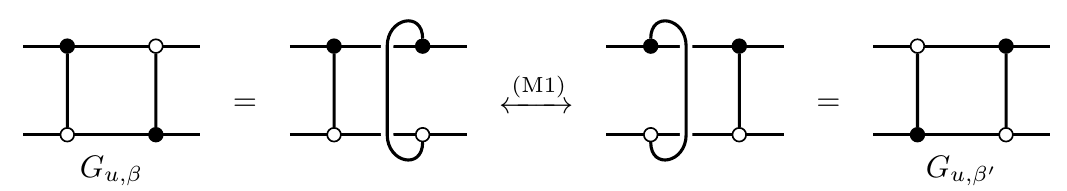}
  \caption{\label{fig:sqmv-vs-cumv}Representing the square move~\protect\sqmv as a composition of two contraction-uncontraction moves~\protect\cumv.}
\end{figure}

\subsection{Moves for double braid words}

Two double braid words are \emph{(double braid) equivalent} if they are related by a sequence of the following double braid moves: %
\begin{enumerate}[\normalfont(B1)]
	\item \label{bm1} $i j \leftrightarrow j i$ \quad if $i, j\in\pmn$ have different signs;
	\item \label{bm2} $i j \leftrightarrow j i$ \quad if $i, j\in\pmn$ have the same sign and $|i-j|>1$;
	\item \label{bm3} $i j i \leftrightarrow j i j$ \quad if $i, j\in\pmn$ have the same sign and $|i-j|=1$.
\end{enumerate}

\begin{definition}\label{dfn:mutation_etc_moves}
We say that a double braid move is \emph{fully solid} if all of the indices involved are solid. Suppose that $i_{d-1}=-j$ and $i_{d}=i$ for some $i,j\in I$ and $2\leq d\leq m$. Then we say that the move~\bmref{bm1} swapping these two indices is \emph{special} if $\upu{d-1} s_i=s_j\upu{d-1}$, and \emph{solid-special} if it is special and fully solid. Motivated by the following theorem, we call ~\bmref{bm1} (solid-special) and~\bmref{bm3} (fully solid) \emph{mutation moves}. All other braid moves are \emph{non-mutation moves}.
\end{definition}

\begin{theorem}\label{thm:invariance} 
  The mutation type of $\Qubr$ is invariant under double braid moves~\bmref{bm1}--\bmref{bm3} on $\br$.  More precisely:
  \begin{enumerate}[\normalfont(1)]
\item\label{item:inv1} Under mutation moves, the quiver $\Qubr$ changes by a mutation. The relative cycles $(\Cycle_c)_{c\in\Jo}$ change according to~\eqref{eq:cycle_mut}. The graph $\Gubr$ changes by a square move \sqmv.
\item\label{item:inv2} Under non-mutation moves, the quiver $\Qubr$ and the relative cycles $(\Cycle_c)_{c\in\Jo}$ are unchanged, up to relabeling. The graph $\Gubr$ changes by a sequence of contraction-uncontraction moves \cumv.
  \end{enumerate}
In both cases, the surface $\Subr$ is unchanged.
\end{theorem}
We prove \cref{thm:invariance} in the next two subsections. Throughout, we let $\br$ and $\brp$ be two braids related by one of the moves in \cref{thm:invariance}. We denote the 3D plabic graphs by $\Gubr$ and $\Gubrp$, the surfaces by $\Subr$ and $\Subrp$, and the relative cycles by $(\Cycle_c)_{c\in\Jo}$ and $(\Cycle'_c)_{c\in\Jop}$, respectively.

\begin{figure}
\begin{tabular}{c|c}
\begin{tikzpicture}[baseline=(Z.base)]
\coordinate(Z) at (0,0);
\node(A) at (0,0.85){\igrw[0.4]{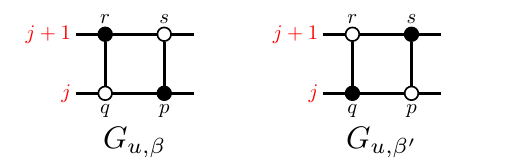}};
\end{tikzpicture}
  &
\begin{tikzpicture}[baseline=(Z.base)]
\coordinate(Z) at (0,0);
\node(A) at (0,0.85){\igrw[0.55]{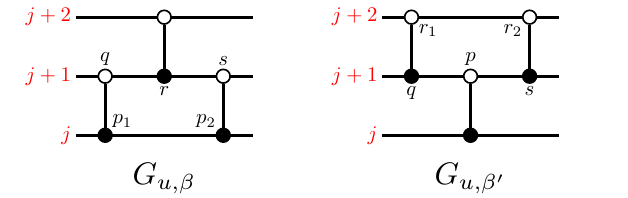}};
\end{tikzpicture}
\end{tabular}
  \caption{\label{fig:mut-graphs}Applying braid moves from part~\eqref{item:inv1} of \cref{thm:invariance} results in applying a square move to $\Gubr$.}
\end{figure}

\subsection{Mutation moves}\label{sec:inv_mut}
We prove part~\eqref{item:inv1} of \cref{thm:invariance}. Assume that we are applying a mutation move involving either the indices $i_{d-1},i_d$ or $i_{d-2},i_{d-1},i_d$ for some $d\in\Jo$. By \cref{lemma:FoGo}, it suffices to show that the relative cycles change according to~\eqref{eq:cycle_mut}, i.e., that $\Cycle'_c=\mu_d(\Cycle_c)$ for all $c\in\Jo$.

\subsubsection{Applying \protect\bmref{bm1} (solid-special)}\label{sec:inv_mut_bm1}
We would like to swap two solid crossings $i_{d-1}=-i$ and $i_{d}=j$ (with $i,j\in I$) such that $\upu{d-1}s_i=s_j\upu{d-1}$.  Since $d-1,d$ are both solid, we have $\upu{d-2}=\upu{d-1}=\upu{d}$. In addition, using $\upu{d}s_i=s_j\upu{d}$, we find that $\upu{d}(i)=j$ and $\upu{d}(i+1)=j+1$. The graph $\Gubr$ therefore contains a $4$-cycle spanned by the pair $(\bridge_{d-1},\bridge_d)$ of bridges of opposite color; see \figref{fig:mut-graphs}(left).  It follows that the graph $\Gubrp$ is obtained from $\Gubr$ by applying a square move \sqmv. In particular, we have $\Subr=\Subrp$; see \cref{fig:sq-surf}. 

We will show that $\Cycle'_c=\mu_d(\Cycle_c)$ by classifying all possible cases of how a relative cycle $\Cycle_c$ can look around the square $psrq$. This amounts to classifying the behavior of the monotone multicurves $\bgamma\pu{c,d},\bgamma\pu{c,d-1},\bgamma\pu{c,d-2}$ around the dots $\Dot:=(i,j)$ and $\Dot':=(i+1,j+1)$ of $\PD(\upu d)$. Note that, with the exception of the relative cycles $\Cycle_d$ and $\Cycle_{d-1}$, the monotone multicurve $\bgamma\pu{c,d}$ determines $\bgamma\pu{c,d-1}$ and $\bgamma\pu{c,d-2}$. Moreover, it suffices to consider the behavior of each monotone curve of $\bgamma\pu{c,d}$ separately; cf. \cref{rmk:two_curves} below.

\begin{figure}
\igrw[0.8]{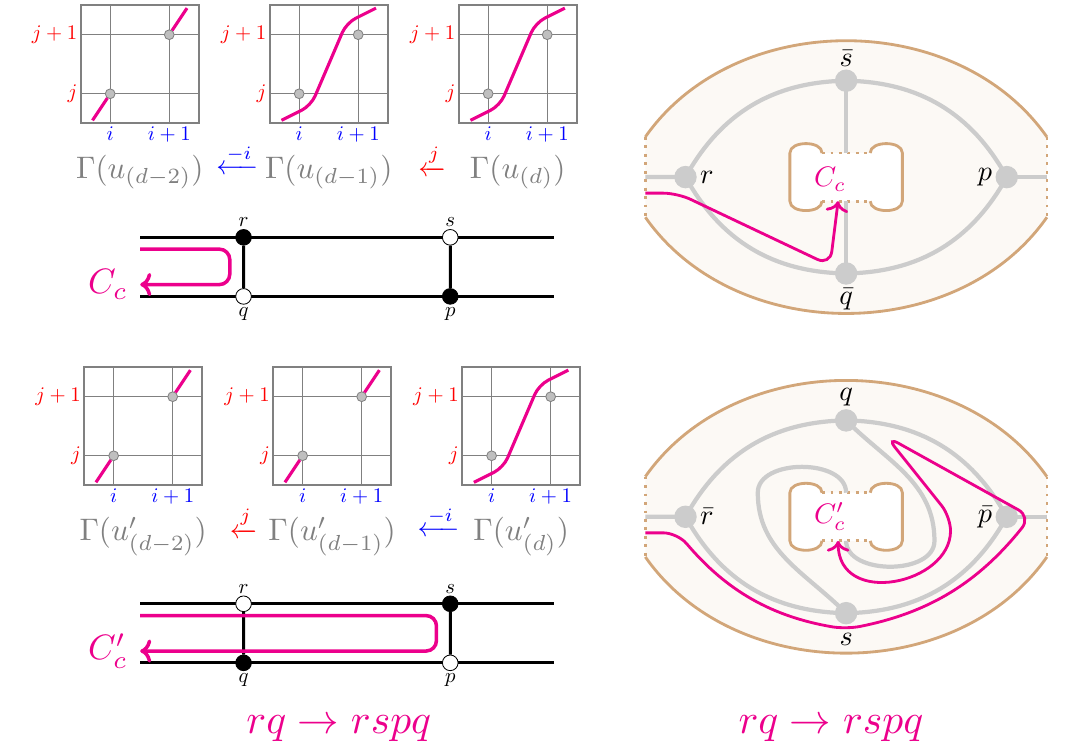}
  \caption{\label{fig:mut-ij-ex}An example of computing the signature $rq\to rspq$ of a relative cycle $\Cycle_c$ under the move~\protect\bmref{bm1} (solid-special).}
\end{figure}

For a relative cycle $\Cycle_c$ in $\Gubr$, its \emph{signature} is the ordered list of vertices of the square $psrq$ that $\Cycle_c$ passes through. For example, let us consider a monotone curve $\gamma$ inside $\PD(\upu d)$ passing below $\Dot$ and above $\Dot'$, as in \cref{fig:mut-ij-ex}. Using the rules in \cref{fig:propag}, we propagate $\gamma$ to $\PD(\upu{d-1})$ and $\PD(\upu{d-2})$, and find that the relevant part of $\Cycle_c$ passes first through $r$ and then through $q$ in $\Gubr$. Therefore, the signature of this relative cycle is $rq$. Repeating the same procedure for the graph $\Gubrp$, we find that the signature of $\Cycle_c$ in $\Gubrp$ is $rspq$; see \figref{fig:mut-ij-ex}(left). As shown in \figref{fig:mut-ij-ex}(right), the cycle with signature $rq$ in $\Gubr$ is homotopic to the cycle with signature $rspq$ in $\Gubrp$ when we view them as cycles in the ambient surface $\Subr=\Subrp$. 

\begin{figure}
\igrw[1.0]{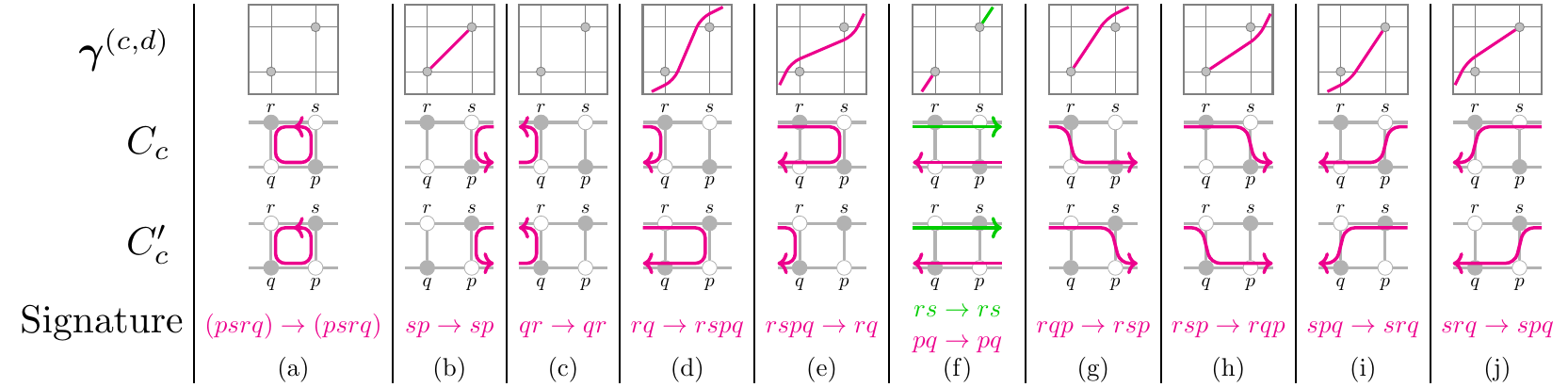}
  \caption{\label{fig:mut-ij-curves} The possible relative cycles and their signatures for the proof in \cref{sec:inv_mut_bm1}.}
\end{figure}

In \cref{fig:mut-ij-curves}, we list all possible relative cycles that pass through at least one vertex of the square $psrq$, given together with their monotone curves $\gamma$ inside $\PD(\upu d)$ and the signatures in $\Gubr$ and $\Gubrp$. For example, $\Cycle_d$ is shown in \figref{fig:mut-ij-curves}(a), and the curve $\gamma$ from \cref{fig:mut-ij-ex} is shown in \figref{fig:mut-ij-curves}(d). 

We will want to understand how $C_c$ and $C'_c$ relate, not as drawn in the planar projections in \cref{fig:mut-ij-curves}, but as drawn on the surface $\Subr=\Subrp$. 
Comparing the signatures to \cref{fig:sq-surf}, we see that for all $c\in\Jo$, the following are equivalent:
\begin{itemize}
\item $\Cycle_c\neq \Cycle'_c$ (as elements of $\Lubr=\Lubrp$);
\item either $c=d$ or $\<\Cycle_c,\Cycle_d\>>0$ (in which case $\<\Cycle_c,\Cycle_d\>=1$);
\item $\Cycle_c$ is shown in \figref{fig:mut-ij-curves}(a,b,c).
\end{itemize}
Note that $\Cycle_d$ is represented in the notation of \cref{fig:sq-surf} as the cycle passing through the vertices of the square $psrq$. Thus, in order to have $\<\Cycle_c,\Cycle_d\>>0$, when $\Cycle_c$ is drawn in \figref{fig:sq-surf}(far left), it needs to start inside $psrq$ and end outside of $psrq$. This happens precisely when $\Cycle_c$ is one of the relative cycles shown in \figref{fig:mut-ij-curves}(b,c). We see that indeed in all cases, we have $\Cycle'_c=\mu_d(\Cycle_c)$.

\begin{remark}\label{rmk:two_curves}
In general, $\Cycle_c$ is represented by a monotone \emph{multi}curve inside $\PD(\upu{d})$, and thus, for example, the two monotone curves shown in \figref{fig:mut-ij-curves}(f) could be parts of a single monotone multicurve. However, this does not affect our analysis because the intersection number $\<\Cycle_c,\Cycle_d\>$ is additive over all intersection points in $\Cycle_c\cap \Cycle_d$.
\end{remark}

\begin{figure}
\igrw[1.0]{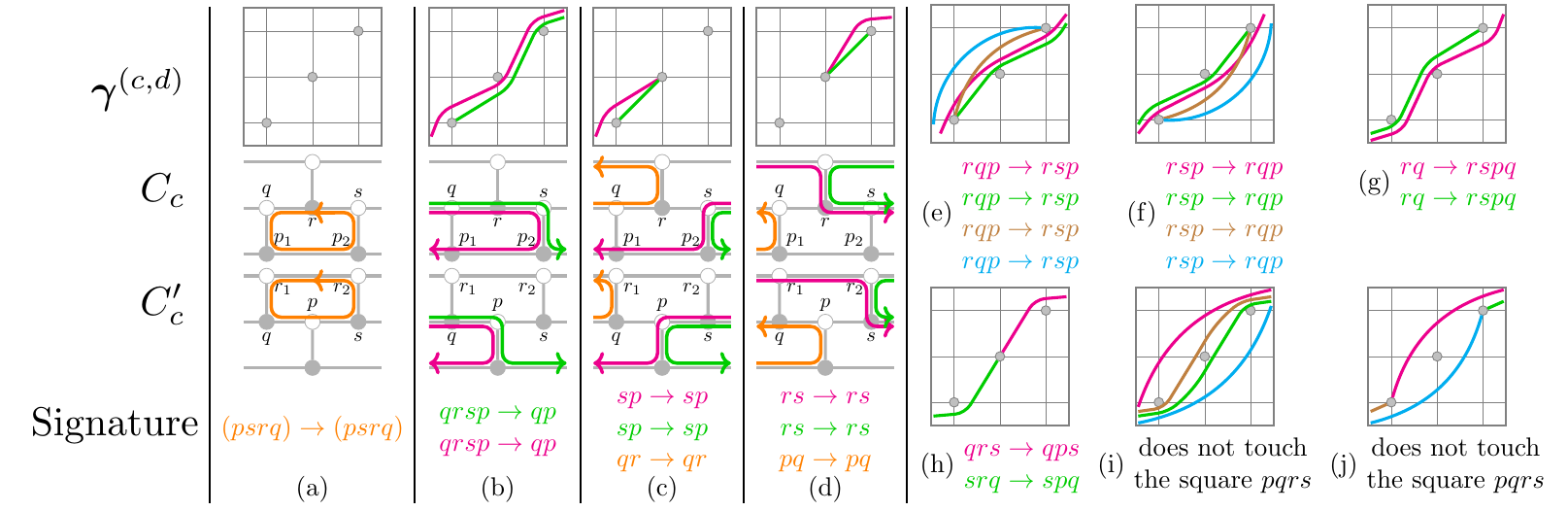}
  \caption{\label{fig:mut-iji-curves}The possible signatures for the proof in \cref{sec:inv_mut_bm3}.}
\end{figure}

\subsubsection{Applying \protect\bmref{bm3} (fully solid)}\label{sec:inv_mut_bm3}
We proceed using a similar strategy. Suppose that $i_{d-2}=j$, $i_{d-1}=j+1$, and $i_d=j$ for some $j\in I$, with all three indices solid. The 3D plabic graphs $\Gubr$ and $\Gubrp$ are shown in \figref{fig:mut-graphs}(right). In particular, after applying a contraction-uncontraction move \cumv to the vertices $p_1,p_2$ of $\Gubr$ and to the vertices $r_1,r_2$ of $\Gubrp$, we obtain two graphs which differ by a square move \sqmv, where the square has vertices $psrq$. Abusing notation, we denote these two graphs again by $\Gubr$ and $\Gubrp$. 

Let us now classify the relative cycles. There are a total of $26$ options for how a monotone curve can look like inside $\PD(\upu d)$. In addition, there are $3$ more relative cycles $\Cycle_d$, $\Cycle_{d-1}$, $\Cycle_{d-2}$ which do not correspond to any monotone curves in $\PD(\upu d)$, shown in orange in \figref{fig:mut-iji-curves}(a,c,d). Out of these $29$ options, $8$ relative cycles shown in \figref{fig:mut-iji-curves}(i,j) do not pass through any of the vertices of the square $psrq$ (after applying the above  contraction-uncontraction moves). The remaining $21$ relative cycles, together with their signatures in $\Gubr$ and $\Gubrp$, are shown in \figref{fig:mut-iji-curves}(a--h). 

Comparing the signatures to \cref{fig:sq-surf}, we see that for all $c\in\Jo$, the following are equivalent:
\begin{itemize}
\item $\Cycle_c\neq \Cycle'_c$;
\item either $c=d$ or $\<\Cycle_c,\Cycle_d\>>0$ (in which case $\<\Cycle_c,\Cycle_d\>=1$);
\item $\Cycle_c$ is shown in \figref{fig:mut-iji-curves}(a,b,c).
\end{itemize}
We again get that $\Cycle'_c=\mu_d(\Cycle_c)$ in each case. This completes the proof of part \eqref{item:inv1} of \cref{thm:invariance}.

\begin{figure}
\igrw[0.5]{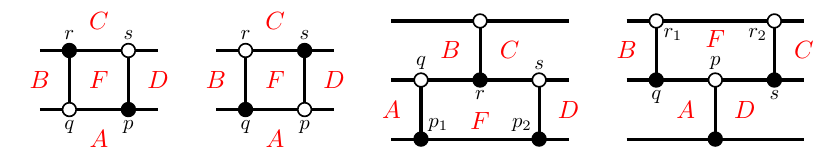}
  \caption{\label{fig:sq-face-lab}Labeling the faces when applying the square move.}
\end{figure}

\subsubsection{Chamber minors}
Suppose we are applying one of the mutation moves from part \eqref{item:inv1} of \cref{thm:invariance} at $\Cycle_d$ for some $d\in\Jmut$. Then (after possibly applying moves \cumv from \cref{sec:inv_mut_bm3}), the graph $\Gubr$ contains a square $psrq$. Let $G$ be the portion of the red projection of $\Gubr$ located in the small neighborhood of $psrq$; thus, $G$ is planar. Let $F$ be the face of $G$ inside this square, and let $A,B,C,D$ be the four faces adjacent to the square in clockwise order; see \cref{fig:sq-face-lab}. (We only consider $A,B,C,D$ in a small neighborhood of the square.) 
For $c\in\Jo$ and a face $E\in \{F,A,B,C,D\}$, we will write $\ord_c(E)=1$ if $E$ is inside $\Cycle_c$ and $\ord_c(E)=0$ otherwise. As in \cref{sec:intro:seed}, we say that $E$ is \emph{inside} $\Cycle_c$ if $E$ is contained inside the red projection of the disk $\Disk_c$, or in other words, if $E$ is to the left of the curve representing the red projection of $\Cycle_c$. We define $\ord'_c(E)$ similarly using the graph $\Gubrp$ and the cycles $\Cycle'_c$ in it. It follows from part \eqref{item:inv1} of \cref{thm:invariance} that the difference $\Cycle_c-\Cycle'_c$ in $\Lubr$ is always a multiple of $\Cycle_d$. Thus, we have $\ord'_{c}(E)=\ord_c(E)$ for $E\neq F$. (Alternatively, this can be seen directly from \cref{fig:mut-ij-curves,fig:mut-iji-curves}.) The following result will be used in the proof of \cref{prop:checkingMutMoves}.
\begin{lemma}
  For all $c\in\Jo\setminus\{d\}$,
  we have
\begin{equation}\label{eq:ord_min_lemma}
  \ord_c(F)+\ord'_c(F)=\min(\ord_c(A)+\ord_c(C),\ord_c(B)+\ord_c(D)).
\end{equation}
\end{lemma}
\begin{proof}
Follows by inspection from \cref{fig:mut-ij-curves,fig:mut-iji-curves}. For example, if $\Cycle_c$ is given in \figref{fig:mut-ij-curves}(d) then~\eqref{eq:ord_min_lemma} becomes $1+0=\min(0+1,1+1)$. Similarly, if both monotone curves in \figref{fig:mut-ij-curves}(f) are parts of the same monotone multicurve corresponding to $\Cycle_c$ then~\eqref{eq:ord_min_lemma} becomes $0+0=\min(0+0,1+1)$.
\end{proof}

\begin{figure}
\igrw[0.8]{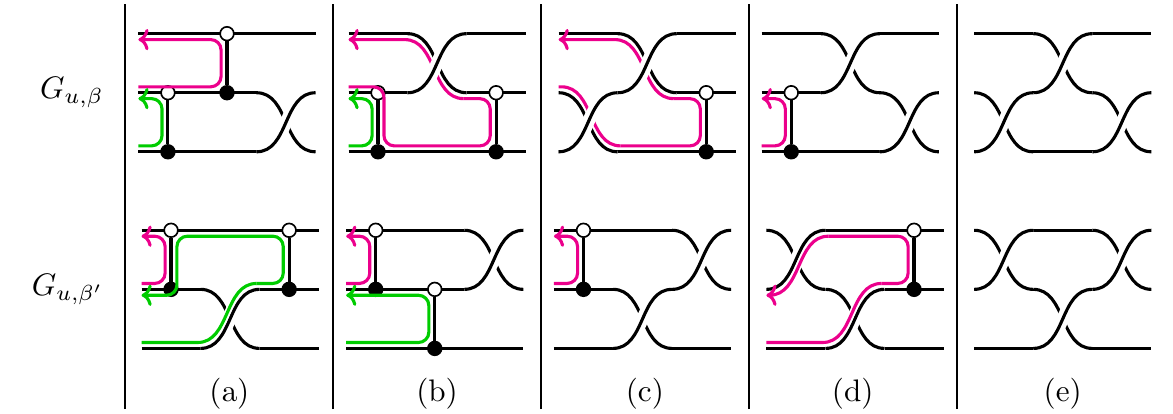}
  \caption{\label{fig:non-mut-iji} Applying move~\protect\bmref{bm3} (not fully solid).}
\end{figure}

\subsection{Non-mutation moves}\label{sec:nonMutQuiver}
We prove part~\eqref{item:inv2} of \cref{thm:invariance}. The following result will be convenient to show equalities of the form $\Cycle'_{c'}=\Cycle_c$ without classifying all possible monotone curves as we did in \cref{sec:inv_mut}.

\begin{lemma}\label{lemma:tree}
  Suppose that the double braids $\br,\br'$ are related by one of the moves~\bmref{bm1}--\bmref{bm3} involving the indices in some interval $(l,r]\subset[m]$ of length $2$ or $3$. Suppose in addition that the portion $\Gubr|_{[l,r]}$ of $\Gubr$ between $\PD(\upu l)$ and $\PD(\upu r)$ has no cycles, and that $\Gubrp|_{[l,r]}$ is obtained from $\Gubr|_{[l,r]}$ via a sequence of contraction-uncontraction moves \cumv. Then, for all $c\in\Jo\setminus(l,r]$, we have $\Cycle_c=\Cycle'_c$.
\end{lemma}
\begin{proof}
By \cref{prop:gswap}, the monotone multicurves of $\Cycle_c$ inside $\PD(\upu l)$ and $\PD(\upu r)$ are determined by the combinatorics of \emph{almost positive subexpressions}, specifically, by the pairs $(\upu l, \vu\pd{c}\apd{l})$ and $(\upu r, \vu\pd{c}\apd{r})$. One can check (cf. \cref{fig:non-mut-iji}) that these pairs are invariant under applying the moves~\bmref{bm1}--\bmref{bm3} on the interval $(l,r]$. Thus, we see that the monotone multicurves of $\Cycle_c$ inside $\PD(\upu l)$ and $\PD(\upu r)$ are preserved under the double braid move. Thus, the locations where $\Cycle_c$ enters and exits $\Gubr|_{[l,r]}$ remain unchanged as we apply the move. Since $\Gubr|_{[l,r]}$ has no cycles, the restriction of $\Cycle_c$ to $\Gubr|_{[l,r]}$ is determined by these locations. It is then clear that as we apply the moves \cumv to $\Gubr|_{[l,r]}$, the relative cycle $\Cycle_c$ is preserved.
\end{proof}

In order to prove part~\eqref{item:inv2} of \cref{thm:invariance}, it suffices to find a \emph{relabeling bijection} $\Jo\xrasim\Jop$, $c\mapsto c'$, such that $\Cycle'_{c'}=\Cycle_c$ for all $c\in\Jo$.

\begin{figure}
\igrw[1.0]{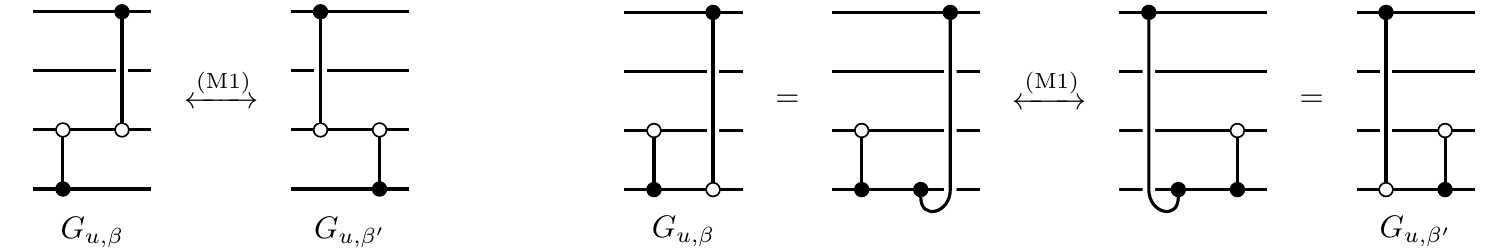}
  \caption{\label{fig:easy-moves}Representing the move~\protect\bmref{bm1} (not special) by the moves from \cref{fig:moves}.}
\end{figure}

\begin{figure}
\igrw[0.8]{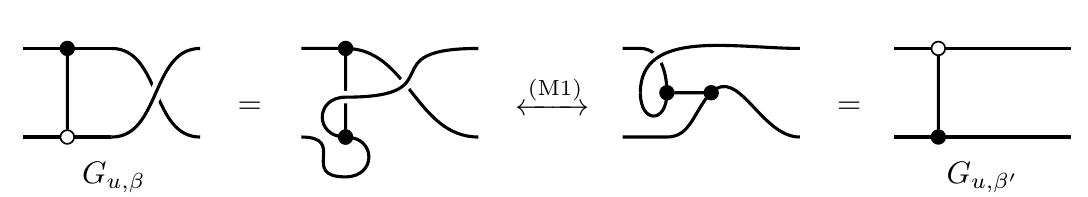}
  \caption{\label{fig:hard-moves-1}Representing the move~\protect\bmref{bm1} (special, not fully solid) by the moves from \cref{fig:moves}. See the warning in the text regarding the rightmost figure.}
\end{figure}

\begin{figure}
\igrw[1.0]{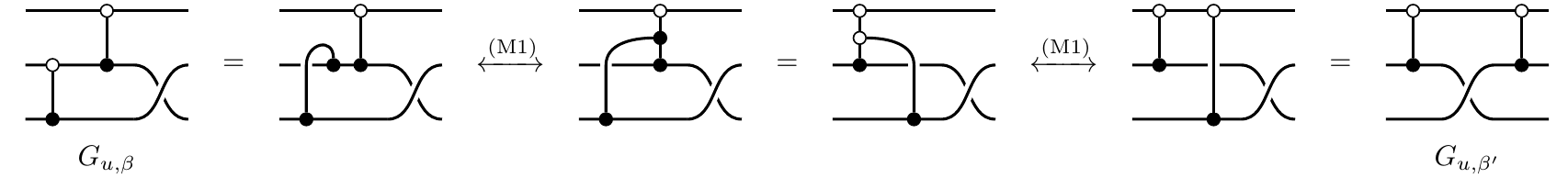}
  \caption{\label{fig:hard-moves-2}Representing the move~\protect\bmref{bm3} (not fully solid) by the moves from \cref{fig:moves}.}
\end{figure}

\subsubsection{Applying \protect\bmref{bm1} (special, not fully solid)}\label{sec:non-mut-bm1-spec}
Just as in \cref{sec:inv_mut_bm1}, we assume that we have two crossings $i_{d-1}=-i$ and $i_{d}=j$ (with $i,j\in I$) such that $\upu{d-1}s_i=s_j\upu{d-1}$, but now at least one of the two crossings is hollow. In this case, we must have that $d$ is hollow and $d-1$ is solid. Indeed, recall from~\eqref{eq:PDS_recursion} that $\upu{d-1}=\upu d \demL s_{j}$ and $\upu{d-2}=s_i\demR \upu{d-1}$. If $\upu{d-1}=\upu d s_j<\upu d$ then $d$ is hollow, and then $s_i\upu{d-1}=\upu{d-1}s_j=\upu d>\upu{d-1}$, so $d-1$ must be solid. Otherwise, we have $\upu{d-1}=\upu d<\upu d s_j$, and then $s_i\upu{d-1}=\upu{d-1}s_j>\upu{d-1}$, so both $d$ and $d-1$ are solid, a contradiction.

After we swap $-i$ and $j$ in $\br$, the crossing $i'_d=-i$ becomes hollow while $i'_{d-1}=j$ becomes solid. Thus, we have $\Jo=\Jop$. We take the relabeling bijection to be the identity map. The graphs $\Gubr$ and $\Gubrp$ are related by a single move \cumv shown in \cref{fig:hard-moves-1}, and their restrictions to $[d-2,d]$ have no cycles. By \cref{lemma:tree}, we get $\Cycle'_{c}=\Cycle_c$ inside $\Lubr=\Lubrp$ for all $c\in\Jo$.

We note that \cref{fig:hard-moves-1} can be confusing because there is a hollow crossing in $\Gubrp$, occurring to the right of the bridge, which is invisible in the red projection.
If we had drawn the blue projection instead, then the hollow crossing in $\Gubrp$ would be visible, but the one in $\Gubr$ would be invisible.

\subsubsection{Applying \protect\bmref{bm1} (not special)}
We take the relabeling bijection $c\mapsto c'$ to be the transposition of $d-1$ and $d$. If one or both of the crossings $d-1,d$ is hollow, we have $\Gubr=\Gubrp$, and we check using \cref{lemma:tree} that $\Cycle_c=\Cycle'_{c'}$ for all $c\in\Jo$. Assume now that both crossings $d-1$ and $d$ are solid.

If the bridges $\bridge_{d-1},\bridge_d$ share two strands in common then the move~\bmref{bm1} is special, a contradiction. If the bridges $\bridge_{d-1},\bridge_d$ share zero strands in common then $\Gubr=\Gubrp$, and we are done by \cref{lemma:tree}. From now on, we assume that the bridges $\bridge_{d-1},\bridge_d$ share exactly one strand in common.

There are two cases: either the start of one bridge is on the same strand as the end of the other bridge, or the start (resp., the end) of one bridge is on the same strand as the start (resp., the end) of the other bridge. In each case, $\Gubr$ and $\Gubrp$ are related by a sequence of moves \cumv shown in \cref{fig:easy-moves}. We are done by \cref{lemma:tree}.

\subsubsection{Applying \protect\bmref{bm2}}
We take the relabeling bijection $c\mapsto c'$ to be the transposition of $d-1$ and $d$. We have $\Gubr=\Gubrp$. We are done by \cref{lemma:tree}.

\subsubsection{Applying \protect\bmref{bm3} (not fully solid)}
Suppose that $i_{d-2}=j$, $i_{d-1}=j+1$, and $i_d=j$ for some $j\in I$. In $\PD(\upu{d-3})$, the dots $(i,j)$, $(i',j+1)$, $(i'',j+2)$ are located in $\prec$-increasing order, i.e., we have $i<i'<i''$. The restriction of $\PD(\upu{d})$ to the rows $j,j+1,j+2$ and columns $i,i',i''$, however, could be any permutation in $S_3$. If this permutation is the identity, then the crossings $d-2,d-1,d$ are all solid, a contradiction.

 For each of the remaining five permutations in $S_3$, the corresponding graphs $\Gubr$ and $\Gubrp$ are shown in \cref{fig:non-mut-iji}. Observe that their restrictions to $[d-3,d]$ have no cycles. In cases~(c), (d) and~(e), we have $\Gubr=\Gubrp$. Cases~(a) and~(b) are related by vertical reflection. The sequence of moves \cumv relating $\Gubr$ to $\Gubrp$ in case~(a) is shown in \cref{fig:hard-moves-2}, and case~(b) is the vertical reflection of this. In each of the five cases, there is a unique relabeling bijection $\Jo\cap \{d-2,d-1,d\}\xrasim \Jop\cap \{d-2,d-1,d\}$ preserving the monotone multicurve inside $\PD(\upu{d-3})$; it is indicated by colored curves in \cref{fig:non-mut-iji}. Extending the relabeling bijection by the identity map outside the interval $\{d-2,d-1,d\}$, we are done by \cref{lemma:tree}. This finishes the proof of \cref{thm:invariance}.

\subsection{Color-switching moves}
We introduce two more moves, which allow us to switch the color of the first or the last letter of $\br$. For $i\in I=[n-1]$, let $i^\ast:=n-i$, so that $s_i\wo=\wo s_{i^\ast}$. %
\begin{enumerate}[\normalfont(B1)]
  \setcounter{enumi}{3}
\item \label{bm4} (assuming $u=\wo$) $\br_0i\leftrightarrow \br_0(-i^\ast)$ for $i\in \pmn$ and $\br_0\in(\pmn)^{m-1}$.
\item \label{bm5} $i\br_0\leftrightarrow (-i)\br_0$ for $i\in \pmn$ and $\br_0\in(\pmn)^{m-1}$;
\end{enumerate}
\noindent Note that the assumption $u=\wo$ in \bmref{bm4} is not restrictive in view of \cref{rmk:u=w0}.

For an ice quiver $\Qice$, its \emph{mutable part} is the induced subquiver of $\Qice$ with vertex set $\Vmut$, where $\Vmut$ is the set of mutable vertices of $\Qice$.
\begin{proposition} \ \label{prop:bm5_bm4}
\begin{enumerate}
\item  Under the move~\bmref{bm4}, $\Qubr$, $\Gubr$, $\Subr$, and $(\Cycle_c)_{c\in\Jo}$ are unchanged.
\item  Under the move~\bmref{bm5}, the mutable part $\Qubrmut$ of $\Qubr$ is unchanged.
\end{enumerate}
\end{proposition}

\begin{proof}
  First, let us consider the move~\bmref{bm4}. Since $u=\wo$, the index $m$ must be hollow. There are no relative cycles passing through the corresponding strands of $\Gubr$, since they only pass through the graph $\Gubar$ from \cref{sec:3D_plabic}. This verifies the first part of the proposition.

  Now, consider the move~\bmref{bm5}. Without loss of generality, assume that $i:=i_1>0$. Since we only care about the mutable part of $\Qubr$, we may disregard all frozen indices, i.e., indices $c\in\Jo$ such that the monotone multicurve $\bgamma\pu{c,0}$ is nonempty. The index $1$ is either hollow or frozen. For any mutable index $d\in\Jo$, if $\bgamma\pu{d,1}$ is empty then the intersection number $\<\Cycle_c,\Cycle_d\>$ clearly stays unchanged under~\bmref{bm5} for all $c\in\Jo$. Let $c\in\Jo$ be a mutable index such that the monotone multicurve $\bgamma\pu{c,1}$ is nonempty (but $\bgamma\pu{c,0}$ is empty). Then the index $1$ must be solid (and therefore frozen). Using a contraction-uncontraction move similar to the one in \cref{fig:hard-moves-1}, we see that the marked surface $\Subrp$ (where $\brp$ is obtained by applying~\bmref{bm5} to $\br$) is obtained from $\Subr$ by swapping the labels of the marked points $i$ and $i+1$.
 Since this swap does not affect the part of $\Subr$ to the right of $\PD(\upu0)$, we see that for any two mutable relative cycles $\Cycle_c,\Cycle_d$ that pass through the bridge $\bridge_1$, their intersection number is unchanged under~\bmref{bm5}. This verifies the second part of the proposition.
\end{proof}

\section{Cluster algebras associated to 3D plabic graphs}\label{sec:cluster+3D}
The goal of this section is to show that the cluster algebra defined by $\Qubr$ is \emph{locally acyclic}~\cite{Mul} and \emph{really full rank}~\cite{LS}.

\subsection{Background on cluster algebras}\label{sec:cluster}
We briefly recall the definition of a cluster algebra and related concepts.

Recall from \cref{sec:intro:quiver} the definition of an ice quiver $\Qice$, with vertex set $\Vice=\Vmut \sqcup \Vfro$. We say that $\Qice$ is \emph{isolated} if its mutable part has no arrows. For a set $S\subset \Vmut$, let $\Qice[S]$ denote the ice quiver obtained from $\Qice$ by further declaring all vertices in $S$ to be frozen. We write $\Qice-S$ for the ice quiver obtained from $\Qice$ by removing the vertices in $S$.

The associated \emph{(extended) exchange matrix} $\BQice=(b_{v, w})_{v \in \Vice, w \in \Vmut}$ is defined by 
\begin{equation*}
	b_{v,w}=\#\{\text{arrows $v\to w$ in $\Qice$}\} - \#\{\text{arrows $w\to v$ in $\Qice$}\}
\end{equation*}
\begin{definition}\label{dfn:really_fr}
We say that $\Qice$ is \emph{really full rank} if the rows of its exchange matrix span $\Z^{|\Vmut|}$ over~$\Z$.
\end{definition}

Let $\Fcal \cong \CC(t_1, \dots, t_r)$ be isomorphic to the field of rational functions in $r$ algebraically independent variables. A \emph{seed} in $\Fcal$ is a pair $\Sigma =(\x, \Qice)$ where $\Qice$ is an ice quiver with $r$ vertices $\Vice=\Vmut \sqcup \Vfro$ and $\x=\{x_v\}_{v \in \Vice}$ is a transcendence basis of $\Fcal$. The tuple $\x$ is the \emph{cluster}, its elements are \emph{cluster variables} and $x_v$ is \emph{mutable} if $v \in \Vmut$ and \emph{frozen} otherwise. 

Given a seed $(\x, \Qice)$ in $\Fcal$ and $d \in \Vmut$, one can \emph{mutate} in direction $d$ to obtain a new seed $\mu_d(\x, \Qice)=(\x', \mu_d(\Qice))$. The quiver of the mutated seed is the mutation of $\Qice$ in direction $d$ (see \cref{defn:quiverMutation}). The cluster $\x'=\{x'_v\}_{v \in \Vice}$ of the new seed satisfies $x'_v = x_v$ for $v \neq d$ and 
$$
x'_d = \frac{\prod_{v \to d} x_v + \prod_{d \to v} x_v}{x_d} \in \Fcal.
$$
By repeatedly mutating, we generate (possibly infinitely) many seeds and cluster variables. 
We denote by $\A(\x, \Qice)$ the \emph{cluster algebra} associated to the seed $(\x, \Qice)$. This is the $\C$-subalgebra of $\Fcal$ generated by all cluster variables and the inverses of frozen variables. We call $\A(\x, \Qice)$ \emph{isolated} (resp., \emph{really full rank}) if $\Qice$ is.

We denote by $\U(\x, \Qice)$ the \emph{upper cluster algebra} associated to the seed $(\x, \Qice)$; see~\cite{BFZ}. It is given by
\[\U(\x, \Qice)= \bigcap_{(\x', \Qice')} \CC[(x_v')^{\pm1}: v \in \Vice] \subset \Fcal,\]
where the intersection is taken over all seeds $(\x', \Qice')$ which can be obtained from $(\x, \Qice)$ by a sequence of mutations.

Note that we have isomorphisms $\A(\x, \Qice)\cong \A(\mathbf{y}, \Qice)$ and $\U(\x, \Qice)\cong \U(\mathbf{y}, \Qice)$ for any two clusters $\x, \mathbf{y}$.
 Thus, we may occasionally write $\A(\Qice)$ and $\U(\Qice)$ if the particular choice of initial cluster does not matter.

In general, we have $\A(\x, \Qice) \subseteq \U(\x, \Qice)$, and the containment may be strict. However, we have equality if $\A(\x, \Qice)$ is \emph{locally acyclic}, a property introduced by Muller \cite{Mul}. We need a few definitions before defining local acyclicity; we follow the presentation of \cite{MullerLA2}. 

Let $\A(\x, \Qice)$ be a cluster algebra and let $S \subset \Vmut$. Then the cluster algebra $\A(\x, \Qice[S])$ obtained by freezing $S$ is a \emph{cluster localization} of $\A(\x, \Qice)$ if
\begin{equation}\label{eq:clusterloc} \A(\x, \Qice[S])= \A(\x, \Qice)[x_v^{-1}: v \in S].\end{equation}
(In general, the left-hand side is contained in the right.)  Lemma 4.3 of \cite{Mul} states that if $\A(\x, \Qice[S])=\U(x, \Qice[S])$, then \eqref{eq:clusterloc} automatically holds.

A collection of cluster localizations $\{\A_i\}$ of a cluster algebra $\A$ is a \emph{cover} if for every prime ideal $P$ of $\A$, there is some $\A_i$ such that $\A_i P \subsetneq \A_i$.
A cluster algebra is \emph{locally acyclic} if it has a cover by isolated cluster algebras. We say that $\Qice$ is \emph{locally acyclic} if $\A(\Qice)$ is.
The motivation for this definition is that, in this case, the collection of varieties $\text{Spec} \A_i$ forms an open cover of $\text{Spec} \A$.

We will use the following facts about locally acyclic cluster algebras.

\begin{proposition}[{\cite[Proposition 3.10]{Mul}}]\label{prop:locAcycMutableEnough}
	Let $\Qice$ be an ice quiver and let $Q$ denote its mutable part, as usual. Then $\A(\Qice)$ is locally acyclic if and only if $\A(Q)$ is locally acyclic.
\end{proposition}

\begin{theorem}[{\cite[Theorem 4.1]{Mul}}]
	If $\A(\x, \Qice)$ is locally acyclic, then $\A(\x, \Qice)= \U(\x, \Qice)$.
\end{theorem}

One frequently studied class of locally acyclic quivers are \emph{Louise} quivers \cite{LS}. We focus instead on \emph{\sinkrec} quivers, defined below, which may not be Louise but are locally acyclic. In the next subsection, we show that our quivers of interest, $\Qice_{u,\br}$, are \sinkrec (see \cref{thm:sinkrec}).

Let $Q$ be a quiver with no frozen vertices. A vertex $s$ of $Q$ is a \emph{sink} if it has no outgoing arrows. We let $\NtQ$ be the set of vertices of $Q$ that have an arrow pointing to $s$. The following notion is analogous to the class of \emph{leaf-recurrent quivers} from~\cite[Section~5.4]{GL_plabic}; see also~\cite[Remark~5.14]{GL_plabic}.
\begin{definition}\label{def:sinkrec}
	The class of \emph{\sinkrec quivers} is defined recursively as follows.
	\begin{itemize}
		\item Any isolated quiver $Q$ is \sinkrec.
		\item Any quiver that is mutation equivalent to a \sinkrec quiver is \sinkrec.
		\item Suppose that a quiver $Q$ has a sink vertex $s$ such that the quivers $Q-\{s\}$ and $Q-(\NtQ\cup\{s\})$ are \sinkrec. Then $Q$ is \sinkrec.
	\end{itemize}
\end{definition}
The above definition refers to mutable quivers (without frozen vertices). We say that an ice quiver $\Qice$ is \emph{\sinkrec} if its mutable part $Q$ is \sinkrec. 

\begin{proposition}\label{prop:sinkrecImpliesLocAcyc}
	If $\Qice$ is a \sinkrec quiver, then $\A(\x, \Qice)$ is locally acyclic.
\end{proposition}

\begin{proof}
If $\Qice$ is isolated, we are done. Otherwise, since local acyclicity is a property of the cluster algebra rather than the quiver, we may assume that $Q$ has a sink $s$ so that $Q-\{s\}$ and $Q-(\NtQ\cup\{s\})$ are \sinkrec (by mutating if necessary).  Consider the freezings $\A_1=\A(\x, Q[s])$ and $\A_2=\A(\x, Q[\NtQ])$. The mutable parts of $Q[s]$ and $Q[\NtQ]$ are $Q-\{s\}$ and $Q-\NtQ$, respectively, which are both \sinkrec, the former by definition, and the latter because $Q-\NtQ$ differs from the \sinkrec quiver $Q-(\NtQ \cup \{t\})$ by an isolated vertex. So $\A_1$ and $\A_2$ are locally acyclic by induction, and thus are cluster localizations.

Let $t$ be any neighbor of $s$.  Then the pair $(s,t)$ forms a \emph{covering pair} in the sense of \cite{Mul}, and by \cite[Lemma 5.3]{Mul} it follows that $x_s$ and $x_t$ cannot simultaneously vanish on $\A(\Qice)$.  It follows that $\A_1$ and $\A_2$ cover $\A$.  	The union of the covers of $\A_1$ and $\A_2$ by isolated cluster algebras is a cover of $\A(\x, Q)$ by isolated cluster algebras. %
\end{proof}

\subsection{Local acyclicity}
Our goal is to show the following result, which by \cref{prop:sinkrecImpliesLocAcyc} implies that the cluster algebra $\A(\x, \Qice)$ is locally acyclic.

\begin{theorem}\label{thm:sinkrec}
For any $u\leq \br$, the ice quiver $\Qubr$ is \sinkrec.
\end{theorem}

\begin{proof}
We proceed by induction on the number $\ell(\br)-\ell(u)$ of vertices of $\Qubr$. For the base case $\ell(\br)=\ell(u)$, we see that $\br$ is a reduced word for $u$ and all crossings are hollow, and thus $\Qubr$ is an empty quiver. Let us now assume that $\ell(\br)>\ell(u)$.

By \cref{rmk:u=w0}, it suffices to consider the case $u=\wo$. Applying the moves~\bmref{bm1} and~\bmref{bm4}, we may assume that all indices in $\br$ are positive. Then, assuming $\br=i\br_0$, we can 
 transform $\br$ into
\begin{equation}\label{eq:rotn}
  \br=i\br_0\xrightarrow{\text{\bmref{bm5}}} (-i)\br_0 \xrightarrow{\text{\bmref{bm1}}} \cdots \xrightarrow{\text{\bmref{bm1}}} \br_0(-i) \xrightarrow{\text{\bmref{bm4}}} \br_0 i^\ast.
\end{equation}
We call the operation~\eqref{eq:rotn} \emph{the conjugation move}.

Since $u=\wo$ and $\ell(\br)>\ell(u)$, the braid word $\br$ must be non-reduced. Then, after applying the moves~\bmref{bm2}--\bmref{bm3}, we may transform $\br$ into a word of the form $\br_1ii\br_2$, for two positive braid words $\br_1,\br_2$. Applying conjugation moves~\eqref{eq:rotn}, we may further transform it into the word $ii\br_2\br_1^\ast$, where $\br_1^\ast$ is obtained from $\br_1$ by applying the map $i\mapsto i^\ast$ to each index. Applying one more move~\bmref{bm5}, we obtain the word $(-i)i\br_2\br_1^\ast$ which we still denote by $\br$. 

By the same argument as in \cref{sec:non-mut-bm1-spec}, we get that the crossing $i_1=-i$ is solid. Let $\brp:=i\br_2\br_1^\ast$ be obtained by omitting $-i$ from $\br$. Since $i_1$ is solid, we have $u\leq\brp$, and by the induction hypothesis, the quiver $\Qubrpmut$ is \sinkrec.

Suppose first that the crossing $i_2=i$ is hollow. It is easy to check from the propagation rules in \cref{fig:propag} that no mutable relative cycle passes through the bridge $\bridge_1$. In particular, we see that $\Qubrmut=\Qubrpmut$, which we know is \sinkrec by induction.

Suppose now that both crossings $i_1=-i$ and $i_2=i$ are solid. The graph $\Gubr$ has a square $psrq$ formed by the two corresponding bridges as in \figref{fig:sq-surf}(left). The index $d:=2$ is mutable and the relative cycle $\Cycle_d$ passes through the vertices of the square in the counterclockwise direction. Any other mutable relative cycle $\Cycle_c$ satisfying $\<\Cycle_c,\Cycle_d\>\neq0$ must have signature $sp$ (i.e., pass through the bridge $\bridge_2$ in the direction opposite to $\Cycle_d$). In particular, we see that $\<\Cycle_c,\Cycle_d\>=1$ whenever $c$ is mutable and $\<\Cycle_c,\Cycle_d\>\neq0$. It follows that $d$ is a sink in $\Qubrmut$.

Let $\brpp:=\br_2\br_1^\ast$ be obtained by omitting both $-i$ and $i$ from $\br$. We still have $u\leq \brpp$, and by the induction hypothesis, the quiver $\Qubrppmut$ is \sinkrec. We have $\Qubrpmut=\Qubrmut-\{d\}$. The quiver $\Qubrppmut$ is obtained from $\Qubrmut$ by deleting $d$ together with all vertices that have an arrow pointing to $d$, since the corresponding cycles become frozen in $\Gubrpp$. The result follows.
\end{proof}

\begin{remark}
Similar reasoning has been recently used in~\cite[Proposition~7.9]{GL_plabic} to study \emph{plabic fences}, which are recovered as special cases of our construction when $u=\id$. %
\end{remark}

From \cref{thm:sinkrec} and \cref{prop:sinkrecImpliesLocAcyc}, we have the following immediate corollary.

\begin{corollary}\label{cor:locAcyclic}
	For any $u \leq \br$, the quiver $\Qubr$ is locally acyclic.
\end{corollary}

\subsection{Really full rank}\label{sec:full_rank}
Our goal is to show the following result; cf. \cref{dfn:really_fr}.
\begin{theorem}\label{prop:full_rank}
The quiver $\Qubr$ is really full rank.
\end{theorem}

We recall what we have described before: We defined $\Lubr = H_1(\Subr, \MP)$, $\Lubrst=H_1(\Subr\setminus\MP,\partial\Subr\setminus\MP)$ and $\Lmut = H_1(\Subr)$. The lattice $\Lmut$ is a sublattice of both $\Lubr$ and $\Lubrst$, and non-canonically splits off as a direct summand of either (\cref{lem:summand}).
We have a perfect pairing $\langle \cdot, \cdot \rangle: \Lubr \times \Lubrst \to \mathbb{Z}$.
For $c \in \Jo$, we have constructed relative cycles $\Cycle_c$ in $\Lubr$; if $c \in \Jmut$ then $\Cycle_c \in \Lmut$.
The matrix $\Bice$ has entries $\langle \Cycle_c, \Cycle_d \rangle$ for $c \in \Jo$ and $d \in \Jmut$. 
We want to show that the rows of this matrix (indexed by $\Jo$) span $\Z^{\Jmut}$ over~$\Z$. 

\begin{lemma} \label{lemma:bridge_basis}
The relative cycles $(\Cycle_c)_{c \in \Jo}$ are a $\Z$-basis of $\Lubr$.
\end{lemma}

\begin{proof}
Since $\Subr$ deformation retracts onto $\Gubr$, we see that $\Lubr\cong \Z^{|\Jo|}$ is a free abelian group whose rank is the number of bridges in $\Gubr$. 
Thus, it is enough to show that the cycles  $(\Cycle_c)_{c \in \Jo}$ span $\Lubr$ over $\Z$.

Let $\gamma = \sum a_e e$ be any relative cycle in $\Lubr$. Removing all bridges from $\Gubr$ leaves $n$ disjoint paths so, if $\gamma$ is nonzero, there must be a bridge $c$ with $a_c \neq 0$; let $c$ be the rightmost such bridge. Replacing $\gamma$ by $\gamma - a_c \Cycle_c$, we replace $\gamma$ by a cycle supported on bridges to the left of $c$. Continuing in this manner, we eventually obtain a cycle supported on no bridges, which must be $0$.
\end{proof}

\begin{proof}[Proof of \cref{prop:full_rank}]
If the rows of the matrix $\Bice$ spanned a proper sublattice of $\Z^{\Jmut}$, then there would be a prime $p$, and integers $\ba := (a_d)_{d \in \Jmut}$, not all $0 \bmod p$, such that the element $\Cx:=\sum_{d\in \Jmut} a_d C_d \in \Lmut\subset\Lubrst$ satisfies 
$\langle C_c, \Cx \rangle \equiv 0 \bmod p$ for all $c \in \Jo$. 
So, using \cref{lemma:bridge_basis} and the fact that the pairing $\<\cdot,\cdot\>$ is perfect, we get $\Cx \in p \Lubrst$. But $\Lmut$ is a direct summand of $\Lubrst$ by \cref{lem:summand}, and thus $\Cx \in p \Lmut$. Using the \cref{lemma:bridge_basis} again to note that the $C_d$, $d\in\Jmut$, are linearly independent modulo $p$, this shows that all the $a_d$ are $0 \bmod p$, contrary to assumption.
\end{proof}

We also take the opportunity to verify a claim made earlier in \cref{rmk:Z_basis_will_show}.
\begin{lemma} \label{lemma:mutable_completes_to_basis}
The cycles $(\Cycle_c)_{c \in \Jmut}$ can be extended to a $\Z$-basis of $\Lmut$.
\end{lemma}

\begin{proof}
By \cref{lemma:bridge_basis}, the cycles $(\Cycle_c)_{c \in \Jmut}$ are linearly independent, so it remains to check that $\Lmut/\Z \cdot \{ \Cycle_c : c \in \Jmut \}$ is torsion free. 
This embeds in the larger abelian group $\Lubr/\Z \cdot \{ \Cycle_c : c \in \Jmut \}$ which, using \cref{lemma:bridge_basis}, we can identify with $\Z^{\Jo}/\Z^{\Jmut} = \Z^{\Jfro}$.
\end{proof}

\section{Double braid varieties}\label{sec:dblBraidVar}

\subsection{Notation}\label{sec:doublebraid_Notation}
Let $G=\SL_n$. Recall that $I:=[n-1]$ indexes the simple transpositions of the Weyl group $W=S_n$, and that for $i \in I$, we denote $i^*:= n-i$.  For $i \in I$, we set $(-i)^* := -i^*$.  Let $B_+$ and $B_-$ denote the subgroup of upper triangular and lower triangular matrices, respectively, and let $U_+$ and $U_-$ denote the respective unipotent subgroups. We use $H$ to denote the torus of diagonal matrices.

For $i\in I$, let $\phi_i:\SL_2 \to G$ denote the homomorphism where $\phi_i(g)$ is the matrix which has $g$ as the $2 \times 2$ submatrix on rows and columns $i, i+1$ and otherwise agrees with the identity matrix.

We use $\phi_i$ to lift $S_n$ to $G$. Let 
\[ \ds_i:= \phi_i \begin{pmatrix}
	0 & -1\\
	1 & 0
\end{pmatrix}.
\]
For $w\in S_n$, we define $\dot{w}:= \ds_{i_1} \dots \ds_{i_n}$ where $s_{i_1} \dots s_{i_n}$ is a reduced expression for $w$. The map $w \mapsto \dot{w}$ is not a homomorphism, but if $\ell(vw)=\ell(v)+ \ell(w)$, then $(vw)^{\bigcdot}=\dot{v}\dot{w}$. In particular, $\dot{w}$ does not depend on the choice of reduced expression. Explicitly, 
\[(\dot{w})_{i,j}= \begin{cases}
	0 &\text{ if } i \neq w(j),\\
	(-1)^{|\{a<j: w(a)>w(j)\}|} &\text{ if } i = w(j).
\end{cases}\]
We omit the dot when the choice of the signs in the permutation matrix of $w$ does not matter: e.g., we write $B_+wB_+$ and $\H w$ in place of $B_+\dw B_+$ and $\H\dw$.

We will also need the %
generators
\begin{equation}\label{eq:gen_dfn}
x_i(t):=\phi_i \begin{pmatrix}
	1& t\\
	0& 1
\end{pmatrix}, \quad 
y_i(t):=\phi_i \begin{pmatrix}
	1& 0\\
	t& 1
\end{pmatrix}, \quad \alphacheck_i(t):=\phi_i \begin{pmatrix}
	t& 0\\
	0& 1/t
      \end{pmatrix},
\end{equation}
 and the \emph{braid matrices}
\[z_i(t):=\phi_i \begin{pmatrix}
	t & -1\\
	1 & 0
\end{pmatrix}
=
x_i(t) \ds_i = \ds_i y_i(-t);\]
\[\z_i(t):=\phi_i \begin{pmatrix}
	t & 1\\
	-1 & 0
\end{pmatrix}
=
x_i(-t) \ds^{-1}_i = \ds^{-1}_i y_i(t).\]

\subsection{Weighted flags}
A \emph{weighted flag} is an element $F = gU_+ \in G/U_+$.  Associated to a weighted flag $F$ is the flag $\pi(F)=gB_+$, the image of $F$ in $G/B_+$.  

\begin{definition}
Let $w \in S_n$. Two weighted flags $(F,F')$ are called \emph{(weakly) $w$-related} if there exist $g \in G$ and $h\in H$ such that $(gF,gF') = (U_+, h\dot{w}U_+) \in G/U_+ \times G/U_+$. Equivalently, $F=g_1 U_+$ and $F'=g_2 U_+$ are weakly $w$-related if and only if $g_1^{-1}g_2 \in B_+ \dot{w} B_+$. We write this as $F \Rwrel{w} F'$.
	
	Two weighted flags $(F,F')$ are called \emph{strictly $w$-related} if there exists $g \in G$ such that $(gF,gF') = (U_+, \dot{w}U_+) \in G/U_+ \times G/U_+$.  Equivalently, $F=g_1 U_+$ and $F'=g_2 U_+$ are strictly $w$-related if and only if $g_1^{-1}g_2 \in U_+ \dot{w} U_+$. We write this as $F \Rrel{w} F'$. 
\end{definition}
Let $F,F'$ be two weighted flags and let $B,B'$ be their images in $G/B_+$.  Then 
\begin{equation}\label{eq:relrel}
\mbox{$F \Rrel{w} F'$ implies $F \Rwrel{w} F'$, which in turn implies $B \Rrel{w} B'$ (see \cref{sec:braid}).}
\end{equation}

We collect some elementary facts about relative position.
\begin{lemma}\label{lem:positionFacts} %
 Let $F, F', F''$ be weighted flags. Suppose $v, w \in S_n$ and $\ell(vw)=\ell(v) + \ell(w)$.
	\begin{enumerate}
		\item $F \Rrel{\id} F'$ if and only if $F=F'$.
		\item If $F \Rrel{v} F' \Rrel{w} F''$, then $F \Rrel{vw}F''$.
		\item If $F \Rrel{vw}F''$, then there exists a unique $F'$ such that $F \Rrel{v} F' \Rrel{w} F''$. If $F \Rwrel{vw} F''$, there exist unique $F_1', F_2' \in G/U_+$ such that $F \Rrel{v} F_1' \Rwrel{w} F''$ and $F \Rwrel{v} F_2' \Rrel{w} F''$.
	\end{enumerate}
\end{lemma}

\begin{proof}
	Parts (1), (2) and the first sentence of (3) can be found in \cite[Appendix]{ShWe}. We show the existence of $F_1'$ in (3); the argument for $F_2'$ is similar. The images $\pi(F)$ and $\pi(F'')$ in $G/B_+$ uniquely determine a flag $L \in G/B_+$ satisfying $\pi(F) \Rrel{v} L \Rrel{w} \pi(F'')$, and $L$ has a unique lift $F_1'$ to $G/U_+$ such that $F \Rrel{v} F_1'$.
\end{proof}

\begin{lemma}\label{lem:paramSingleStep}%
	Suppose $F \Rrel{s_i} F'$ and say $F=gU_+$. Then there exists a unique $t \in \CC$ such that $F'=gz_i(t) U_+$. Similarly, if $F'=g'U_+$, there exists a unique $t' \in \CC$ such that $F=g' \z_i(t')U_+$.
      \end{lemma}
      \begin{remark}\label{rmk:t_depends}
The parameters $t,t'$ in \cref{lem:paramSingleStep} depend on the choices of the representative matrices $g,g'$.
      \end{remark}

\subsection{Double braid variety}

For a pair $(u,\br) \in S_n \times \DRW$ with $u \leq \br$, let $\bigBR_{u,\br}$ denote the space of tuples of weighted flags satisfying %
\begin{equation}\label{eq:braidDiagram}
	\begin{tikzcd}
		X_0& \arrow[l,"\textcolor{red}{s_{i_1}^+}"'] X_1& \arrow[l,"\textcolor{red}{s_{i_2}^+}"'] \cdots& \arrow[l,"\textcolor{red}{s_{i_m}^+}"'] 
		X_m \\
		Y_0 \arrow[r,"\textcolor{blue}{s_{i_1^\ast}^-}"'] \arrow[u, Rightarrow, "\wo \cdot \id"']& Y_1 \arrow[r,"\textcolor{blue}{s_{i_2^\ast}^-}"'] & \cdots \arrow[r,"\textcolor{blue}{s_{i_m^\ast}^-}"'] & Y_m \arrow[u, rightarrow, "\wo \cdot u"']
	\end{tikzcd}
\end{equation}
Also define $\BigBR_{u, \br}$ by omitting the condition that $(X_0,Y_0)$ are weakly $\wo$-related.  Then $\BigBR_{u, \br}$ is a $\C^l$-bundle over $G/U_+$ (where $l=m+\ell(\wo u)$), and $\bigBR_{u, \br}$ is an open subset of $\BigBR_{u, \br}$. We denote points in $\bigBR_{\ubr}$ and $\BigBR_{\ubr}$ by $(\Xbul, \Ybul)$. %

\begin{remark}\label{rmk:braidVarIntuition}
  Intuitively, a point in $\BigBR_{u, \br}$ is a walk in $G/U_+ \times G/U_+$ from a pair of weighted flags $X_m \Lrel{\wo u} Y_m$ to a pair of weighted flags $X_0 \Lwrel{\wo} Y_0$.
  The ``direction" of step $c$ is dictated by the crossing $i_c$ of $\beta$.
  A red crossing $i$ in $\beta$ means the step changes the $i$-dimensional subspace of the first flag in the pair. A blue crossing $j$ means the step changes the $(n-|j|)$-dimensional subspace of the second flag. Given an arbitrary pair $(X_m,Y_m)$ satisfying $X_m \Lrel{\wo u} Y_m$,
  we can parametrize $(X_{c-1},Y_{c-1})$ iteratively for $c=m,m-1,\dots,1$ using \cref{lem:paramSingleStep}: assuming $(X_c,Y_c)=(g_cU_+,g'_cU_+)$, we set
  \begin{equation}\label{eq:propagate_left}
    (X_{c-1},Y_{c-1}):=
    \begin{cases}
      (g_cz_{i_c}(t_c)U_+,g'_cU_+), &\text{if $i_c>0$,}\\
      (g_cU_+,g'_c\z_{|i_c|^*}(t_c)U_+), &\text{if $i_c<0$,}\\
    \end{cases}
  \end{equation}
 for arbitrary parameters $\bt:=(t_1,t_2,\dots,t_m)\in\C^m$. For $(\Xbul, \Ybul)$ to be a point in $\bigBR_{u, \br}$, we require further that $X_0 \Lwrel{\wo} Y_0$, which is an extra open condition on the parameters $(\bt,X_m,Y_m)$.
\end{remark}

The group $G$ acts on $\bigBR_{u, \br}$ and $\BigBR_{u,\br}$ by acting simultaneously on all the weighted flags by left multiplication.
\begin{lemma}
	The action of $G$ on $\bigBR_{u, \br}$ is free.
\end{lemma}
\begin{proof}
	$G$ acts freely on pairs $(X_0,Y_0)$ which are weakly $\wo$-related. Indeed, it suffices to consider the case $(X_0, Y_0)=(U_+, h \dot{w}_0 U_+)$. The stabilizer of the pair $(U_+, h \dot{w}_0 U_+)$ is $U_+ \cap h\dot{w}_0 U_+ \dot{w}_0^{-1}h^{-1}$, which is contained in $U_+ \cap B_-=\{1\}$.
\end{proof}

\begin{definition}
	The \emph{double braid variety} $\BR_{u,\br} = G\backslash \bigBR_{u, \br}$ is the quotient of $\bigBR_{u, \br}$ by $G$.
\end{definition}

Our notation is consistent with \cref{sec:Rich} due to the following.

\begin{proposition}\label{prop:braidRich}
Let $\beta\in I^m$ be a positive braid word. Then the double braid variety $\BR_{\ubr}$ is isomorphic to the braid Richardson variety defined in \eqref{eq:braidRich}.
\end{proposition}
\begin{proof}
Write $\pi$ for the natural map $G/U_+ \to G/B_+$.
Let $(X_\bullet,Y_\bullet) \in \BR_{\ubr}$.  The $G$-action can be gauge-fixed by assuming $(X_0, Y_0)=(h U_+, \dot{w}_0 U_+)$, where $h \in H$.    With this gauge-fix, we obtain a map $\pi^{m}:\BR_{\ubr}\to (G/B_+)^m$ sending $(X_\bullet,Y_\bullet)\mapsto (\pi(X_1),\pi(X_2),\dots,\pi(X_m))$.  We now show that this map is an isomorphism onto the braid Richardson variety of \eqref{eq:braidRich}.  Since $\pi(X_0)=B_+$ and $\pi(Y_0) = B_-$, \eqref{eq:relrel} shows that $\pi^m$ maps $\BR_{\ubr}$ into the space \eqref{eq:braidRich}.  

For the inverse of $\pi^m$: Suppose we have two flags $B,B'\in G/B_+$ satisfying $B \xrightarrow{s_i} B'$ and a lift $F \in G/U_+$ of $B$ is given. Then there is a unique $F' \in G/U_+$ that lifts $B'$ and satisfies $F \xrightarrow{s_i}F'$. Now, given a point $B_{\bullet}$ in the space \eqref{eq:braidRich}, we set $Y_0 = \cdots = Y_m= \dot{\wo} U_+$ and set $X_m$ to be the unique lift of $B_m$ such that $Y_m \Rrel{\wo u} X_m$. Lifting $B_{m-1}, B_{m-2}, \dots, B_1$ from right to left (using the observation at the beginning of the paragraph), we obtain a unique point in $\BR_{\ubr}$ which $\pi^m$ maps to $B_{\bullet}$.
\end{proof}

\begin{remark}
When both blue and red subwords of $\br$ are reduced, the same argument as in the proof of \cref{prop:braidRich} shows that our double braid varieties are isomorphic to the varieties considered by Webster and Yakimov \cite{WY}.  When $u$ is the identity, $\BR_{\id,\br}$ is isomorphic to the double Bott--Samelson cells of Shen and Weng \cite{ShWe}.
\end{remark}

\begin{remark}
The isomorphism of \cref{prop:braidRich} is not compatible with total positivity, in the sense that it does not send the totally positive part $\Rtp_u^w$ of $\Rich_u^w$ (introduced by Lusztig~\cite{Lus2}) to the subset of the braid variety $\BR_{u,\br}$ where all cluster variables take positive value. (Here $\br$ is a reduced word for $w$.) In order to fix that, one must compose this isomorphism with an automorphism of $\Rich_u^w$ discussed in \cref{sec:comparison}.
\end{remark}

\subsection{Geometry of double braid varieties}

In this section, we discuss some straightforward geometric results on double braid varieties. As preparation, we first show that it is enough to consider the case when $u =\wo$.

\begin{notation}\label{notn:w0}
	For any object defined in terms of a pair $(u, \br)$, we suppress the permutation $u$ from the notation when $u=\wo$. For example, we write $\BR_{\br}:=\BR_{\wo ,\br}$.
\end{notation}

\begin{lemma}\label{lem:w0enough}
	Given $u \leq \br$, let $\gamma^+\in I^{\ell(\wo u)}$ be a reduced word for $\wo  u$ using red letters and $\gamma^-\in (-I)^{\ell(u\wo )}$ be a reduced word for $u \wo$ in blue letters. Then
	\begin{equation*}%
		\bigBR_{u, \br} \cong \bigBR_{\br\gamma^+} \cong \bigBR_{\br\gamma^-} \quad\text{and}\quad \BR_{u, \br} \cong \BR_{\br\gamma^+} \cong \BR_{\br\gamma^-}.
	\end{equation*}
\end{lemma}

\begin{proof}
	The isomorphisms
	\[ \bigBR_{\br\gamma^+} \to \bigBR_{u, \br} \qquad \bigBR_{\br\gamma^-} \to \bigBR_{u, \br} \]
	are given by truncating $(\Xbul,\Ybul)$ after the pair $(X_{\ell(\br)}, Y_{\ell(\br)})$. The choice of $\gamma^+$ and $\gamma^-$ together with \lemmaref{lem:positionFacts}(2) ensures that the maps are well defined. Injectivity and surjectivity follow from \lemmaref{lem:positionFacts}(3). The isomorphisms are $G$-equivariant and so descend to the quotient.
\end{proof}

\begin{proposition}
	For $u\leq \br$, the double braid variety $\BR_{\ubr}$ is a smooth, affine, irreducible complex algebraic variety of dimension equal to $d(u,\br) := \ell(\br) -\ell(u)$.
\end{proposition}

\begin{proof}
Using \cref{lem:w0enough}, we may assume $u=\wo$. Consider the space of tuples of weighted flags satisfying
		\begin{equation}\label{eq:Ugauge}
		\begin{tikzcd}
			U_+& \arrow[l,"\textcolor{red}{s_{i_1}^+}"'] X_1& \arrow[l,"\textcolor{red}{s_{i_2}^+}"'] \cdots& \arrow[l,"\textcolor{red}{s_{i_m}^+}"'] 
			X_m \\
			Y_0 \arrow[r,"\textcolor{blue}{s_{i_1^\ast}^-}"'] %
			& Y_1 \arrow[r,"\textcolor{blue}{s_{i_2^\ast}^-}"'] & \cdots \arrow[r,"\textcolor{blue}{s_{i_m^\ast}^-}"'] & Y_m \arrow[u, rightarrow, "\id"']
		\end{tikzcd} 
	\end{equation}
This space is an iterated $\C^1$-bundle and thus affine. Imposing the condition that $U_+$ and $Y_0$ are weakly $\wo$-related (that is, $Y_0\in B_+ \dot{w}_0 B_+=B_+ \dot{w}_0 U_+$) cuts out a nonempty smooth affine open subset $V$ of the iterated $\C^1$-bundle. The braid variety $\BR_{\br}$ is the quotient of $V$ by the diagonal action of $U_+=\text{Stab}_G(U_+)$. The group $U_+$ acts freely on $B_+ \dot{w}_0 U_+$ and thus acts freely on $V$.  It follows that the quotient $\BR_{\br}$ is also smooth and affine; it is also clearly irreducible.

For the dimension, note that 
\[\dim( \bigBR_{\br})= \dim (G/U_+) + \ell(\br)= \dim(G)- \ell(\wo ) + \ell(\br)\]
so $\dim(\BR_{\br})=\dim( \bigBR_{\br})-\dim(G) = \ell(\br)-\ell(\wo )$.
\end{proof}

Braid moves induce natural isomorphisms between braid varieties, which we give below. In \cref{sec:movesOnSeed}, we will analyze the pullbacks of these isomorphisms and their effect on the cluster algebra associated to a braid variety.

\begin{proposition}\label{prop:isoVar} 
 Fix $u \leq \br$. Suppose $\br'$ is related to $\br$ by one of the moves  \bmref{bm1}--\bmref{bm5}. There is a natural isomorphism 
\begin{equation*}%
  \mis: \BR_{u, \br} \xrasim \BR_{u, \br'},\quad 	(\Xbul,\Ybul) \mapsto (\Xbul',\Ybul')
\end{equation*}
 given as follows.
	\begin{enumerate}
	\item[\bmref{bm1}--\bmref{bm3}:] If the braid move involves indices $l,l+1,\dots,r$, then $(\Xbul',\Ybul') \in \BR_{u, \br}$ is the unique tuple such that $(X_c', Y'_c)=(X_c, Y_c)$ for $c \notin \{l, l+1, \dots, r-1\}$. The remaining weighted flags $(X_l', Y_l'),\dots,(X_{r-1}', Y_{r-1}')$ are uniquely determined using \cref{lem:positionFacts}.
	\item[\bmref{bm4}:] Assume $u=\wo$. Suppose the last letter of $\br$ is in $I$. Then $(\Xbul',\Ybul')$ is given by $X'_{m-1}=X'_m=Y'_m:=X_{m-1}$, $Y'_{m-1}:=Y_{m-1}$, and $(X'_c,Y'_c):=(X_c,Y_c)$ for all $0\leq c<m-1$. 
	\item[\bmref{bm5}:] Suppose the first letter of $\br$ is in $I$. For $c\in[m]$, we set $(X'_c,Y'_c):=(X_c,Y_c)$. We also set $X'_0:=X_1$ and $Y_1':=Y_0=Y_1$. Recalling that $X_0\Rwrel{\wo} Y_0$, we let $Y_0'$ be the unique weighted flag satisfying $X_0\Rwrel{\wo s_{i^\ast}} Y_0' \Rrel{s_{i^\ast}} Y_0$ (cf. \lemmaref{lem:positionFacts}(3)). It follows that $X'_0\Lwrel{\wo}Y'_0$ and $Y'_0\Rrel{s_{i^\ast}} Y'_1$, so $(\Xbul',\Ybul')\in\BR_{\br'}$.\qed
	\end{enumerate}
\end{proposition}

From \cref{lem:w0enough} and \cref{prop:isoVar}, we have an immediate corollary.

\begin{corollary}
Any double braid variety $\BR_{\ubr}$ with $u\leq\br$ is isomorphic to a braid variety $\BR_{\br'}$ where $\br'$ has only red letters.
\end{corollary}
\noindent The purpose of defining double braid varieties is to obtain more seeds (or more Deodhar tori), one for each double braid word.

\subsection{Grid and \cham minors}
We next discuss regular functions on $\BR_{\ubr}$, which we eventually show are characters of cluster tori. Given $(\Xbul,\Ybul)$, let
\begin{equation}\label{eq:Z_c_dfn}
  Z_c := Y_c^{-1} X_c \in U_+ \backslash G / U_+.
\end{equation}
It is clear that $Z_{c}$ is well defined for a point in $\BR_{\ubr}$, since it is unchanged by the action of $G$. We frequently abuse notation and use $Z_c$ to denote both the double coset and a representative of this double coset. Our goal is to use the matrix $Z_c$ to introduce certain regular functions on $\BR_\ubr,\bigBR_\ubr,\BigBR_\ubr$ which we refer to as \emph{grid minors}. 

We first discuss the general problem of finding regular functions on $U_+ \backslash B_+ v B_+ / U_+$ and  $U_+ \backslash \overline{B_+ v B_+} / U_+$. 
Define a partial order $\preceq_v$ on $[n]$ by $i \preceq_v j$ if $i \leq j$ and $v(i) \geq v(j)$. Thus, if $v$ is the identity, then all distinct elements of $[n]$ are incomparable for $\preceq_v$ and, if $v = \wo$, then $\preceq_v$ is the usual order.
\begin{lemma}
Let $J$ be a lower order ideal for $\preceq_v$. Then the minor $\Delta_{v(J), J}$, considered as a function on $\overline{B_+ v B_+}$, is invariant for the left and right action of $U_+$.
\end{lemma}

\begin{proof}
We put a partial order on the collection of $k$-element subsets of $[n]$ by $\{ x_1 < x_2 < \cdots < x_k \} \preceq \{ y_1 < y_2 < \cdots < y_k \}$ if $x_i \leq y_i$ for each $i$.

Put $I = v(J)$. 
It is enough to prove left and right invariance on $B_+ v B_+$. Every element of $B_+ v B_+$ can be written as $u_1 v t u_2$ for $u_1$, $u_2 \in U_+$ and a diagonal matrix $t$. We will show that $\Delta_{I, J}(u_1 v t u_2) = \Delta_{I, J}(v t)$. 

By the Cauchy-Binet formula, we have
\[ \Delta_{I, J}(u_1 v t u_2)  = \sum_{I', J'} \Delta_{I, I'}(u_1) \Delta_{I', J'}(v t) \Delta_{J', J}(u_2) . \]
Since $u_1 \in U_+$, the first factor is nonzero only if $I \preceq I'$, and similarly the last factor is nonzero if and only if $J' \preceq J$. 
Meanwhile, the middle factor is nonzero if and only if $v(J') = I'$. 
Since $J$ is an order ideal for $\preceq_v$, the inequalities $I \preceq v(J')$ and $J' \preceq J$ can only be satisfied with $J' = J$ and, thus, $I'= I$.
So the only nonzero term is $\Delta_{I,I}(u_1) \Delta_{I,J}(vt) \Delta_{J,J}(u_2) = \Delta_{I,J}(vt)$, as promised.
\end{proof}

\begin{corollary} \label{cor:WellDefinedMinors}
For any $1 \leq h \leq n-1$, the minors $\Delta_{v[h], [h]}$ and $\Delta_{[n+1-h, n], v^{-1} [n+1-h,n]}$ are $U_+ \times U_+$ invariant on $\overline{B_+ v B_+}$.
\end{corollary}

\begin{remark}
In fact, the functions $\Delta_{v(J), J}$ generate the ring of $U_+ \times U_+$ invariant functions on $\overline{B_+ v B_+}$; see \cite{MO418142}.
\end{remark}

Now, recall from \cref{sec:double_braid_words} the definition of the $u$-PDS $\bu=(u\pd{0}, \dots, u\pd{m})$.

\begin{definition}\label{dfn:grid_minors}
	For $c \in \Jo$ a crossing and $h \in I$, we define the \emph{red grid minor} as
	\begin{equation}\label{eq:DeltaGrid}
          \grid_{c, h}(\Xbul,\Ybul)= \Delta_{\wo u\pd{c}[h], [h]}(Z_c),
	\end{equation}
        and the \emph{blue grid minor} as
\begin{equation}\label{eq:DeltaGrid_blue}
  \grid_{c,-h}(\Xbul,\Ybul)=\Delta_{[n+1-h, n], u\pd{c}^{-1}\wo [n+1-h, n]}(Z_c), %
\end{equation}
where $Z_c=Y_c^{-1}X_c$. 
\end{definition}
\noindent Note that the color of a grid minor $\grid_{c,h}$ is determined by the sign of $h$. 
\begin{lemma}\label{lemma:grid_minors_are_functions}
  For each $c \in[0,m]:=\{0,1,\dots,m\}$ and $h \in I$, the grid minors $\grid_{c,\pm h}$ give rise to $G$-invariant regular functions on $\BigBR_\ubr$ and $\bigBR_\ubr$, and to regular functions on $\BR_\ubr$. These regular functions are compatible with the quotient map $\bigBR_\ubr\to\BR_\ubr$ and the inclusion map $\bigBR_\ubr\hookrightarrow \BigBR_\ubr$.
\end{lemma}

\begin{proof}
By the definition of $\bigBR_\ubr$, there is some $g$ and some $u' \succeq u\pd{c}$ such that $(X_c, Y_c) \in g \wo u' B_+ \times g  B_+$ and thus
$Z_c \in B_+ \wo u' B_+    \subset \overline{B_+ \wo u\pd{c} B_+}$. 
From \cref{cor:WellDefinedMinors}, the minors $\Delta_{\wo u\pd{c}[h],\ [h]}$ and $\Delta_{[n+1-h, n], u\pd{c}^{-1}\wo [n+1-h, n]}$ are $U_+ \times U_+$ invariant functions on $\overline{B_+ \wo u\pd{c} B_+}$, and hence descend to well defined functions on the quotient $\BR_\ubr$.
\end{proof}

\begin{lemma}\label{lem:manyMinorsStable}
	Consider a crossing $c\in[m]$. Suppose $h \neq i_c$ is of the same sign as $i_c$. Then 
	\[
		\grid_{c-1, h} = \grid_{c, h}.
	\]
\end{lemma}

\begin{proof}
	Let $i:=i_c$. The result follows from the fact that $Z_{c-1}= Z_c x_{i}(t) \ds_{i}$ or $Z_{c-1}=\ds_{i^*} x_{i^*}(t) Z_c$ (as usual using $Z_c$ also to denote a representative for the double coset) for some $t \in \CC$, depending on whether $c$ is red or blue.
\end{proof}

We will be particularly interested in the following subclass of grid minors.
\begin{definition}\label{dfn:cham}
  Define the \emph{\cham minor} of $c$ as
\begin{equation}\label{eq:chamber}
  \crossing{c}:=\Delta_{c-1, i_{c}}.
\end{equation}
\end{definition}
\noindent Note that $\crossing{c}$ is a minor of $Z_{c-1}$. %
 In words, \cref{lem:manyMinorsStable} shows that a crossing of color $x\in\{\text{red},\text{blue}\}$ can change exactly one grid minor of color $x$ (but may change many grid minors of the opposite color). The ``changed" grid minor of color $x$ to the left of the bridge $\bridge_c$ is exactly the \cham minor $\crossing c$.

\begin{figure}
\def\wid{0.98\textwidth}
\begin{tabular}{c}
  \begin{tabular}{ccc}
    \igrw[0.4]{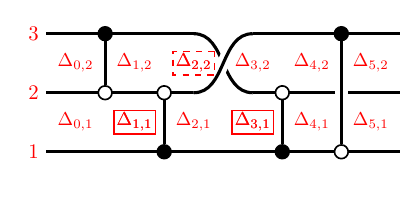} & \qquad\qquad&\igrw[0.4]{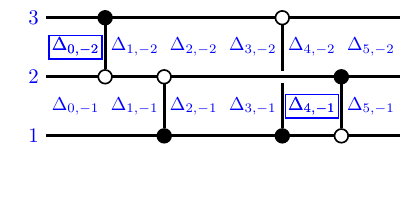} \\[-30pt]
    (a) $\pired(\Gubr)$, red grid minors (labels) & &(b) $\piblue(\Gubr)$, blue grid minors (labels)
  \end{tabular}\\[50pt]
  \igrw[1.0]{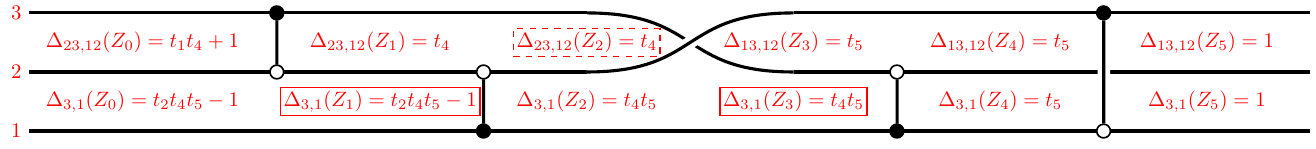}\\
 (c) $\pired(\Gubr)$, red grid minors (values)\\[10pt]
  \igrw[1.0]{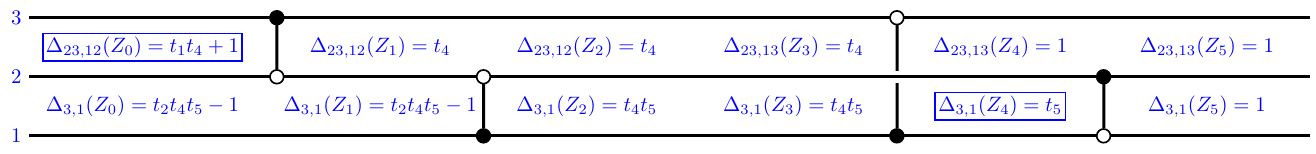}\\
  (d) $\piblue(\Gubr)$, blue grid minors (values)\\[10pt]
  \igrw[1.0]{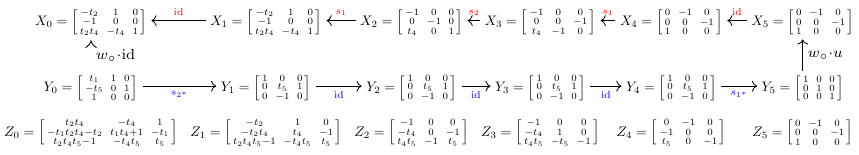}\\
  (e) the matrices $X_c,Y_c,Z_c$ for $c\in[0,m]$
\end{tabular}
  \caption{\label{fig:grid}Grid minors for $u=s_2$ and $\br=(-2,1,2,1,-1)$; see \cref{ex:3D}, \cref{fig:2D-ex}, and \cref{ex:grid}. The \cham minors $\crossing{c}$, $c\in[m]$, are boxed.}
\end{figure}

\begin{remark}\label{rmk:grid_3D}
We explain the relationship between grid minors and 3D plabic graphs; compare \cref{fig:grid} to \cref{fig:2D-ex}. Recall from \cref{sec:3D_plabic} that for a 3D plabic graph $\Gubr$ with coordinates $(i,j,t)$, we consider its red projection $\pired(\Gubr)$ with coordinates $(t,j)$ and its blue projection $\piblue(\Gubr)$ with coordinates $(t,i)$. For $c\in[0,m]$ and $h\in I$, let us place the red grid minor $\grid_{c,h}$ at the point $(t=c,j=h+0.5)$, lying in some face of $\pired(\Gubr)$. Similarly, we place the blue grid minor $\grid_{c,-h}$ at the point $(t=c,i=h+0.5)$, lying in some face of $\piblue(\Gubr)$. This labeling is shown in \figref{fig:grid}(a,b). Each \cham minor $\crossing{c}$ appears immediately to the left of the bridge $\bridge_c$ in the projection of the corresponding color; these minors are boxed in \cref{fig:grid}.
\end{remark}

\begin{remark}
\Cref{lem:manyMinorsStable} can be seen as a special case of the observation that any two red (resp., blue) grid minors that belong to the same face of $\pired(\Gubr)$ (resp., $\piblue(\Gubr)$) are equal.
\end{remark}

\begin{example}\label{ex:grid}
We continue \cref{ex:3D}; thus, $u=s_2$, $\br=(-2,1,2,1,-1)$, and $\Jo=\{1,2,4,5\}$. The matrices $X_c,Y_c,Z_c$ are computed in \figref{fig:grid}(e) from right to left using~\eqref{eq:propagate_left}. The grid minors given by~\eqref{eq:DeltaGrid}--\eqref{eq:DeltaGrid_blue} are computed in \figref{fig:grid}(c,d). The solid (resp., hollow) \cham minors are boxed in solid (resp., dashed) lines.
\end{example}

See \cref{sec:comparison} for a comparison of the \cham minors defined here with minors that have previously appeared in work on double Bruhat cells and open Richardson varieties.

\section{Deodhar geometry and seeds}\label{sec:seeds}
\def\Jos{J_\br}
\def\Josp{J_{\brp}}
\def\Jmuts{J^{\mut}_\br}
\def\Jfros{J^{\fro}_\br}
Let $u\leq \br$. We will use the $u$-positive distinguished subexpression  (cf. \cref{sec:double_braid_words}) to define a torus in $\BR_{\ubr}$, which will ultimately be a cluster torus. We will then use $(u,d)$-almost positive sequences (cf. \cref{sec:APS}) to define hypersurfaces in $\BR_{\ubr}$, and in turn to define cluster variables for $\BR_{\ubr}$.

\subsection{The Deodhar torus}
We continue to denote the $u$-PDS by $\bu=(u\pd{0}, \dots, u\pd{m})$.
\begin{definition}
	The \emph{Deodhar torus} $\torus(\ubr) \subset \BR_{\ubr}$ is the subset of $\BR_{\ubr}$ given by the conditions
	\begin{equation}\label{eq:DTcond}
	X_c \Lwrel{\wo  u\pd{c}} Y_c \qquad \mbox{	for $c=0,1,\ldots,m$.}
	\end{equation}
\end{definition}

\begin{remark}\label{rmk:torusIntuition}
	Points in $\torus(\ubr)$ lift to walks of the sort described in \cref{rmk:braidVarIntuition} where at each step, one greedily increases the relative position of $X_c$ and $Y_c$.
\end{remark}

Recall from the proof of \cref{lemma:grid_minors_are_functions} that all grid minors are nonvanishing on the Deodhar torus $\torus(\ubr)$.

As the next proposition shows, the Deodhar torus is in fact a torus, and the solid \cham minors form a basis of characters.

\begin{proposition}[{\cite[Corollary 2.8, Proposition 2.12]{GLSBS2}}]\ \label{prop:deodharIsTorus} 
\begin{enumerate}
\item The Deodhar torus $\torus(\ubr)$ is an open subset of $\BR_{\ubr}$, isomorphic to an algebraic torus of dimension $d(\ubr)$.
\item  The character lattice of $\torus(\ubr)$ consists of Laurent monomials in the solid \cham minors $\Delta_c$, $c \in \Jo$.  The ring of regular functions on $\torus(\ubr)$ is the ring of Laurent polynomials in the \cham minors.
\item The grid minors are characters of $\torus(\ubr)$.
\end{enumerate}

\end{proposition}

While the braid variety $\BR_{\ubr}$ is unchanged by braid moves, the Deodhar torus $T_{\ubr}$ may change. In particular, we will show later that non-mutation moves preserve the Deodhar torus, while mutation moves change it.

\subsection{Deodhar hypersurfaces}

For $d\in \Jo$, recall the notion of the $(u,d)$-almost positive sequence $(\vu\pd{0}\apd{d}, \dots, \vu\pd{m}\apd{d})$ from \cref{def:almostPos} and recall from \cref{cor:mutable-from-aps} that $d$ is mutable if and only if $\vu\pd{0}\apd{d} = \id$. In this section, we use the $(u,d)$-APS for solid~$d$ to define Deodhar hypersurfaces.

\begin{definition}
Let $d \in \Jo$.  Define the \emph{Deodhar hypersurface} $V_d \subset \BigBR_{\ubr}$ to be the closure of the locus satisfying
\begin{equation}\label{eq:deodharHyp}
  X_{c} \Lwrel{\wo  \vu\pd{c}\apd{d}} Y_c \quad\text{for all $c \in [0,m]$.}
\end{equation}
\end{definition}

Note that an index $d \in \Jo$ is mutable (resp., frozen) if and only if $V_d \subset \bigBR_{\ubr}$ (resp., $V_d \cap \bigBR_{\ubr}= \emptyset$). If $d$ is mutable, then the $G$-action on $V_d$ is free and $V_d/G$ is a subvariety of $\BR_{\ubr}$. 

\begin{remark}
	Points in the locus satisfying \eqref{eq:deodharHyp} lift to walks of the sort described in \cref{rmk:braidVarIntuition} where at every step besides step $d$, one greedily increases the relative position of $X_c$ and $Y_c$. At step $d$, one makes a ``mistake'' and decreases the relative position of $X_d$ and $Y_d$. 
\end{remark}

The Deodhar hypersurfaces form the complement of the Deodhar torus, as the next proposition shows. We denote by $\ttorus(\ubr)$ the pre-image of $\torus(\ubr)$ under the quotient map $\bigBR_{\ubr} \to \BR_{\ubr}$.

\begin{proposition}[{\cite[Proposition 2.19]{GLSBS2}}]\label{prop:deodharHyp}
The closed subset $\BigBR_{\ubr} \setminus \ttorus(\ubr)$ is the union of the Deodhar hypersurfaces $V_d$ for $d \in \Jo$. Each Deodhar hypersurface $V_d$ is irreducible and has codimension~$1$ in $\BigBR_{\ubr}$.
\end{proposition}

For a grid minor $\grid_{c, h}$ and $d \in \Jo$, we denote by $\ord_{V_d} \grid_{c,j}$ the order of vanishing of $\grid_{c, h}$ on the Deodhar hypersurface $V_d \subset \BigBR_{\ubr}$. By \cref{lemma:grid_minors_are_functions}, $\grid_{c, h}$ is regular on $\BigBR_{\ubr}$, so $\ord_{V_d} \grid_{c,j} \geq 0$. The next proposition determines $\ord_{V_d} \grid_{c,j}$ exactly, using the $u$-PDS and $(u,d)$-APS.

\begin{proposition}\label{prop:combGeoMatch} Fix $d \in \Jo$. For $j \in I$,
  \begin{equation*}%
    \ord_{V_d} \grid_{c,j} = \begin{cases}
		1 & \text{ if }u\pd{c}[j] \neq \vu\pd{c}\apd{d}[j],\\
		0 & \text{ otherwise,}
              \end{cases}
              \quad\text{and}\quad
              \ord_{V_d} \grid_{c,-j} = \begin{cases}
		1 & \text{ if }u^{-1}\pd{c}[j] \neq (\vu\pd{c}\apd{d})^{-1}[j],\\
		0 & \text{ otherwise.}
	\end{cases}
  \end{equation*}
\end{proposition}
\begin{remark}\label{rem:orderone}
In the case of braid varieties of arbitrary type treated in~\cite{GLSBS2}, the analog of \cref{prop:combGeoMatch} becomes more complicated. It is still the case that  $\ord_{V_d} \grid_{c,j}=0$ if and only if $u\pd{c}[j]=\vu\pd{c}\apd{d}[j]$; see~\cite[Proposition~2.21]{GLSBS2}. However, if $u\pd{c}[j]\neq \vu\pd{c}\apd{d}[j]$ then we have $\ord_{V_d} \grid_{c,j}\geq1$, and determining the exact value of $\ord_{V_d} \grid_{c,j}$ becomes a highly nontrivial problem; see~\cite[Section~7]{GLSBS2} for a combinatorial algorithm.
\end{remark}
\begin{proof}
  We compute the order of vanishing inside the space $S\subset\BigBR_\ubr$ of tuples $(\Xbul,\Ybul)$ of the form~\eqref{eq:Ugauge} (i.e., the space where $X_0$ has been gauge-fixed to $U_+$).  By \cref{lem:paramSingleStep}, this space $S$ is an affine space with coordinates $\t = (t_1,t_2,\ldots)$, as each $X_i$ or $Y_i$ is of the form  
$\z_{i_{j_1}}(t_{j_1}) \cdots \z_{i_{j_r}}(t_{j_r})U_+$ for some $r$. (These coordinates differ from those given by~\eqref{eq:propagate_left}.)  The $Z_c$ are then also products of inverses of braid matrices in some subset of the parameters $\t$.

Let $Z$ be any matrix that is a product of braid matrices or their inverses with parameters from $\t$, each parameter used at most once.  Then it is easy to see that every minor of $Z$ is linear or constant in each variable $t_i$.  The Deodhar hypersurface $V_d$ is cut out by a minor of $Z_{d-1}$, and this minor has degree one in some parameter $t = t_{i(d)}$.  Any grid minor $\grid_{c,j}$ is at most degree one in $t$, and thus vanishes to order at most one on $V_d$.

Suppose $j \in I$.  Then $u\pd{c}[j] = \vu\pd{c}\apd{d}[j]$ if and only if for a generic point in $V_d$, the $j$-th subspace in the weighted flag $X_c$ has the correct relative position (given by $u\pd{c}[j]$) with respect to $Y_c$.  This holds if and only if $\grid_{c,j}$ does not vanish on $V_d$. As we showed above, $\ord_{V_d}\grid_{c,j}\leq 1$, and thus 
 we obtain the stated formula for $j \in I$.  The argument for $j \in -I$ is identical.
\end{proof}

\begin{proposition}\label{prop:ordDeterminesCharacter}
The map 
\begin{equation}\label{eq:ord}
f \mapsto (\ord_{V_c} f)_{c \in \Jo} 
\end{equation}
is an isomorphism from the character lattice of $\torus(\ubr)$ to $\Z^{d(\ubr)}$. 
\end{proposition} 

\begin{proof}
	Let $C \subset \Z^{d(\ubr)}$ be the image of $\{\crossing{c}\}_{c \in \Jo}$. By \cref{prop:combGeoMatch}, the \cham minor $\crossing{c}$ vanishes to order $1$ on $V_c$, and does not vanish on $V_d$ for $d<c$. Thus, there is an upper unitriangular matrix taking the standard basis of $\Z^{d(\ubr)}$ to $C$. In particular, the images of the \cham minors form a basis of $\Z^{d(\ubr)}$. 
\end{proof}

\subsection{Seeds}\label{ssec:seeds}
The isomorphism \eqref{eq:ord} allows us to define new distinguished characters on $\torus(\ubr)$, which will be the cluster variables.

\begin{definition}\label{def:cluster-var}
For $c \in \Jo$, we define the \emph{cluster variable} $x_c$ to be the unique character of $\torus(\ubr)$ that vanishes to order one on $V_c$ and has neither a pole nor a zero on $V_{c'}$ for $c'\in\Jo\setminus\{c\}$. We denote the cluster by $\xubr=\{x_c\}_{c \in \Jo}$.
\end{definition}

\begin{example}\label{ex:grid2}
  Continuing \cref{ex:grid}, we see that in the notation of \cref{fig:grid}, the cluster variables are given by
  \begin{equation*}%
    x_1=t_1t_4+1,\quad x_2=t_2t_4t_5-1,\quad x_4=t_4,\quad x_5=t_5.
  \end{equation*}
  In particular, $\Delta_4=t_4t_5=x_4x_5$ factors as a product of two cluster variables.
\end{example}

It is immediate from the definition that each cluster variable is a regular function on $\BR_{\ubr}$ and in particular each frozen variable is a unit in $\C[\BR_\ubr]$. It is also immediate that the cluster variables in $\xubr$ are algebraically independent and irreducible.

Recall the definition of the ice quiver $\Qubr$ from \eqref{eq:bice_dfn}. The vertices of $\Qubr$ are labeled by $c \in \Jo$, as are the elements of $\xubr$. So we define the seed
\begin{equation}\label{eq:seed}
\seedubr:=(\xubr, \Qubr).
\end{equation}

We may now state our main result.
\begin{theorem}\label{thm:main}
  For all $u\leq\br$, we have
\begin{equation*}%
  \C[\BR_\ubr]= \Acal(\seedubr)
\end{equation*}
as subrings of $\C(\BR_{\ubr})$. 
 Moreover, the cluster algebra $\Acal(\seedubr)$ is locally acyclic and really full rank.
\end{theorem}
\noindent The proof of the isomorphism will occupy Sections~\ref{sec:movesOnSeed}--\ref{sec:proofmain}. The local acyclicity and the really full rank property of $\Acal(\seedubr)$ are \cref{cor:locAcyclic,prop:full_rank}.

We say that a cluster variable $x_c$ \emph{appears} in a grid minor $\grid_{d, h}$ if $\ord_{V_c} \grid_{d, h}=1$. One characterization for when this happens is given by \cref{prop:combGeoMatch}. We now relate it to the combinatorics of 3D plabic graphs; cf. \cref{sec:intro:seed}. Recall from \cref{sec:rel_cycles} that for each $d\in\Jo$ we have a relative cycle $\Cycle_d$ in $\Gubr$ which bounds a disk $\Disk_d$. Recall also from \cref{rmk:grid_3D} that we decorate the faces of the projections of $\Gubr$ with grid minors.

\begin{lemma}\label{lem:ordViaSoap}
Let $d \in \Jo$, $c \in [0, m]$, and $h \in I$ (resp., $h\in -I$). Then $x_d$ appears in $\grid_{c, h}$ if and only if $\grid_{c,h}$ lies in the red (resp., blue) projection of the disk $D_d$.
\end{lemma}
\begin{proof}
	This follows from \cref{prop:combGeoMatch,prop:gswap}. \Cref{prop:combGeoMatch} computes $\ord_{V_d} \grid_{c,h}$ by comparing certain subsets: $u\pd{c}[h]$ and $\vu\pd{c}\apd{d}[h]$ for the red minor, and  $(u\pd{c})^{-1}[|h|]$ and $(\vu\pd{c}\apd{d})^{-1}[|h|]$ for the blue minor. 
	
	\Cref{prop:gswap} shows that the permutation $\vu\pd{c}\apd{d}$ can be recovered from the monotone multicurve $\bgamma^{d,c}$. From the details of this construction, one can conclude that $u\pd{c}[h] \neq \vu\pd{c}\apd{d}[h]$ if and only if one of the curves in $\bgamma^{d,c}$ intersects the horizontal line $j=h+0.5$ inside $\PD(\upu c)$. This is equivalent to the point $(c, h+0.5)$ lying in the red projection of $D_d$. Similarly, $(u\pd{c})^{-1}[h]\neq(\vu\pd{c}\apd{d})^{-1}[h]$ if and only if one of the curves in $\bgamma^{d,c}$ intersects the vertical line $i=h+0.5$ in $\PD(\upu c)$ if and only if $(c, h+0.5)$ lies in the blue projection of $D_d$. Comparing with \cref{prop:combGeoMatch} gives the result.
      \end{proof}
\begin{example}
Recall from \cref{ex:grid2} that in \cref{fig:grid}, we have $\Delta_4=x_4x_5$. This is consistent with the fact that the face of $\pired(\Gubr)$ containing $\Delta_4=\Delta_{3,1}$ lies inside the cycles $\Cycle_4,\Cycle_5$ shown in \figref{fig:2D-ex}(e).
\end{example}

We now give an alternate construction of $\Qubr$ related to the half-arrow description in \cref{sec:half-arrow-descr}, which will be quite useful in analyzing the effects of braid moves on $\seedubr$ in \cref{sec:movesOnSeed}.

First, consider a collection $P_{\ubr}^+$ (resp., $P_{\ubr}^-$) of half-arrows between the red (resp., blue) grid minors $\{\grid_{c,h}\}_{c \in [0, m], h \in I}$ (resp., $\{\grid_{c,h}\}_{c \in [0, m], h \in -I}$). It is obtained by placing the configuration of half-arrows shown in \cref{fig:half-arrows-and-minors} around each bridge $b_c$ with $i_c>0$ in $\pired(\Gubr)$ (resp., each blue bridge $b_c$ with $i_c<0$ in $\piblue(\Gubr)$), where $i:=|i_c|$.
\begin{figure}[h]
	\includegraphics*[width=0.45\textwidth]{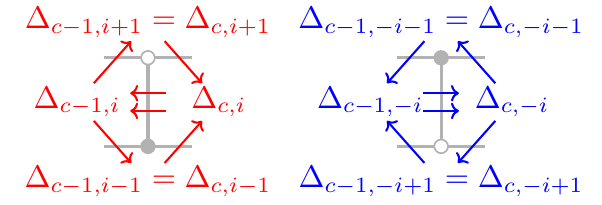}
	\caption{Half-arrow configuration to obtain $P_{\ubr}^+$ (left) and $P_{\ubr}^-$ (right) for \cref{prop:seedFromLabeledGraphs}.}
	\label{fig:half-arrows-and-minors}
\end{figure}

\begin{proposition}\label{prop:seedFromLabeledGraphs}
Define a quiver with vertex set $\xubr$ as follows. Write each grid minor as 
\[\grid_{c, h}= \prod_{d \in \Jo} x_d^{q^d_{c,h}},\]
where $q^d_{c,h}:= \ord_{V_d} \grid_{c, h}\geq0$.
For every half-arrow $\grid_{c, h} \to \grid_{c', h'}$ in $P_{\ubr}^+$ and $P_{\ubr}^-$, draw $q^d_{c, h}q^{d'}_{c', h'}$ arrows from $x_d$ to $x_d'$. (In other words, for every $x_d$ appearing in $\grid_{c, h}$ and $x_d'$ appearing in $\grid_{c', h'}$, draw a half-arrow $x_d \to x_d'$.) Then delete loops, 2-cycles and arrows between frozen vertices. 
The resulting collection of half-arrows agrees with the quiver $\Qubr.$ That is, the signed number of half-arrows between any two vertices is twice the number of arrows between the corresponding vertices of $\Qubr$.
\end{proposition}
\begin{proof}
	This follows directly from \cref{lem:ordViaSoap} and \cref{prop:half_arr}.
      \end{proof}

\section{Moves preserve the cluster algebra}\label{sec:movesOnSeed}
To each $u\leq \br$, we have associated a seed $\seedubr$ and thus also a cluster algebra $\Aubr:=\A(\seedubr) \subset \CC(\BR_{\ubr})$. Our ultimate goal is to show that $\CC[\BR_{\ubr}]$ is isomorphic to $\Aubr$. As an intermediate step, we show the following.

\begin{theorem}\label{thm:inducedIso} Suppose $\br$ and $\br'$ are related by one of the moves \bmref{bm1}--\bmref{bm5}. Let
 $\mis: \BR_{\ubr}\xrightarrow{\sim} \BR_{\ubrp}$ be the corresponding isomorphism from \cref{prop:isoVar}. Then the isomorphism $\mis^*: \CC(\BR_{\ubrp})\xrightarrow{\sim} \CC(\BR_{\ubr})$ restricts to an isomorphism $\Aubrp \xrightarrow{\sim} \Aubr$.
\end{theorem}

\begin{remark}\label{rmk:onlyCheckCluster}
	If $\br$ and $\br'$ are related by moves \bmref{bm1}--\bmref{bm4}, then the quivers $\Qubr$ and $\Qubrp$ are related either by relabeling the vertices or by mutation (cf. \cref{thm:invariance}, \cref{prop:bm5_bm4}). So in these cases, we have an isomorphism $\Aubr \cong \Aubrp$, and to prove \cref{thm:inducedIso}, we just need to check that the isomorphism is induced by the appropriate isomorphism of braid varieties.
\end{remark}

Throughout the following three subsections, we assume that $u=\wo$ (cf. \cref{notn:w0,lem:w0enough}), and that $\br$ and $\br'$ are related by a single move \bmref{bm1}--\bmref{bm5}.  We let $\mis:\BR_{\br} \to \BR_{\brp}$ denote the corresponding isomorphism, defined in \cref{prop:isoVar}. Any pullback mentioned is a pullback by $\mis$. We set $\Josp$ to be the set of solid crossings for $\brp$. For $x_c \in \x_{\br'}$, we denote $p_c := \mis^*(x_c)$ and $\pbr:=\{p_c\}_{c \in \Josp}$. The isomorphism $\mis$ also induces a map on the double cosets $Z_c$; we write $Z'_c$ to denote the image of $Z_c$ under this map. Pullbacks of grid minors are also denoted with primes.

\subsection{Non-mutation moves \texorpdfstring{\bmref{bm1}--\bmref{bm4}}{(B1)--(B4)} do not change the seed}

In this section, we prove \cref{thm:inducedIso} for non-mutation moves other than \bmref{bm5}. We show that in this case, $\seedbr$ and $(\pbr, \Qbrp)$ differ only by reindexing. As noted in \cref{rmk:onlyCheckCluster}, it suffices to check this statement for cluster variables.

\begin{proposition}\label{prop:nonMutPreserveVar}
 Suppose $\br$ and $\br'$ are related by a non-mutation move \bmref{bm1}--\bmref{bm4} and let $\alpha$ be the corresponding relabeling bijection $\Jos \to \Josp$ from \cref{sec:nonMutQuiver}. Then $x_c = p_{\alpha(c)}$.
\end{proposition}

\begin{proof} %
For \bmref{bm4}, $\alpha$ is the identity. It is clear that $\grid_{c, h}= \grid_{c, h}'$ and $x_d$ appears in $\grid_{c, h}$ if and only if $p_d$ appears in $\grid_{c, h}'$. The claim follows. Indeed, suppose $d \in \Jmuts$ and for all solid $c>d$, we already know $x_c=p_{\alpha(c)}$. Then $\grid_{d-1, i_d}= x_d M$ and $\grid_{d-1, i_d}' = p_{\alpha(d)} M'$, where $M$ (resp., $M'$) is a product of cluster variables $x_c$ (resp., $p_{\alpha(c)}$) with $c>d$. By above, the left-hand sides of these two equations are equal and $M=M'$. So e.g. by restricting to the Deodhar torus, we may conclude $x_d = p_{\alpha(d)}$.

Now, for non-mutation moves \bmref{bm1}--\bmref{bm3}. Suppose the non-mutation move involves the indices in some interval $(l, r] \subset [m]$ of length $2$ or $3$. Let $G^L$ (resp., $G^R$) denote the portion of $\Gbr$ between $\Gamma(u\pd{0})$ and $\Gamma(u\pd{l})$ (resp., between $\Gamma(u\pd{r})$ and $\Gamma(u\pd{m})$). Note that $G^L$ (resp., $G^R$) is also equal to the portion of $\Gbrp$ between $\Gamma(u'\pd{0})$ and $\Gamma(u'\pd{l})$ (resp., between $\Gamma(u'\pd{r})$ and $\Gamma(u'\pd{m})$). By \cref{sec:nonMutQuiver}, for all $c\in\Jos$, the cycles $C_c$ and $C'_{\alpha(c)}$ behave identically on $G^L$ and $G^R$. This implies a point is contained in the red (resp., blue) projection of $D_c$ if and only if it is contained in the red (resp., blue) projection of $D'_{\alpha(c)}$. Also, for $c \notin (l, r):=[l+1,r-1]$, we have $\grid_{c, h}=\grid_{c, h}'$ for all $h \in \pm I$. Using \cref{lem:ordViaSoap}, we obtain the following.
\begin{lemma}\label{lem:nonMutHelper}
Suppose $c \notin (l, r)$ and $d\in\Jos$. Then $x_d$ appears in $\grid_{c, h}$ if and only if $p_{\alpha(d)}$ appears in $\grid_{c, h}'=\grid_{c, h}$.
\end{lemma}
 
 Note that $\alpha$ is the identity on $(r, m]$. The same argument as in the \bmref{bm4} case shows $x_c = p_{\alpha(c)}$ for solid crossings $c>r$. For solid $c \in (l, r]$, to show $x_c=p_{\alpha(c)}$, it suffices to find a grid minor $\grid_{d, h}$ with $d\leq l$ which is a product of $x_c$ and other $x_{c'}$ which are already known to be equal to $p_{\alpha(c')}$. In fact, a grid minor $\grid_{l, h}$ will always work; this can be seen from \cref{fig:non-mut-iji,fig:easy-moves,fig:hard-moves-1} and the rules governing cycles. Now, $\alpha$ is also the identity on $[1, l]$, so the same argument as for \bmref{bm4} shows $x_c = p_{\alpha_c}$ for solid $c\leq l$.
\end{proof}

\subsection{Mutation moves mutate the seed}\label{sec:exc}
In this section, we prove \cref{thm:inducedIso} for mutation moves. We show that in this case, $\seedbr$ and $(\pbr, \Qbrp)$ are related by mutation. This statement has already been checked for the quivers, so we just check the cluster variables.

\begin{lemma}
Suppose $\br$ and $\br'$ are related by the mutation move whose rightmost crossing is $c+1$.
Then for $d \in \Jos\setminus \{c+1\}$, we have $x_d=p_d$.
\end{lemma}
\begin{proof}
This argument is similar to the proof of \cref{prop:nonMutPreserveVar}. Say the mutation move occurs on indices in the interval $(l, r] \in [m]$, with $r:=c+1$. Then for $h \in [0, l] \cup [r, m]$, the cluster variable $x_{c+1}$ does not appear in grid minors $\grid_{d, h}$ and $\grid_{d, h}= \grid_{d, h}'$. Moreover, a cluster variable $x_{c'} \neq x_{c+1}$ appears in $\grid_{d, h}$ for $h \in [0, l] \cup [r, m]$ if and only if $p_{c'}$ appears in $\grid_{d, h}'$. The desired equality now follows from a triangularity argument.
\end{proof}

Now, we verify that the single new cluster variable in $\pbr$ satisfies the exchange relation given by $\seedbr$. We will need some identities for grid minors, which we obtain from standard determinantal identities.

The following relations hold on $G$.
\begin{proposition}[Desnanot--Jacobi identity]
Suppose $p, q \in S_n$ and $a\in I$ with $\ell(ps_a) = \ell(p)+1$ and $\ell(qs_a) = \ell(q) +1$. Then
\begin{equation}\label{eq:FZ1}
\Delta_{p[a],q[a]} \Delta_{ps_a[a],qs_a[a]}= \Delta_{ps_a[a],q[a]} \Delta_{p[a],qs_a[a]}+  \Delta_{p[a-1],q[a-1]}\Delta_{p[a+1],q[a+1]}.
\end{equation}
\end{proposition}

\begin{proposition}[{\cite[Theorem 1.16(1)]{FZ_double}}]
Let $p, q\in S_n$ and $a,b\in I$. Suppose $(s_as_b)^3 = 1$ and $\ell(qs_as_bs_a) = \ell(q)+3$.  Then
\begin{equation}\label{eq:FZ2}
\Delta_{p[a], qs_a[a]} \Delta_{p[b], qs_b [b]} = \Delta_{p[a], q [a]} \Delta_{p[b], qs_as_b [b]} +\Delta_{p[a], qs_bs_a [a]} \Delta_{p[b], q[b]}.
\end{equation}
\end{proposition}

\begin{lemma} \label{lem:sSwitch} Let $A, B \subset [n]$ with $|A|=|B|$, $a\in I$, and suppose $s_a (A) \geq A$ in the lexicographic order. Then
	\[\Delta_{s_a(A), B}(\ds_a M)=\Delta_{A, B}(M) \quad \text{and} \quad \Delta_{B, s_a(A)}(M\ds_a)=-\Delta_{B, A}(M).\]
\end{lemma}

Now, we use the determinantal identities above to obtain relations on grid minors before and after a mutation move. 

\begin{proposition}\ 
	\begin{enumerate}
		\item Suppose $\br, \br'$ are related by a \bmref{bm1} special solid move $ j i \to i j$ on crossings $c, c+1$. Then the grid minors on $\BR_{\br}$ satisfy
		\begin{equation}\label{eq:M1identity}
			\grid_{c, j} \grid_{c,j}' = \grid_{c+1, j} \grid_{c-1, j}  + \grid_{c, j-1} \grid_{c, j+1}.
		\end{equation}
	 \item Suppose $\br, \br'$ are related by a \bmref{bm3} fully solid move $iji \to jij $ on crossings $c-1, c, c+1$. Then the grid minors on $\BR_{\br}$ satisfy
	 \begin{equation}\label{eq:M2identity}
	 	\grid_{c, i} \grid_{c, j}' = \grid_{c+1, i} \grid_{c-2, j} + \grid_{c+1, j} \grid_{c-2, i}.
	 \end{equation}
	\end{enumerate}
\end{proposition}
\begin{proof} To avoid clutter, we set $v:=u\pd{c+1}$.
	
For (1): We will assume $j<0$ and $i>0$, as the other case is similar. Let $k=-j$. Note that we can write $Z_{c-1}, Z_c$ and $Z_c'$ in terms of $Z_{c+1}$:
\[Z_{c-1}=(\z_{k^*}(t'))^{-1}Z_{c+1}z_i(t)=\ds_{k^*}x_{k^*}(t') Z_{c+1} x_{i}(t) \ds_i\]
while $Z_c=Z_{c+1} x_{i}(t) \ds_i$ and $Z_c' = \ds_{k^*}x_{k^*}(t') Z_{c+1}$.

Now, apply \eqref{eq:FZ1} to $Z_{c-1}$ with $p=\wo  s_k$, $q=v^{-1} $, and $a=k$. Using $v s_i = s_k v >v$, \cref{lem:sSwitch}, \cref{lem:manyMinorsStable}, and the $U_+ \times U_+$-invariance of the grid minors, we obtain Equation \eqref{eq:M1identity}. For example, 
\begin{align*}
	\Delta_{p[a],qs_a[a]}&=\Delta_{s_{k^*}\wo [k],s_iv^{-1}[k]}(Z_{c-1})
	=-\Delta_{\wo [k],s_iv^{-1}[k]}(Z_{c})
	=\Delta_{\wo [k],v^{-1}[k]}(Z_{c+1})=\grid_{c+1, j}.
\end{align*}

For (2): Using the same strategy as above, we apply \eqref{eq:FZ2} to $Z_{c-2}$ (or its transpose) to obtain \eqref{eq:M2identity}. If $i>0$ and $j=i+1$, for example, we set $p=\wo  v$, $q=\id$, $a=i$ and $b=i+1$.
\end{proof}

We need one additional identity on grid minors to show that $p_{c+1}$ satisfies the correct exchange relation.

\begin{lemma}\label{lem:minorShortRelations}
	Choose a crossing $c$.  Let $a \in [n]$ and let $b=-u\pd{c}(a)$. Then the relation 
	\[ \grid_{c, a} \grid_{c, b+1} = \grid_{c, b} \grid_{c, a-1}\]
	holds on $\BR_{\br}$, where we set $\grid_{c,0}=\grid_{c,\pm n}:=1$.
\end{lemma}

\begin{proof}
  It suffices to show that this relation holds on the Deodhar torus $\torus(\br)$, which is dense. %
  A point in $ \torus(\br)$ is represented by $(X_\bullet, Y_\bullet) \in \bigBR_{\br}$ with $X_c \Lwrel{\wo  u\pd{c}} Y_c$. Using the $G$-action, we may choose the representative $(X_\bullet, Y_\bullet)$ so that the matrix $Z_c= Y_c^{-1} X_c$ is of the form
	\[g:=Z_c=h (\wo  u\pd{c})^{\bigcdot}\]
	for some $h \in H$. Only the entries $g_{\wo  u\pd{c}(i), i}$ are nonzero. So for $i>0$, $\grid_{c, i}$ is (up to sign) the product of the nonzero entries of $g$ in columns $1, \dots, i$. For $i<0$, $\grid_{c, i}$ is (up to sign) the product of the nonzero entries of $g$ in rows $\wo (|i|), \wo (|i|)+1 , \dots, n$. Using this, we find
	\[\frac{ \grid_{c, a}}{\grid_{c, a-1}} =  \frac{ \grid_{c, b}}{\grid_{c, b+1}} = (-1)^q g_{\wo  (|b|), a}. \] 
	where $q$ is the number of nonzero entries of $g$ southwest of $g_{\wo  (|b|), a}$.
\end{proof}

\begin{proposition}\label{prop:checkingMutMoves}
	Let $\br$ and $\br'$ be related by a mutation move whose rightmost crossing is $c+1$. Then $(\pbr, \Qbrp)=\mu_{c+1}(\seedbr)$.
\end{proposition}

\begin{proof}\ 
  \textbf{Case 1:} Suppose the mutation move is a \bmref{bm1} special solid move $j i \to i j$ on crossings $c, c+1$. Without loss of generality, we may assume $j<0$ and $i>0$. We see from \figref{fig:mut-ij-curves}(a) and 
  \cref{lem:ordViaSoap} that $x_{c+1}$ appears only in grid minors $\grid_{c, i}$ and $\grid_{c, j}$. Using \cref{prop:seedFromLabeledGraphs}, we find that the cluster variable $x'_{c+1}$ in $\mu_{c+1}(\seedbr)$ satisfies the exchange relation
	\begin{align}
	x_{c+1} x'_{c+1}&=(\grid_{c+1, i}\grid_{c-1, j} + (\grid_{c, i-1} \grid_{ c, i+1}\grid_{c, j-1} \grid_{c, j+1})^{1/2})/f\\
\label{eq:sqrt}		&=(\grid_{c+1, i}\grid_{c-1, j} + \grid_{c, i-1} \grid_{c, j-1})/f
	\end{align}
	where $f$ is the cluster monomial
	\[ f= \prod_{d \neq c+1} x_d^{\min(\ord_{V_d}(\grid_{c+1, i}\grid_{c-1, j}), \ord_{V_d}(\grid_{c, i-1} \grid_{c, j-1}))}. \]
The equality in~\eqref{eq:sqrt} follows from \cref{lem:minorShortRelations}: since $j=-\upu c(i)$ and $j-1=-\upu c(i+1)$, we get $\grid_{c,j+1}=\grid_{c,i-1}\grid_{c,j}/\grid_{c,i}$ and $\grid_{c,i+1}=\grid_{c,j-1}\grid_{c,i}/\grid_{c,j}$.
	
	On the other hand, using \eqref{eq:M1identity} and \cref{lem:minorShortRelations} together, we have that 
	\begin{equation}\label{eq:geoRel}
		\grid_{c, j}' \grid_{c, i}=\grid_{c+1, i} \grid_{c-1, j} + \grid_{c, i-1}\grid_{c, j-1}
	\end{equation}
	on $\BR_{\br}$, where $\grid_{c, j}'$ is the grid minor evaluated after the braid move. This can be seen by working on the Deodhar torus and multiplying \eqref{eq:M1identity} by $\grid_{c+1, i-1}/\grid_{c+1, j+1}$; one must also use $\grid_{c+1, j+1}=\grid_{c, j+1}$, which follows from the move being special solid.
	
	 To show that $p_{c+1}$ obeys the exchange relation, we will prove that the left-hand side of \eqref{eq:geoRel} is equal to $f x_{c+1}p_{c+1}$. It is easy to check that $x_{c+1}$ appears in $\grid_{c, i}$ and $p_{c+1}$ appears in $\grid_{c, j}'$. So all that remains is to show that for $d \neq c+1$,
	\[\ord_{V_d} (\grid_{c, j}'\grid_{c, i}) = \min(\ord_{V_d}(\grid_{c+1, i}\grid_{c-1, j}), \ord_{V_d}(\grid_{c, i-1} \grid_{c, j-1})). \]
	 \cref{eq:ord_min_lemma} implies an analogous statement for \eqref{eq:M1identity}; that is, the order of vanishing of the left-hand side on $V_d$ is the minimum of the orders of vanishing of the two terms on the right. Since \eqref{eq:geoRel} differs from \eqref{eq:M1identity} by a Laurent monomial in cluster variables, we obtain the desired equality for \eqref{eq:geoRel}. 
	 
	 \textbf{Case 2:} Suppose the mutation move is a \bmref{bm3} fully solid move $iji \to jij $ on crossings $c-1, c, c+1$. We assume $i>0$; the $i<0$ case is similar. The proof strategy is similar to Case 1, so we will be brief. By \cref{prop:combGeoMatch}, the only red grid minors $x_{c+1}$ appears in are $\Delta_{c, i}=\Delta_{c-1, i}$. So by \cref{prop:seedFromLabeledGraphs}, the exchange relation for $x'_{c+1}$ is 
	 \[	x_{c+1} x'_{c+1}=(\grid_{c+1, i} \grid_{c-2, j} + \grid_{c+1, j} \grid_{c-2, i})/f\]
	 where $f$ is the cluster monomial	 \[f= \prod_{d \neq c+1} x_d^{\min(\ord_{V_d}(\grid_{c+1, i} \grid_{c-2, j}), \ord_{V_d} (\grid_{c+1, j} \grid_{c-2, i}))}.\]
	 
	 On the other hand, \eqref{eq:M2identity} and \eqref{eq:ord_min_lemma} together imply that the right-hand side of the exchange relation is equal to $x_{c+1} p_{c+1}$.
\end{proof}

\subsection{Move \texorpdfstring{\bmref{bm5}}{(B5)} rescales the seed}

In this section, we prove \cref{thm:inducedIso} for \bmref{bm5}. In particular, we show \cref{thm:inducedIso} for $\br=(-i) \gamma$ and $\br'=i \gamma$ where $i \in I$, which also implies the theorem for $\br=i \gamma$, $\br'=(-i) \gamma$ and $i \in I$.

If the first crossing of $\br=i\gamma$ is hollow, this is straightforward (\cref{lemma:straightforward}). If the first crossing is solid, we show that $\seedbr$ and $(\pbr, \Qbrp)$ differ only in a single frozen variable and arrows involving that frozen variable, in a way that preserves the cluster algebra (\cref{prop:preserves}).

\begin{lemma}\label{lemma:straightforward}
	Suppose $\br= (-i) \gamma$ and $\br'= i \gamma$ for some positive double braid $\gamma$. If the first crossing is hollow, then $\seedbr = (\pbr, \Qbrp)$.
\end{lemma}

\begin{proof}
	Note that the first crossing is hollow in $\br$ if and only if it is hollow in $\br'$. The color of the hollow crossing does not matter in the definition of the surface $\Subr$ and does not affect the interaction of any mutable cycle with any other cycle, so $\Subr=\Subrp$ and $\Qbr=\Qbrp$. Further, the upper triangular matrix relating \cham minors and cluster variables is the same for both seeds, and all \cham minors pull back to \cham minors. This implies $x_c =p_c$ for all $c \in \Jos$.
\end{proof}

We now assume that the first crossings of $\br$ and $\br'$ are solid. We would like to apply the following proposition to the two seeds at hand.

\begin{proposition}[{\cite{Fraser}, cf. \cite[Proposition 5.11]{LS}}]\label{prop:preserves}
	Suppose $(\x, \Qt)$ and $(\x', \Qtp)$ are two seeds in $\mathcal{F}$ such that 
	\begin{itemize}
		\item $\Qt$ and $\Qtp$ are ice quivers on the same set of vertices, whose mutable parts coincide,
		\item the mutable variables in $\x$ and $\x'$ are the same, and the two sets of frozen variables are related by an invertible monomial transformation, and 
		\item for each mutable vertex $c$, 
		\[\prod_{v \in \Qt} x_v^{\#\text{arrows } v \to c \text{ in } \Qt}= \prod_{v \in \Qtp} (x_v')^{\#\text{arrows } v \to c \text{ in } \Qtp}. \]
	\end{itemize}
	Then $\A(\x, \Qt)= \A(\x', \Qtp)$.
\end{proposition}

We first analyze the relationship between the clusters $\xbr$ and $\pbr$.

\begin{lemma}\label{lem:bm5Equalities}
	Suppose $\br= (-i) \gamma$ and $\br'= i \gamma$ and the first crossing is solid. Then for $(c, h) \neq (0, \pm i)$, $x_d$ appears in $\grid_{c, h}$ if and only if $p_d$ appears in $\grid_{c, h}$. Further, $x_c = p_c$ for $c \neq 1$.
\end{lemma}

\begin{proof} Note that for $(c, h) \neq (0, \pm i)$, $\grid_{c, h}=\grid_{c, h}'$.
	
	For the first claim: by \cref{prop:combGeoMatch}, the appearance of a cluster variable $x_d$ in a grid minor $\grid_{c, h}$ depends only on $d$ and the suffix $i_{c+1} \dots i_{m}$ of $\br$. The same statement holds for $p_d$, $\grid_{c, h}'=\grid_{c,h}$ and $\br'$. Since for $c>0$, the suffixes of $\br$ and $\br'$ coincide, this implies the first claim for all grid minors with $c>0$. For $c=0$, by \cref{lem:manyMinorsStable}, each blue grid minor for $c=0$ besides $\grid_{0, -|i|}$ is equal to a blue grid minor for $c=1$. It follows from the definition that $\grid_{0, h} = \grid_{0, -h}$, so this gives the first claim for grid minors with $(c,h) \neq (0, \pm i)$.
	
	The second claim follows from the first, since cluster variables are unitriangularly related to \cham minors, and the first claim implies the unitriangular matrices for the two seeds differ only in the first row (or column).
\end{proof}

\begin{lemma}
	Suppose $\br= (-i) \gamma$ and $\br'= i \gamma$ with $i \in I$ and the first crossing is solid. Then $p_1=x_1^{-1} M$ where $M$ is a Laurent monomial in the frozen variables of $\xbr$ other than $x_1$.
\end{lemma}

\begin{proof} For $h \in H$, let $h_j:=h_{j,j}$.
	
	Gauge-fix as in the proof of \cref{prop:isoVar}, so $Z_0=h \dot{w}_0$ and $Z_0'=h' \dot{w}_0$, where $h'$ is obtained from $h$ by swapping the positions of $h_{n-i+1}$ and $h_{n-i}$. It is easy to see that $\grid_{0, a}=\grid_{0, -a}= h_n \cdots h_{n-|a|+1}$, and similarly for $\grid_{0, a}'$. Now, because $x_1$ appears in $\grid_{0, -i}$ but not $\grid_{0, -(i+1)}$, we have $h_{n-i}=x_1^{-1} N$ where $N$ is a Laurent monomial in the other frozen variables of $\xbr$. Similarly, because $p_1$ appears in $\grid_{0, i}'$ and not $\grid_{0, i-1}'$, we have $h'_{n-i+1}=h_{n-i}=p_1 N'$ where $N'$ is a Laurent monomial in the other frozen variables of $\pbr$. So $p_1 = x_1^{-1} N/N'$. Since all frozen variables in $\pbr \setminus \{p_1\}$ are equal to frozen variables of $\xbr$, we have proved the claim.
\end{proof}

\begin{lemma}Suppose $\br= (-i) \gamma$ and $\br'= i \gamma$ with $i \in I$ and the first crossing is solid.
  For each mutable crossing $c \in \Jmuts$, we have
  \begin{equation}\label{eq:prod=prod}
    \prod_{d \in \Qbr} x_d^{\#\text{arrows } d \to c \text{ in } \Qbr}= \prod_{d \in \Qbrp} p_d^{\#\text{arrows } d \to c \text{ in } \Qbrp}.
  \end{equation}
\end{lemma}

\begin{proof}
The result follows straightforwardly from \cref{prop:seedFromLabeledGraphs}. The quiver $P^+_{\br}$ is obtained from $P^+_{\brp}$ by adding the half-arrows coming from the leftmost red bridge of $\pired(\Gbr)$; the quiver $P^-_{\brp}$ is obtained from $P^-_{\br}$ by adding the half-arrows coming from the leftmost blue bridge of $\piblue(\Gbrp)$. Together with \cref{lem:bm5Equalities}, this implies that the quivers $\Qbr$ and $\Qbrp$ can be obtained from a third quiver $\Qt$ by adding a vertex $x_1$ or $p_1$ and adding the half-arrows contributed by the leftmost bridge.

As before, for $h \in H$, let $h_j:=h_{j,j}$. Gauge-fix as in \cref{prop:isoVar}, so $Z_0=h \dot{w}_0$ and $Z_0'=h' \dot{w}_0$, where $h'$ is obtained from $h$ by swapping the positions of $h_{n-i+1}$ and $h_{n-i}$.

Fix $c \in \Jmuts$, and let $L$ and $R$ denote the left- and right-hand sides of~\eqref{eq:prod=prod}. If $x_c$ does not appear in $\grid_{1, i}$, then there are no arrows between $x_c$ and $x_1$ and the arrows $x_d \to x_c$ in $\Qbr$ are in bijection with arrows $p_d \to p_c$ in $\Qbrp$. We find that~\eqref{eq:prod=prod} follows from \cref{lem:bm5Equalities}.

If $x_c$ does appear in $\grid_{1, i}$, then the arrows in $P^+_{\br}$ around the leftmost red bridge contribute
\[\frac{(\grid_{0, i+1} \grid_{0, i-1})^{1/2}}{\grid_{0, i}}= \frac{h_{n-i}^{1/2}}{h_{n-i+1}^{1/2}}\]
to $L$.

From \cref{lem:bm5Equalities}, we know $p_c$ appears in $\grid_{1, -i}'=\grid_{1,i}$. The half-arrows in $P^-_{\brp}$ around the leftmost blue bridge contribute 
\[\frac{\grid_{0, -i}'}{(\grid_{0, -(i+1)} \grid_{0, -(i-1)})^{1/2}}= \frac{h_{n-i}^{1/2}}{h_{n-i+1}^{1/2}}\]
to $R$.

Now, if we divide $L$ and $R$ by $(h_{n+i-1}/h_{n-i})^{1/2}$, we obtain two expressions that are equal. Indeed, the expressions are monomials in $\xbr \setminus \{x_1\}$ and $\pbr \setminus \{p_1\}$, respectively, and exponents in both expressions are determined by arrows in $\Qt$.
\end{proof}

\section{Proof of Theorem~\ref{thm:main}}\label{sec:proofmain}

In this section, we continue to assume $u=\wo$.  Recall that to each double braid word $\br \geq \wo$, we have associated a cluster algebra $\A_\br \subset \CC(\BR_{\br})$. By \cref{thm:inducedIso}, this cluster algebra is invariant under \bmref{bm1}--\bmref{bm5}.  

We will prove that $\C[\BR_{\br}]$ is equal to $\A_\br$ by induction on $m- \ell(\wo)$, where $m$ is the number of letters in $\br$. The key step of the argument is \cite[Theorem 4.10]{GLSBS2}. For the convenience of the reader, before recalling this theorem, we first outline how to translate between the conventions for seeds in \cite{GLSBS2} and those here. 

Let $\OmegaOp_\br$ be the seed obtained from the seed defined in \cite[Section 1.2]{GLSBS2} for $G= \SL_n$ by reversing all arrows.

\begin{lemma}\label{lem:seed-same-as-gen-type}
	The seed $\seedbr$ defined in \eqref{eq:seed} is identical to the seed $\OmegaOp_{\br}$.
\end{lemma}
\begin{proof}
	\cref{def:cluster-var} is identical to \cite[Definition 2.23]{GLSBS2}. So all that remains is to show that the quivers agree.
	
	The quiver for $\OmegaOp_\br$ is defined in \cite[Section 2.8]{GLSBS2} in terms of 2-forms $\om_{\br, c}$. 
	For $c \in \Jo$, let $i := i_c$. Then $\om_{\br, c}$ is defined as
	\begin{equation}\label{eq:crossing-2-forms}
		(\dlog \grid_{c-1, i} - \frac{1}{2}\dlog \grid_{c, i-1} - \frac{1}{2}\dlog \grid_{c, i+1}) \wedge  (\dlog \grid_{c, i} - \frac{1}{2} \dlog \grid_{c, i-1} - \frac{1}{2} \dlog \grid_{c, i+1})
	\end{equation}
	if $i \in I$ and as the negative of \eqref{eq:crossing-2-forms} if $i \in -I$; see \cite[Equation~(2.21)]{GLSBS2}.%
	 Note that $\om_{\br, c}$ encodes the half-arrows between grid minors shown in \cref{fig:half-arrows-and-minors}: rewriting $\om_{\br, c}$ as a sum over $\dlog \grid_{d, h} \wedge \dlog \grid_{d', h'}$, we see that the negative of the coefficient of each $\dlog \grid_{d, h} \wedge \dlog \grid_{d', h'}$ gives precisely the number of half-arrows from $\Cham_{d, h}$ to $\Cham_{d', h'}$ in \cref{fig:half-arrows-and-minors}.
	
	The quiver of $\OmegaOp_{\br}$ is defined by taking the sum $\omop_{\br}:=-\sum_{c \in \Jo} \om_{\br, c}$, expressing each grid minor as a Laurent monomial in the cluster variables, and then rewriting in the following way: 
	\[\omop_{\br}= \sum_{i < j \in \Jo} B_{ij} \dlog x_i \wedge \dlog x_j.\]
	The number of arrows from $x_i$ to $x_j$ (for at least one of $x_i,x_j$ mutable) is $B_{ij}$. This process is the algebraic counterpart of the combinatorial procedure described in \cref{prop:seedFromLabeledGraphs}. That is, $B_{ij}$ also gives the number of arrows from $x_i$ to $x_j$ in the quiver resulting from the procedure of \cref{prop:seedFromLabeledGraphs}. As the proposition asserts, this quiver is exactly $\Qbr$, the quiver of $\seedbr$.
\end{proof}

We now state the key part of our inductive step.

\begin{theorem}[{\cite[Theorem 4.10]{GLSBS2}}]\label{thm:inductive-step-from-gen-type}
	Suppose $\br=ii \br'$ is a double braid word on positive letters and $\br \geq \wo$. If $\C[\BR_{\br'}] = \A_{\br'}$ and $\C[\BR_{i\br'}] = \A_{i\br'}$
then $\C[\BR_{ \br}]= \A_{\br}$.
\end{theorem} 
\begin{remark}
If one of the first two indices of $\br=ii\br'$ is hollow then $\BR_{\br'}$ is empty and the assumptions of \cref{thm:inductive-step-from-gen-type} only need to be checked for $\BR_{i\br'}$.
\end{remark}
\begin{remark}
In~\cite[Theorem 4.10]{GLSBS2}, we additionally assumed that the quivers $\Qice_{\br'}$ and $\Qice_{i\br'}$ are sink-recurrent and really full rank. These conditions are satisfied by \cref{prop:sinkrecImpliesLocAcyc,prop:full_rank}. Alternatively, these properties can be deduced by induction during the proof of \cref{thm:main} below: if $\Qice_{\br'}$ and $\Qice_{i\br'}$ are sink-recurrent and really full rank then by~\cite[Theorem~4.10, Remark~4.11]{GLSBS2}, so is $\Qice_{ii\br'}$.
\end{remark}

 \begin{proof}[Proof of \cref{thm:main}]
 	We induct on $|\Jo|= m - \ell(\wo)$. The base case, when $|\Jo|=0$, is trivial, as in that case $\BR_{\br}$ is a point.
 	
 	By \cref{thm:invariance}, if $\A_{\br}= \C[\BR_\br]$, then we also have $\A_{\gamma}= \C[\BR_\gamma]$ for all $\gamma$ related to $\br$ by a sequence of braid moves. The argument in the beginning of \cref{thm:sinkrec} shows that every double braid word is related by a sequence of braid moves to a word of the form $\br=ii \br'$ consisting of positive letters. So it suffices to consider a braid word $\br=ii\br'$ with all positive letters.

Assume $\br=ii\br'$. By the inductive hypothesis, we have $\C[\BR_{\br'}] = \A_{\br'}$ and $\C[\BR_{i\br'}] = \A_{i\br'}$. Applying \cref{thm:inductive-step-from-gen-type} gives $\C[\BR_{\br}] = \A_{\br}$.
 \end{proof}

\section{Applications and relations to previous work}\label{sec:applications}

\subsection{Curious Lefschetz}\label{sec:CL}
Let $X$ be a smooth, affine, complex algebraic variety of dimension $d$.  Then the cohomology $H^*(X) = H^*(X,\C)$ has a mixed Hodge structure, which gives a Deligne (weak) splitting $H^k(X) = \bigoplus_{p,q} H^{k,(p,q)}(X)$.  We say that $X$ is of \emph{mixed Tate type} if $H^k(X) = \bigoplus_{p} H^{k,(p,p)}(X)$.  We say that $X$ satisfies the \emph{curious Lefschetz property} if $X$ is of mixed Tate type and there is a class $[\gamma] \in H^{2,(2,2)}(X)$ such that cup product induces isomorphisms
$$
[\gamma]^{d-p}: H^{p+s,(p,p)}(X) \xrightarrow{\sim} H^{2d-p+s,(2d-p,2d-p)}(X)
$$
for all $p$ and $s$.  Curious Lefschetz implies the \emph{curious \Poincare symmetry}
\begin{equation*}%
  \dim H^{p+s,(p,p)}(X) = \dim H^{2d-p+s,(2d-p,2d-p)}(X).
\end{equation*}
In \cite{LS} it is shown that a large class of cluster varieties satisfy curious Lefschetz, and the following result is a consequence.  %

\begin{theorem}\label{thm:CL}
Suppose that $\Qice$ is an ice quiver that is sink-recurrent, full rank, and has an even number of vertices.  Then the cluster variety $X(\Qice)$ satisfies curious Lefschetz.
\end{theorem}
The even-dimensional hypothesis is not a significant restriction since we can always add an extra isolated frozen vertex.  See also \cite[Theorem 8.3]{LS}.

\begin{proof}
Let $X$ be a smooth, affine complex algebraic variety, and $X = U \cup V$ be an open covering of $X$, and $[\gamma] \in H^{2,(2,2)}(X)$.  By \cite[Theorem 3.5]{LS}, if all three $(U,[\gamma]),(V,[\gamma]),(U \cap V,[\gamma])$ satisfy curious Lefschetz then so does $(X,[\gamma])$.  

By \cite[Proposition 8.2]{LS}, full rank isolated cluster varieties satisfy the curious Lefschetz property with respect to a (nearly canonical) \emph{GSV form} $[\gamma]$; cf.~\cite{GSV}.

Let $X = \Spec \A(\Qice)$ be a sink-recurrent, full rank cluster variety and let $X_1 = \Spec \A_1$ and $X_2 = \Spec \A_2$ give an open cover of $X$, as in the proof of \cref{prop:sinkrecImpliesLocAcyc}.  Both $X_1$ and $X_2$ are sink-recurrent, and $X_1 \cap X_2$ is also a cluster localization with mutable quiver $Q-(\NtQ\cup\{s\})$, which is sink recurrent.  Thus, by induction we may assume that $X_1,X_2,X_1 \cap X_2$ all satisfy curious Lefschetz with respect to (the respective restriction of) the GSV form $[\gamma]$ of $X$.  Thus $(X,[\gamma])$ satisfies curious Lefschetz as well.
\end{proof}
We do not know whether the Louise condition \cite{LS} is satisfied for the quivers in this work.  The cohomology of open Richardson varieties are particularly interesting because of the relations to both Category $\mathscr{O}$ and to knot homology~\cite{GL_qtcat}; see also~\cite{GHM}.

\begin{corollary}
Even-dimensional braid Richardson varieties $\BR_{\ubr}$ satisfy the curious Lefschetz property and thus the curious \Poincare symmetry.  Odd-dimensional braid varieties satisfy the curious \Poincare symmetry.
\end{corollary}

\subsection{Braid Richardson links}
By applying \cref{rmk:u=w0} and the moves~\bmref{bm1}--\bmref{bm4}, we reduce to considering pairs $u\leq \br$ such that $\br$ has only positive indices.
 Then we can consider a \emph{braid Richardson link} $L_{u,\br}$ obtained as the braid closure of $\br\cdot \beta(u)^{-1}$,
 where $\br$ is viewed as a positive braid, and $\beta(u)$ is the positive braid lift of $u$. 
 This construction generalizes the \emph{Richardson links} of~\cite{GL_qtcat}.  At the same time, a link $L_G$ inside $\R^3$ can be constructed from a plabic graph $G$; see \cite{STWZ,FPST,GL_plabic}.  The construction can be extended to give a link $L_{\Gubr}$ for our 3D plabic graphs $\Gubr$, using the following convention for bridges and for hollow crossings:
 \begin{center}
   \includegraphics*[width=0.5\textwidth]{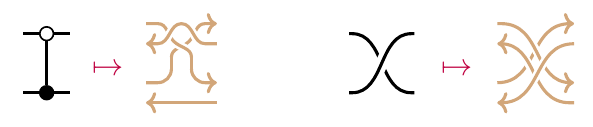}
 \end{center}
The construction in \cite[Section~4.1]{GL_plabic} can be extended to our setting to show that the braid Richardson link $L_{u,\br}$ agrees with the plabic graph link $L_{\Gubr}$.

If we treat the strands of $L_{\Gubr}$ as the boundaries of ribbons, as suggested in the figure above, 
the convention for hollow crossings forces the ribbons to intersect in $\R^3$.
 There is a natural way to alter the surface $\Subr$ locally at hollow crossings to obtain a Seifert surface $\Seif(L_{u,\br})$ for $L_{u,\br}$.  However, $\Subr$ and $\Seif(L_{u,\br})$ are different as abstract surfaces.  Nevertheless, one can draw versions of the relative cycles $(\Cycle_c)_{c\in\Jo}$ on $\Seif(L_{u,\br})$ rather than on $\Subr$, preserving their intersection numbers.

The surface $\Seif(L_{u,\br})$ has the ``wrong" homology, so an alternative is to consider a different embedding $\Subr'$ of the surface $\Subr$ in $\R^4$, defined so that there are no intersections of ribbons at hollow crossings.  %
The surface $\Subr'$ is diffeomorphic to $\Subr$, but it will now have the ``correct" link $L_{\ubr}$ as its boundary in $\R^3$.

\subsection{Point count}
Let $Q$ be a mutable quiver.  Following~\cite{GL_plabic}, define the \emph{point count rational function} $R(Q;q)$ to be the function $q \mapsto \#X(\Qice)({\mathbb F}_q)/(q-1)^a$ where ${\mathbb F}_q$ denotes a finite field with $q$ elements and $\Qice$ is a really full rank ice quiver with mutable part $Q$ and $a$ frozen vertices.  

Suppose $Q$ is sink-recurrent.  Then the covering considered in the proof of \cref{prop:sinkrecImpliesLocAcyc} shows that $R(Q;q)$ is a rational function not depending on the choice of $\Qice$; see~\cite[Section 5.3]{GL_plabic}.  We obtain the following variant of  \cite[Conjecture 2.8]{GL_plabic}.

\begin{theorem}\label{thm:pc}
Let $u\leq \br$, and suppose that the 3D plabic graph $\Gubr$ has $\conn(\Gubr)$ connected components. Then we have
\begin{equation*}%
  R(\Qubrmut;q)=(q-1)^{\conn(\Gubr)-1} \Ptop_{L_{u,\br}}(q),
\end{equation*}
where $\Ptop_{L_{u,\br}}(q)$ is obtained from the top $a$-degree term of the \emph{HOMFLY polynomial $\HOMP(L_{u,\br};a,z)$} by substituting $a:=\qqi$ and $z:=\qq-\qqi$.
\end{theorem}
\begin{proof}
By \cref{thm:sinkrec,prop:full_rank} and our main result \cref{thm:main}, it follows that $R(\Qubrmut;q)$ is equal to $(q-1)^{-a} \#\BR_{u,\br}({\mathbb F}_q)$, where $a$ is the number of frozen variables in $\Qubr$.  A recursion for the point count of $\BR_{\ubr}$ is given in \cite[Section 5]{GLTW}.  The comparison with the top $a$-degree term of the HOMFLY polynomial proceeds in the same way as the proof of \cite[Theorem 2.1]{GL_qtcat}.
\end{proof}
The special case of \cref{thm:pc} when $\br$ is a reduced word for some permutation $w\in\Sn$ follows from our main results combined with~\cite[Theorem~1.11]{GL_qtcat}.  See~\cite{GL_qtcat,GL_cat_combin,GL_plabic} for further details and examples.

\subsection{Relation of 3D plabic graphs to Postnikov's plabic graphs}\label{sec:postn-plab-graphs}
We explain how our 3D plabic graphs generalize the plabic graphs of~\cite{Pos}. A permutation $w\in\Sn$ is called \emph{$k$-Grassmannian} if $w(1)<\dots<w(n-k)$ and $w(n-k+1)<\dots<w(n)$. We denote by $\Skn$ the set of $k$-Grassmannian permutations. Thus, any reduced word for $w\in\Skn$ ends with $s_{n-k}$, and all reduced words for $w$ are related by commutation moves. It is well known that $k$-Grassmannian permutations are in bijection with Young diagrams that fit inside a $k\times(n-k)$ rectangle. We draw Young diagrams in English notation.
 A \emph{Le-diagram} is a way of placing a dot in some of the boxes of $\la$ so that whenever a box of $\la$ has a dot above it in the same column and to the left of it in the same row, it must also contain a dot. We will find it useful to note the contrapositive formulation: If a box of $\la$ is empty, then either all the boxes to its left are empty, or all the boxes above it are empty (or both).
 
  For each dot, we draw horizontal and vertical line segments connecting it to the southeastern boundary of $\la$. An example of a Le-diagram $\Gamma$ is shown in \figref{fig:Guw-Le}(left). One can convert a Le-diagram $\Gamma$ into a plabic graph $G(\Gamma)$ using the local rule \begin{tikzpicture}[baseline=(Z.base)]
   \coordinate(Z) at (0,0);
\node(A) at (0,0.05){\igrw[0.15]{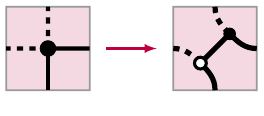}};
 \end{tikzpicture}. The graph $G(\Gamma)$ is drawn in a disk which we identify with the Young diagram $\la$. Thus, $G(\Gamma)$ has $n$ boundary vertices on the southeastern boundary of $\la$, and it has a unique northwestern boundary face which we denote $F_0$. Rotating $\Gamma$ by $135^\circ$ clockwise, replacing all empty boxes with crossings, and all dots with black-white bridges, we obtain a graph $\Guw$, where $\bw$ is a reduced word for $w$ and $u\leq w$ in the Bruhat order on $\Sn$; see \cref{fig:Guw-Le}.

\begin{figure}
\setlength{\tabcolsep}{-5pt}
\begin{tabular}{cc}
\begin{tikzpicture}[baseline=(Z.base)]
\coordinate(Z) at (0,0);
\node(A) at (0,-0.12){\igrw[0.55]{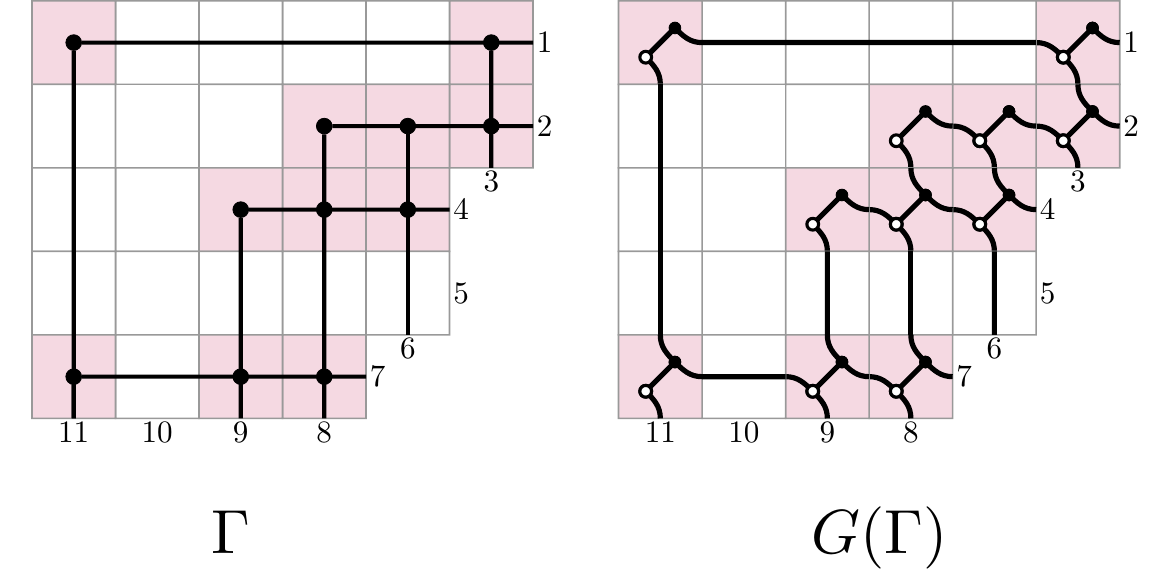}};
\end{tikzpicture}
  &
\begin{tikzpicture}[baseline=(Z.base)]
\coordinate(Z) at (0,0);
\node(A) at (0,0){\igrw[0.45]{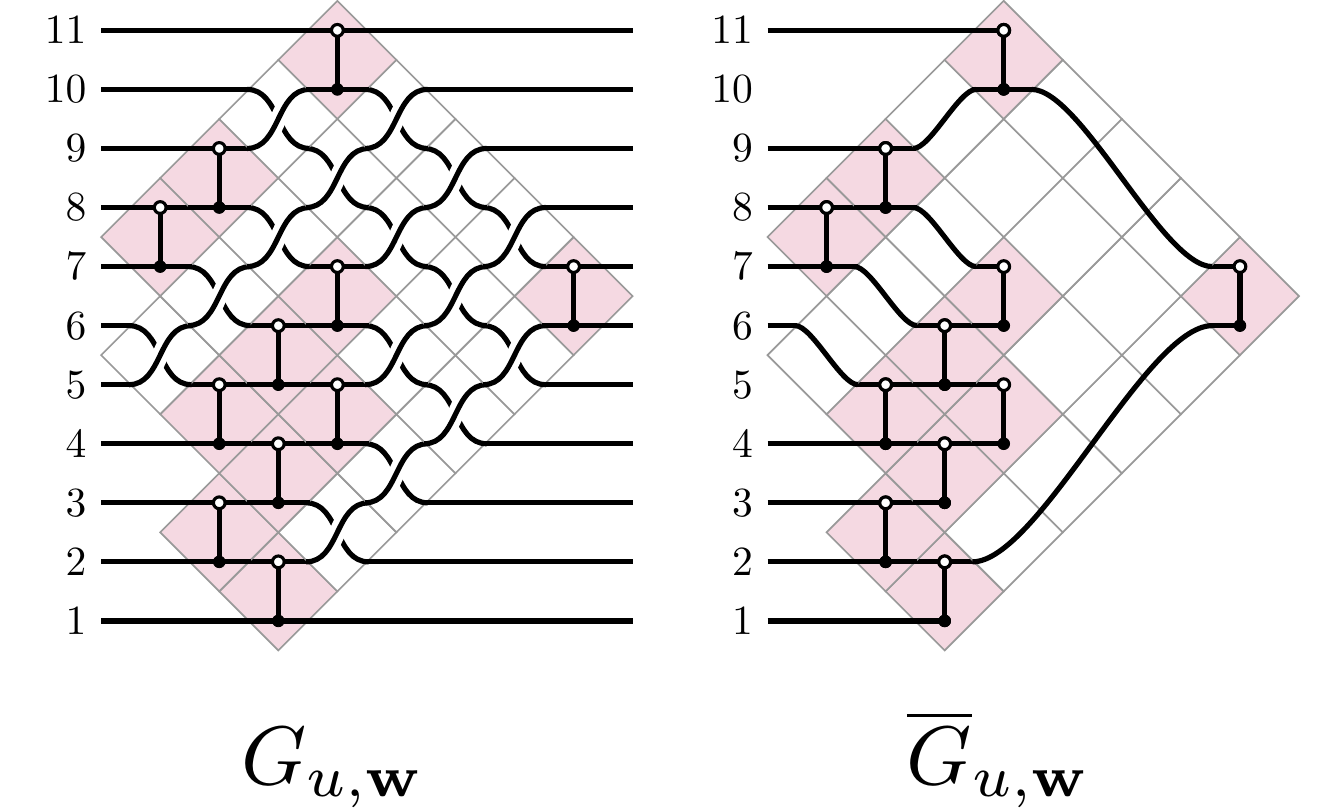}};
\end{tikzpicture}
\end{tabular}
  \caption{\label{fig:Guw-Le} Left: converting a Le-diagram $\Gamma$ into a plabic graph $G(\Gamma)$. Right: the graphs $\Guw$ and $\Guwar$ for $(u,w)$ corresponding to $\Gamma$.}
\end{figure}

\begin{proposition}
Let $\Gamma$ be a Le-diagram and let $u\leq w$ be the corresponding pair of permutations. Then we have 
\begin{equation*}%
G(\Gamma)=\Guwar.  
\end{equation*}
 Moreover, each relative cycle $\Cycle_c$ of $\Guwar$ traverses the boundary of a face of $G(\Gamma)$ in the counterclockwise direction, and this gives a bijection between the relative cycles $(\Cycle_c)_{c\in\Jow}$ and all faces of $G(\Gamma)$ except $F_0$.
\end{proposition}
\begin{proof}
From the contrapositive formula of the definition of a Le-diagram, for any hollow crossing $c$ of $\Guw$, at least one of the strands passing through $c$ does not participate in any solid crossings to the right of $c$.
 In particular, all hollow crossings disappear when we pass from $\Guw$ to $\Guwar$, and thus it follows that $G(\Gamma)=\Guwar$; see also~\cite[Figure~5]{Karpman2} and~\cite[Figure~7]{GL_cluster}. An example is given in \figref{fig:Guw-Le}(right).

To compute the relative cycles in $\Guw$, we may use the propagation rules in \cref{fig:intro:propag}. From here, the statement that the relative cycles correspond to the faces of $G(\Gamma)$ follows immediately.
\end{proof}

\subsection{Comparison with previously defined chamber minors for Richardson and double Bruhat cells}\label{sec:comparison} %
As we mentioned in the introduction, braid varieties include double Bruhat cells, open positroid varieties, and open Richardson varieties, cluster structures on which have been studied previously in many works including~\cite{FZ_double,GY,Scott,SSBW,GL_cluster,Lec,Ing}. We briefly explain how our chamber minors relate to chamber minors defined in some of these previous works. %

Let $g\mapsto g^{-\iota}$ be the involutive automorphism of $G$ defined by
\begin{equation*}%
	x_i(t)\mapsto x_i(-t),\quad y_i(t)\mapsto y_i(-t),\quad \ds_i\mapsto \ds_i^{-1},\quad h\mapsto h,
\end{equation*}
for all $i\in I$, $t\in \C$, and $h\in \H$. This map is a composition of the involution $g\mapsto g^\iota$ studied in~\cite[Section~2.1]{FZ_double} with the involution $g\mapsto g^{-1}$ (these two involutions commute). The properties of the involution $g\mapsto g^{-\iota}$ in relation to total positivity were first studied in~\cite[Section~6.2]{GL_twist}. Since $G=\SL_n$, one can check that for a matrix $g=(g_{i,j})_{i,j\in[n]}$, we have $g^{-\iota}=((-1)^{i+j}g_{i,j})_{i,j\in[n]}$. 

We will use the following relations for relative positions of weighted flags:
\begin{equation}\label{eq:MR_move}
	U_+\Lrel{\textcolor{red}{s_{i}}} \ds_i^{-1}U_+, \quad \alphacheck_i(1/t)U_+ \Lrel{\textcolor{red}{s_{i}}} y_i(-t)U_+, \quad \alphacheck_{i^\ast}(t) \d\wo ^{-1}U_+ \Rrel{\textcolor{blue}{s_{i^\ast}}} x_{i^\ast}(-t) \d\wo ^{-1}U_+,
\end{equation}
for all $t\in \Cx$ and $i\in I$, where $\alphacheck_i$ was defined in~\eqref{eq:gen_dfn}.

\subsubsection{Open Richardson varieties}\label{sec:gracie_compare}
Observe that the map $gB_+\mapsto g^{-\iota}B_+$ preserves the subsets $B_-uB_+$, $B_+wB_+$, and therefore yields an involutive automorphism of $\Rich_u^w$. Choose a reduced word\footnote{We denote the reduced word for $w$ by $\br$ as opposed to $\bw$ in order to make the visual differences between the varieties $\Rich_u^w$ and $\BR_{\ubr}$ more apparent.} $\br$ for $w$ and consider the isomorphism between an open Richardson and a braid Richardson variety
\begin{equation}\label{eq:Rich_BR_iota}
	\Rich_u^w\xrasim \BR_{\ubr},\quad gB_+\mapsto (\Xbul,\Ybul),
\end{equation}
where $(\Xbul,\Ybul)$ satisfies the following conditions:
\begin{equation*}%
	\pi(X_0)=B_+,\quad \pi(X_m)=g^{-\iota}B_+, \quad\text{and}\quad Y_0=Y_1=\cdots =Y_m=\d\wo ^{-1} U_+.
\end{equation*}
As explained in the proof of \cref{prop:braidRich}, these conditions determine the tuple $(\Xbul,\Ybul)\in\BR_{\ubr}$ uniquely. We explain how to do this explicitly when the element $gB_+$ is \emph{MR-parametrized}; such parametrizations were introduced in~\cite{MR} in relation to total positivity for flag varieties~\cite{Lus2}. Let
\begin{equation}\label{eq:MR_param}
	g:= g_1\cdots g_m,
	\qquad \text{where} \qquad
	g_c = \begin{cases} \ds_{i_c} & \mbox{if $c \notin \Jo$,} \\
		y_{i_c}(t_c) &\mbox{if $c \in \Jo$.} \end{cases}
\end{equation}
Here $\bt=(t_1,\dots,t_m)$ consists of some nonzero parameters. We get
\begin{equation*}%
	g^{-\iota}=g_1^{-\iota}\cdots g_m^{-\iota},
	\qquad \text{where} \qquad
	g_c^{-\iota} = \begin{cases} \ds_{i_c}^{-1} & \mbox{if $c \notin \Jo$,} \\
		y_{i_c}(-t_c) &\mbox{if $c \in \Jo$.} \end{cases}
\end{equation*}
We first check that we may set $X_m:=g^{-\iota}U_+$, i.e., that $X_m \Lrel{\wo  u} Y_m$. Indeed, it is well known~\cite[Equation~(2.8)]{GL_twist} that $g\in U_-\du$. Thus, $g^{-\iota}\in U_- \ddot{u}$, where for a reduced word $u=s_{j_1}\cdots s_{j_l}$, we set 
$\ddot{u}:=\du^{-\iota}=((u^{-1})^{\bigcdot})^{-1}=\ds_{j_1}^{-1}\cdots \ds_{j_l}^{-1}$. The element $\ddot{u}$ satisfies $\d\wo \ddot{u}=(\wo u)^{\bigcdot}$. It follows that $\d\wo \cdot (X_m,Y_m)=(\d\wo g^{-\iota}U_+,U_+)$, where $\d\wo g^{-\iota}U_+\in U_+(\wo u)^{\bigcdot}U_+$, and thus indeed $X_m \Lrel{\wo  u} Y_m$. We now may compute $X_{m-1},\dots,X_0$ iteratively using $X_{c-1} \Lrel{\textcolor{red}{s_{i_c}}} X_c$ together with the relations in~\eqref{eq:MR_move}.  Comparing our \cham minors with the chamber minors of~\cite{Ing}, we arrive at the following result.
\begin{proposition}\label{prop:Ing_iso}
	The isomorphism~\eqref{eq:Rich_BR_iota} sends the chamber minors of~\cite{Ing} to the \cham minors $(\crossing{c})_{c\in \Jo}$ from \cref{dfn:cham}. In particular, all \cham and grid minors from \cref{dfn:grid_minors,dfn:cham} take positive values on the image of the totally positive part $\Rtp_u^w\subset\Rich_u^w$ under~\eqref{eq:Rich_BR_iota}. %
\end{proposition}
\begin{remark}
	The fact that the red grid minors of $Z_c$ take positive values when restricted to the subset $\Rtp_u^w\subset\Rich_u^w$ is a reflection of the fact that the \emph{reversal map} $gB_+\mapsto \d\wo g^{-\iota}B_+$ preserves total positivity; see~\cite[Proposition~6.4]{GL_twist}.
\end{remark}
\begin{remark}
	For a comparison between the chamber minors of~\cite{Ing,Lec,MR}, see~\cite[Section~11]{GL_twist}.
\end{remark}

\begin{remark}%
  Let $\br=\bw$ be a reduced word for $w\in\Sn$. Comparing the half-arrow description in \cref{prop:seedFromLabeledGraphs} to~\cite[Definition~VII.2]{Ing} and applying \cref{prop:Ing_iso}, it follows that the isomorphism~\eqref{eq:Rich_BR_iota} sends Ingermanson's seed to our seed $\seedubr$. Thus, our cluster structure on the open Richardson variety $\Rich_u^w$ recovers the upper cluster structure on $\Rich_u^w$ constructed in~\cite{Ing}.
\end{remark}

\begin{figure}
	\begin{tabular}{c}
		\igrw[1.0]{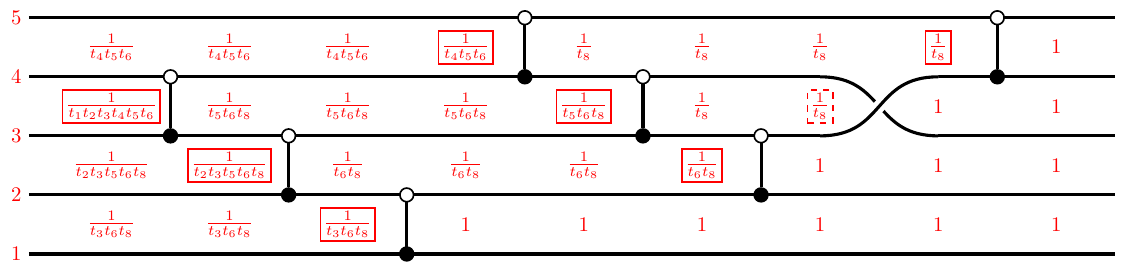}\\
		(a) open Richardson variety (\cref{ex:grid_gracie})\\[10pt]
		\igrw[1.0]{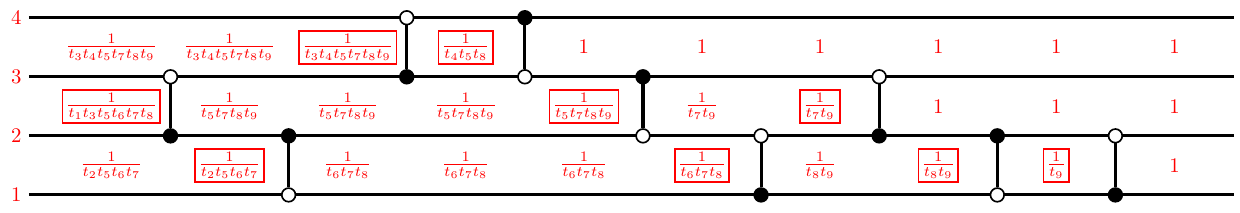}\\
		(b) double Bruhat cell (\cref{ex:grid_FZ})
	\end{tabular}
	\caption{\label{fig:grid_vs} Comparing grid minors to chamber minors of~\cite{Ing,FZ_double}.}
\end{figure}
\begin{example}\label{ex:grid_gracie}
	Consider the running example of~\cite{Ing}: $w=s_3s_2s_1s_4s_3s_2s_3s_4$ and $u=s_3$, which was already considered in the introduction (\cref{fig:gracie_intro}). Applying the parametrization~\eqref{eq:MR_param} and computing $(\Xbul,\Ybul)$ from it using the above algorithm, we arrive at the following sequence of matrices $Z_c=Y_c^{-1}X_c$:
	\def\XYscl{0.768}
	\begin{equation*}%
		\text{\scalebox{\XYscl}{$Z_{0}=\smat{0 & 0 & 0 & 0 & t_{4} t_{5} t_{6} \\
					0 & 0 & 0 & -t_{1} t_{2} t_{3} & 0 \\
					0 & 0 & \frac{t_{8}}{t_{1} t_{4}} & 0 & 0 \\
					0 & -\frac{1}{t_{2} t_{5}} & 0 & 0 & 0 \\
					\frac{1}{t_{3} t_{6} t_{8}} & 0 & 0 & 0 & 0}$,}}\quad
		\text{\scalebox{\XYscl}{$Z_{1}=\smat{0 & 0 & 0 & 0 & t_{4} t_{5} t_{6} \\
					0 & 0 & t_{1} t_{2} t_{3} & -\frac{t_{8}}{t_{4}} & 0 \\
					0 & 0 & t_{2} t_{3} & 0 & 0 \\
					0 & -\frac{1}{t_{2} t_{5}} & 0 & 0 & 0 \\
					\frac{1}{t_{3} t_{6} t_{8}} & 0 & 0 & 0 & 0}$,}}
		\quad \dots, \quad
		\text{\scalebox{\XYscl}{$Z_{8}=\smat{0 & -t_{4} t_{5} t_{6} & t_{4} & t_{4} t_{5} - t_{8} & 1 \\
					t_{1} t_{2} t_{3} & -t_{1} t_{2} - t_{1} t_{6} - t_{5} t_{6} & 1 & t_{1} + t_{5} & 0 \\
					t_{2} t_{3} & -t_{2} - t_{6} & 0 & 1 & 0 \\
					t_{3} & -1 & 0 & 0 & 0 \\
					1 & 0 & 0 & 0 & 0}$.}}
	\end{equation*}
	The associated grid minors are computed in \figref{fig:grid_vs}(a). Applying the monomial transformation in~\cite[Example~11.7]{GL_twist}, we see that these minors coincide with the ones given in~\cite[Figure~7.8]{Ing}.
\end{example}

\subsubsection{Double Bruhat cells}
For $w,v\in \Sn$, consider the \emph{double Bruhat cell} $\Gvw$ and the \emph{reduced double Bruhat cell} $\Lvw$ defined by
\begin{equation*}%
	\Gvw:=B_+wB_+\cap B_-vB_- \quad\text{and}\quad \Lvw:=\Gvw/\H.
\end{equation*}
Let $\br=(i_1,i_2,\dots,i_m)$, where $m=\ell(w)+\ell(v)$, be a double reduced word for $(w,v)$; that is, it is a shuffle of a reduced word for $w$ on positive indices and a reduced word for $v$ on negative indices. 
The following map is an isomorphism by an argument similar to~\cite[Proposition~2.1]{WY}:
\begin{equation}\label{eq:Bruh_to_BR}
	\Lvw\xrasim \BR_{\id,\br},\quad g\mapsto (\Xbul,\Ybul),
\end{equation}
where $(\Xbul,\Ybul)$ satisfies the following conditions:
\begin{equation*}%
	\pi(X_0)=B_+,\quad X_m=g^{-\iota}U_+, \quad \pi(Y_0)=B_-,\quad \text{and}\quad Y_m=g^{-\iota}\d\wo ^{-1} U_+.
\end{equation*}
Applying the relations~\eqref{eq:MR_move}, we compute $X_{m-1},\dots,X_0$ and $Y_{m-1},\dots,Y_0$ iteratively. In order to compare our minors to the minors considered in~\cite{FZ_double,BFZ}, let us choose a parametrization
\begin{equation*}%
	g:= g_1\cdots g_m,
	\qquad \text{where} \qquad
	g_c = \begin{cases} y_{i_c}(t_c) & \mbox{if $i_c>0$,} \\
		x_{|i_c|}(t_c) &\mbox{if $i_c<0$.} \end{cases}
\end{equation*}
The nonzero parameters $t_1,\dots,t_m$ are expressed as monomials in the cluster variables of~\cite{FZ_double,BFZ} computed on the \emph{twisted matrix} $\xivw(g)$; see~\cite[Definition~1.5]{FZ_double}. Given a matrix $x\in B_-B_+$, let us denote by $([x]_-,[x]_0,[x]_+)\in U_-\times \H\times U_+$ its LDU factorization. One can check that for any $g\in \Gvw$, the matrix $g_0:=[\dw^{-1}g]_0$ is well defined.

It follows from the definition of the map~\eqref{eq:Bruh_to_BR} that we have $Z_c\in \H \d\wo$ for each $c\in[0,m]$. In particular, the red and the blue grid minors coincide: $\Delta_{c,h}=\Delta_{c,-h}$ for all $c,h$. We leave the verification of the following result to an interested reader.
\begin{proposition}
	Under the isomorphism~\eqref{eq:Bruh_to_BR}, the grid minors $\Delta_{c,h}$ are equal to the chamber minors of~\cite[Section~4.5]{FZ_double} evaluated at $g_0\xivw(g)$. In particular, all \cham and grid minors from \cref{dfn:grid_minors,dfn:cham} take positive values on the image of the totally positive part $\Gvwtp\subset\Gvw$ under~\eqref{eq:Bruh_to_BR}.
\end{proposition}

\begin{example}\label{ex:grid_FZ}
	We consider the running example of~\cite[Section~4.5]{FZ_double}, except that we ignore the $\H$-part factors in their decomposition (marked by green points in~\cite[Figure~4]{FZ_double}). Thus, we have
	\begin{equation*}%
		\beta=(2,-1,3,-3,-2,1,2,-1,1),\quad w=s_2s_3s_1s_2s_1,\quad v=s_1s_3s_2s_1,\quad u=\id.
	\end{equation*}
	The matrices $g$, $g_0$, and $g_0\xivw(g)$ are given by
	\def\XYscl{0.768}
	\begin{align*}%
		\text{\scalebox{\XYscl}{$g= \smat{
					t_{2} t_{5} t_{7} t_{9} + t_{2} t_{6} t_{8} t_{9} + t_{2} t_{6} + t_{2} t_{9} + t_{8} t_{9} + 1 & t_{2} t_{5} t_{7} + t_{2} t_{6} t_{8} + t_{2} + t_{8} & t_{2} t_{5} & 0 \\
					t_{5} t_{7} t_{9} + t_{6} t_{8} t_{9} + t_{6} + t_{9} & t_{5} t_{7} + t_{6} t_{8} + 1 & t_{5} & 0 \\
					t_{1} t_{5} t_{7} t_{9} + t_{1} t_{6} t_{8} t_{9} + t_{1} t_{6} + t_{1} t_{9} + t_{7} t_{9} & t_{1} t_{5} t_{7} + t_{1} t_{6} t_{8} + t_{1} + t_{7} & t_{1} t_{5} + 1 & t_{4} \\
					t_{3} t_{7} t_{9} & t_{3} t_{7} & t_{3} & t_{3} t_{4} + 1
				}$,}}
		\quad
		\text{\scalebox{\XYscl}{$g_0= \smat{
					t_{3} t_{7} t_{9} & 0 & 0 & 0 \\
					0 & \frac{t_{1} t_{6}}{t_{9}} & 0 & 0 \\
					0 & 0 & \frac{1}{t_{6} t_{7}} & 0 \\
					0 & 0 & 0 & \frac{1}{t_{1} t_{3}}
				}$,}}\quad\text{and}\\
		\text{\scalebox{\XYscl}{$g_0\xivw(g)= \smat{
					\frac{t_{2} t_{5} t_{7} t_{9} + t_{2} t_{6} t_{8} t_{9} + t_{2} t_{6} + t_{2} t_{9} + t_{8} t_{9} + 1}{t_{2} t_{5}} & \frac{t_{8} t_{9} + 1}{t_{8}} & \frac{1}{t_{4} t_{5} t_{8}} & 1 \\
					\frac{t_{2} t_{6} + 1}{t_{2} t_{5} t_{9}} & \frac{1}{t_{8} t_{9}} & \frac{1}{t_{4} t_{5} t_{8} t_{9}} & \frac{1}{t_{9}} \\
					\frac{1}{t_{2} t_{5} t_{6} t_{7}} & \frac{1}{t_{6} t_{7} t_{8}} & \frac{t_{6} t_{8} + 1}{t_{4} t_{5} t_{6} t_{7} t_{8}} & \frac{t_{6} t_{8} + 1}{t_{6} t_{7}} \\
					0 & \frac{t_{2}}{t_{1} t_{3} t_{8}} & \frac{t_{1} t_{5} t_{8} + t_{2} t_{6} t_{8} + t_{2} + t_{8}}{t_{1} t_{3} t_{4} t_{5} t_{8}} & \frac{t_{1} t_{3} t_{4} t_{5} t_{8} + t_{1} t_{5} t_{8} + t_{2} t_{6} t_{8} + t_{2} + t_{8}}{t_{1} t_{3}}
				}$.}}
	\end{align*}
	One can check that the chamber minors from~\cite[Figure~6]{FZ_double} computed on the matrix $g_0\xivw(g)$ coincide with the red grid minors computed in \figref{fig:grid_vs}(b).
\end{example}
\begin{remark}
	The chamber minors of~\cite{FZ_double} computed on the matrix $\xivw(g)$ have a description in terms of strands in a double wiring diagram; see~\cite[Theorem~4.11]{FZ_double}. If one computes them on the matrix $g_0\xivw(g)$ instead, one gets a similar description: each chamber minor equals $\left(\prod t_k\right)^{-1}$, where the product is taken over \emph{all} crossings which are to the right of the chamber and such that the chamber is located vertically between the two strands participating in the crossing. In other words, the transformation $\xivw(g)\mapsto g_0\xivw(g)$ gauge-fixes to $1$ the chamber minors corresponding to the chambers open on the right.
\end{remark}

\bibliographystyle{alpha}
\bibliography{type_a}

\newcommand{\etalchar}[1]{$^{#1}$}
\begin{thebibliography}{GLTW22}

\bibitem[BFZ05]{BFZ}
Arkady Berenstein, Sergey Fomin, and Andrei Zelevinsky.
\newblock Cluster algebras. {III}. {U}pper bounds and double {B}ruhat cells.
\newblock {\em Duke Math. J.}, 126(1):1--52, 2005.

\bibitem[BGY06]{BGY}
K.~A. Brown, K.~R. Goodearl, and M.~Yakimov.
\newblock Poisson structures on affine spaces and flag varieties. {I}. {M}atrix
  affine {P}oisson space.
\newblock {\em Adv. Math.}, 206(2):567--629, 2006.

\bibitem[CGG{\etalchar{+}}22]{CGGLSS}
Roger Casals, Eugene Gorsky, Mikhail Gorsky, Ian Le, Linhui Shen, and José
  Simental.
\newblock {Cluster structures on braid varieties}.
\newblock {\em \arxiv{2207.11607v1}}, 2022.

\bibitem[CGGS20]{CGGS}
Roger Casals, Eugene Gorsky, Mikhail Gorsky, and José Simental.
\newblock {Algebraic Weaves and Braid Varieties}.
\newblock {\em \arxiv{2012.06931v1}}, 2020.

\bibitem[CGGS21]{CGGS2}
Roger Casals, Eugene Gorsky, Mikhail Gorsky, and José Simental.
\newblock {Positroid Links and Braid varieties}.
\newblock {\em \arxiv{2105.13948v1}}, 2021.

\bibitem[CK22]{CaoKel}
Peigen Cao and Bernhard Keller.
\newblock {On Leclerc's conjectural cluster structures for open Richardson
  varieties}.
\newblock {\em \arxiv{2207.10184v1}}, 2022.

\bibitem[CW22]{CaWe}
Roger Casals and Daping Weng.
\newblock {Microlocal Theory of Legendrian Links and Cluster Algebras}.
\newblock {\em \arxiv{2204.13244v2}}, 2022.

\bibitem[Deo85]{Deo}
Vinay~V. Deodhar.
\newblock On some geometric aspects of {B}ruhat orderings. {I}. {A} finer
  decomposition of {B}ruhat cells.
\newblock {\em Invent. Math.}, 79(3):499--511, 1985.

\bibitem[FG06]{FoGo_moduli}
Vladimir Fock and Alexander Goncharov.
\newblock Moduli spaces of local systems and higher {T}eichm\"{u}ller theory.
\newblock {\em Publ. Math. Inst. Hautes \'{E}tudes Sci.}, (103):1--211, 2006.

\bibitem[FG09]{FoGo_cluster}
Vladimir~V. Fock and Alexander~B. Goncharov.
\newblock Cluster ensembles, quantization and the dilogarithm.
\newblock {\em Ann. Sci. \'{E}c. Norm. Sup\'{e}r. (4)}, 42(6):865--930, 2009.

\bibitem[FPST22]{FPST}
Sergey Fomin, Pavlo Pylyavskyy, Eugenii Shustin, and Dylan Thurston.
\newblock Morsifications and mutations.
\newblock {\em J. Lond. Math. Soc. (2)}, 105(4):2478--2554, 2022.

\bibitem[Fra16]{Fraser}
Chris Fraser.
\newblock Quasi-homomorphisms of cluster algebras.
\newblock {\em Adv. in Appl. Math.}, 81:40--77, 2016.

\bibitem[FZ99]{FZ_double}
Sergey Fomin and Andrei Zelevinsky.
\newblock Double {B}ruhat cells and total positivity.
\newblock {\em J. Amer. Math. Soc.}, 12(2):335--380, 1999.

\bibitem[FZ02]{FZ}
Sergey Fomin and Andrei Zelevinsky.
\newblock Cluster algebras. {I}. {F}oundations.
\newblock {\em J. Amer. Math. Soc.}, 15(2):497--529 (electronic), 2002.

\bibitem[Gal23]{GalTut}
Pavel Galashin.
\newblock Braid variety cluster structures program.
\newblock {\em
  \textup{\url{https://www.math.ucla.edu/~galashin/DoubleBraidCluster/tutorial.html}}},
  2023.

\bibitem[GHM21]{GHM}
Eugene Gorsky, Matthew Hogancamp, and Anton Mellit.
\newblock {Tautological classes and symmetry in Khovanov--Rozansky homology}.
\newblock {\em \arxiv{2103.01212v1}}, 2021.

\bibitem[GK13]{GoKe}
Alexander~B. Goncharov and Richard Kenyon.
\newblock Dimers and cluster integrable systems.
\newblock {\em Ann. Sci. \'{E}c. Norm. Sup\'{e}r. (4)}, 46(5):747--813, 2013.

\bibitem[GL19]{GL_cluster}
Pavel Galashin and Thomas Lam.
\newblock {Positroid varieties and cluster algebras}.
\newblock {\em Ann. Sci. Éc. Norm. Supér., \normalfont{to appear.}
  \arxiv{1906.03501v2}}, 2019.

\bibitem[GL20]{GL_qtcat}
Pavel Galashin and Thomas Lam.
\newblock {Positroids, knots, and $q,t$-Catalan numbers}.
\newblock {\em \arxiv{2012.09745v2}}, 2020.

\bibitem[GL21]{GL_cat_combin}
Pavel Galashin and Thomas Lam.
\newblock {Positroid Catalan numbers}.
\newblock {\em \arxiv{2104.05701v1}}, 2021.

\bibitem[GL22a]{GL_plabic}
Pavel Galashin and Thomas Lam.
\newblock {Plabic links, quivers, and skein relations}.
\newblock {\em \arxiv{2208.01175v1}}, 2022.

\bibitem[GL22b]{GL_twist}
Pavel Galashin and Thomas Lam.
\newblock {The twist for Richardson varieties}.
\newblock {\em \arxiv{2204.05935v1}}, 2022.

\bibitem[GLSB23]{GLSBS2}
Pavel Galashin, Thomas Lam, and Melissa Sherman-Bennett.
\newblock Braid variety cluster structures, {II}: general type.
\newblock {\em \arxiv{2301.07268}}, 2023.

\bibitem[GLTW22]{GLTW}
Pavel Galashin, Thomas Lam, Minh-Tam~Quang Trinh, and Nathan Williams.
\newblock {Rational Noncrossing Coxeter-Catalan Combinatorics}.
\newblock {\em \arxiv{2208.00121v1}}, 2022.

\bibitem[GSV10]{GSV}
Michael Gekhtman, Michael Shapiro, and Alek Vainshtein.
\newblock {\em Cluster algebras and {P}oisson geometry}, volume 167 of {\em
  Mathematical Surveys and Monographs}.
\newblock American Mathematical Society, Providence, RI, 2010.

\bibitem[GY20]{GY}
K.~R. Goodearl and M.~T. Yakimov.
\newblock The {B}erenstein--{Z}elevinsky quantum cluster algebra conjecture.
\newblock {\em J. Eur. Math. Soc. (JEMS)}, 22(8):2453--2509, 2020.

\bibitem[hes]{MO418142}
David E~Speyer (https://mathoverflow.net/users/297/david-e speyer).
\newblock Coordinates on $n_+ \backslash \overline{B_+ w B_+} / n_+$.
\newblock MathOverflow.
\newblock URL:https://mathoverflow.net/q/418142 (version: 2022-03-14).

\bibitem[Ing19]{Ing}
Grace Ingermanson.
\newblock {\em Cluster {A}lgebras of {O}pen {R}ichardson {V}arieties}.
\newblock ProQuest LLC, Ann Arbor, MI, 2019.
\newblock Thesis (Ph.D.)--University of Michigan.

\bibitem[Kar16]{Karpman2}
Rachel Karpman.
\newblock Bridge graphs and {D}eodhar parametrizations for positroid varieties.
\newblock {\em J. Combin. Theory Ser. A}, 142:113--146, 2016.

\bibitem[Kar18]{Karpman}
Rachel Karpman.
\newblock Total positivity for the {L}agrangian {G}rassmannian.
\newblock {\em Adv. in Appl. Math.}, 98:25--76, 2018.

\bibitem[KLS13]{KLS}
Allen Knutson, Thomas Lam, and David~E. Speyer.
\newblock Positroid varieties: juggling and geometry.
\newblock {\em Compos. Math.}, 149(10):1710--1752, 2013.

\bibitem[Lec16]{Lec}
B.~Leclerc.
\newblock Cluster structures on strata of flag varieties.
\newblock {\em Adv. Math.}, 300:190--228, 2016.

\bibitem[LS16]{LS}
Thomas Lam and David~E. Speyer.
\newblock {Cohomology of cluster varieties. I. Locally acyclic case}.
\newblock {\em \arxiv{1604.06843v1}}, 2016.

\bibitem[Lus94]{Lus2}
G.~Lusztig.
\newblock Total positivity in reductive groups.
\newblock In {\em Lie theory and geometry}, volume 123 of {\em Progr. Math.},
  pages 531--568. Birkh\"auser Boston, Boston, MA, 1994.

\bibitem[Lus98]{Lus98}
G.~Lusztig.
\newblock Total positivity in partial flag manifolds.
\newblock {\em Represent. Theory}, 2:70--78, 1998.

\bibitem[Mel19]{Mellit_cell}
Anton Mellit.
\newblock {Cell decompositions of character varieties}.
\newblock {\em \arxiv{1905.10685v1}}, 2019.

\bibitem[MR04]{MR}
R.~J. Marsh and K.~Rietsch.
\newblock Parametrizations of flag varieties.
\newblock {\em Represent. Theory}, 8:212--242, 2004.

\bibitem[MS16]{MSLA}
Greg Muller and David~E. Speyer.
\newblock Cluster algebras of {G}rassmannians are locally acyclic.
\newblock {\em Proc. Amer. Math. Soc.}, 144(8):3267--3281, 2016.

\bibitem[Mul13]{Mul}
Greg Muller.
\newblock Locally acyclic cluster algebras.
\newblock {\em Adv. Math.}, 233:207--247, 2013.

\bibitem[Mul14]{MullerLA2}
Greg Muller.
\newblock {$\mathcal{A}=\mathcal{U}$} for locally acyclic cluster algebras.
\newblock {\em SIGMA Symmetry Integrability Geom. Methods Appl.}, 10:Paper 094,
  8, 2014.

\bibitem[Mé22]{Menard}
Etienne Ménard.
\newblock {Cluster algebras associated with open Richardson varieties: an
  algorithm to compute initial seeds}.
\newblock {\em \arxiv{2201.10292v1}}, 2022.

\bibitem[Pos06]{Pos}
Alexander Postnikov.
\newblock Total positivity, {G}rassmannians, and networks.
\newblock {\em \textup{Preprint,
  \href{http://math.mit.edu/~apost/papers/tpgrass.pdf}{\texttt{http://
  math.mit.edu/\textasciitilde apost/papers/tpgrass.pdf}}}}, 2006.

\bibitem[Sco06]{Scott}
J.~S. Scott.
\newblock Grassmannians and cluster algebras.
\newblock {\em Proc. London Math. Soc. (3)}, 92(2):345--380, 2006.

\bibitem[SSB22]{SSB}
Khrystyna Serhiyenko and Melissa Sherman-Bennett.
\newblock Leclerc's conjecture on a cluster structure for type {A} {R}ichardson
  varieties.
\newblock {\em arXiv:2210.13302}, 2022.

\bibitem[SSBW19]{SSBW}
K.~Serhiyenko, M.~Sherman-Bennett, and L.~Williams.
\newblock Cluster structures in {S}chubert varieties in the {G}rassmannian.
\newblock {\em Proc. Lond. Math. Soc. (3)}, 119(6):1694--1744, 2019.

\bibitem[STWZ19]{STWZ}
Vivek Shende, David Treumann, Harold Williams, and Eric Zaslow.
\newblock Cluster varieties from {L}egendrian knots.
\newblock {\em Duke Math. J.}, 168(15):2801--2871, 2019.

\bibitem[SW21]{ShWe}
Linhui Shen and Daping Weng.
\newblock Cluster structures on double {B}ott-{S}amelson cells.
\newblock {\em Forum Math. Sigma}, 9:Paper No. e66, 89, 2021.

\bibitem[WY07]{WY}
Ben Webster and Milen Yakimov.
\newblock A {D}eodhar-type stratification on the double flag variety.
\newblock {\em Transform. Groups}, 12(4):769--785, 2007.

\end{thebibliography}

\end{document}